\documentclass[a4paper, 11pt]{article}
\usepackage[a4paper, total={6.5in, 8.5in}]{geometry}
\usepackage[english]{babel}
\usepackage[utf8]{inputenc}
\usepackage{amsmath}
\usepackage{setspace}
\usepackage{slashed} 
\usepackage{enumerate}
\usepackage{changepage}
\usepackage{amsfonts}
\usepackage{amssymb}
\usepackage{tikz-cd}
\usepackage{mathtools}
\usepackage{mathrsfs}
\usepackage{amsbsy}
\usepackage{xcolor}
\usepackage{stmaryrd}
\usepackage{varwidth}
\usepackage{array}
\usepackage{tabularx}
\usepackage{tikz}
\usepackage{tikz-cd}
\usetikzlibrary{matrix}
\tikzcdset{cramped}
\usepackage{gensymb}
\usepackage{bm}
\usepackage{graphicx}
\usepackage{tensor}
\usepackage{hyperref}
\hypersetup{%
	pdfborder = {0 0 0}
}
\usepackage{mathtools}
\usepackage{braket}
\usepackage{bbm} 
\usepackage[symbol]{footmisc}

\usepackage[backend=biber,
style=numeric,
bibstyle=numeric,
citestyle=numeric,
sorting=nyt,
hyperref=true,
backref=true,
giveninits=true,
doi=false,
isbn=false,
url=true,
]{biblatex}
\renewbibmacro{in:}{}
\usepackage{amsthm}
\theoremstyle{plain}
\newtheorem{theorem}{Theorem}[section]
\newtheorem{lemma}[theorem]{Lemma}
\newtheorem{proposition}[theorem]{Proposition}
\newtheorem{corollary}[theorem]{Corollary}

\theoremstyle{definition}
\newtheorem{definition}[theorem]{Definition}
\newtheorem{example}[theorem]{Example}
\newtheorem{remark}[theorem]{Remark}

\newtheorem{prop/def}[theorem]{Proposition/Definition}

\numberwithin{figure}{section}
\numberwithin{table}{section}
\numberwithin{equation}{section}

\makeatletter
\newcommand*{\transpose}{%
	{\mathpalette\@transpose{}}%
}
\newcommand*{\@transpose}[2]{%
	\raisebox{\depth}{$\m@th#1\intercal$}%
}
\makeatother
\newcommand*\an{\textnormal{an }}
\newcommand*\id{\textnormal{id}}

\newcommand*\vir{\textnormal{vir}}

\newcommand*\sms{\textnormal{ss }}

\newcommand*\pa{\textnormal{pa}}
\newcommand*\td{\textnormal{td}}

\newcommand*\Ob{\textnormal{Ob}}
\newcommand*\ch{\textnormal{ch}}
\newcommand*\tw{\textnormal{tw}}

\newcommand*{\Ext}{\text{Ext}}
\newcommand*{\Coh}{\text{Coh}}
\newcommand*\BB{\mathbb{B}}
\newcommand*\CC{\mathbb{C}}

\newcommand*\EE{\mathbb{E}}
\newcommand*\FF{\mathbb{F}}
\newcommand*\GG{\mathbb{G}}
\newcommand*\HH{\mathbb{H}}

\newcommand*\LL{\mathbb{L}}

\newcommand*\PP{\mathbb{P}}
\newcommand*\QQ{\mathbb{Q}}

\newcommand*\ZZ{\mathbb{Z}}

\newcommand\Hom{\mathcal{H}\textnormal{om}}
\newcommand*\msM{\mathscr{M}}
\newcommand*\msQ{\mathscr{Q}}

\newcommand*\msY{\mathsf{Y}}
\newcommand{\mA}{{\mathcal A}}
\newcommand{\mB}{{\mathcal B}}
\newcommand{\mC}{{\mathcal C}}
\newcommand{\mD}{{\mathcal D}}
\newcommand{\mE}{{\mathcal E}}
\newcommand{\mF}{{\mathcal F}}

\newcommand{\mI}{{\mathcal I}}

\newcommand{\mM}{{\mathcal M}}
\newcommand{\mN}{{\mathcal N}}
\newcommand{\mO}{{\mathcal O}}
\newcommand{\mP}{{\mathcal P}}

\newcommand{\mT}{{\mathcal T}}
\newcommand{\mU}{{\mathcal U}}
\newcommand{\mV}{{\mathcal V}}

\newcommand{\mX}{{\mathcal X}}

\newcommand{\BU}{\text{BU}}

\newcommand{\pl}{{\textnormal{rig}}}
\newcommand{\ob}{{\textnormal{ob}}}

\newcommand*\Map{\text{Map}}
\newcommand*\inv{\textnormal{inv}}
\newcommand*\rig{\textrm{pl}}
\newcommand*\pe{\textrm{pe}}
\newcommand*\inva{\textnormal{inv}}

\newcommand*\Quot{\textnormal{Quot}_S(E_S,n)}
\newcommand*\QuotC{\textnormal{Quot}_C(E_C,n)}
\newcommand*\QuotCC{\textnormal{Quot}_S(\mO^{e}_S,n)}
\newcommand*\Hilb{\textnormal{Hilb}^n(Y)}
\newcommand*\QuotCCC{\textnormal{Quot}_C(\mO^{e}_C,n)}
\newcommand*\QuotX{\textnormal{Quot}_X(E_X,n)}

\newcommand*\QuotY{\textnormal{Quot}_Y(E_Y,n)}
\newcommand*\QuotS{\textnormal{Quot}_S(E_S,n)}
\newcommand*\QuotSC{\textnormal{Quot}_{S}(\mathcal{O}^{e},n)}

\newcommand*\T{T^{\textnormal{vir}}}

\newcommand*\rnk{\textnormal{rk}}
\newcommand*\JS{\textnormal{JS}}
\newcommand*\uHom{\underline{\textnormal{Hom}}}
\newcommand*\cob{\textnormal{cob}}

\newcommand{\GL}{\textnormal{GL}}

\DeclarePairedDelimiter\floor{\lfloor}{\rfloor}

\newcommand\numberthis{\addtocounter{equation}{1}\tag{\theequation}}

\newcommand\Item[1][]{%
	\ifx\relax#1\relax  \item \else \item[#1] \fi
	\abovedisplayskip=0pt\abovedisplayshortskip=0pt~\vspace*{-\baselineskip}}
\title{Wall-crossing for punctual quot-schemes}
\author{Arkadij Bojko\thanks{Address: Professur für Mathematik ETH Zürich, HG J 16.2 Rämistrasse 101, Zürich 8092, Switzerland , E-mail: arkadij.bojko@math.ethz.ch
}}
\date{}
\begin{document}
	
	\maketitle
	\begin{abstract}
		We study punctual quot-schemes of torsion-free sheaves $E_Y$ on smooth projective curves, surfaces and Calabi--Yau fourfolds via their virtual geometry. Our goal is to give a complete description of the virtual fundamental classes and their tautological integrals. In the fourfold case, we first construct these classes under additional conditions.
  
		We use novel methods relying on the wall-crossing of Joyce. Our results include
  \begin{enumerate}
      \item the dependence of the cobordism classes on the torsion-free sheaf $E_Y$ where $Y$ is a surface,
      \item relations to the previous results in the literature, which addressed the case of a trivial $E_Y$,
      \item a new 12-fold correspondence relating Segre and Verlinde invariants for curves, surfaces and Calabi--Yau fourfolds based on the one observed by Arbesfeld--Johnson--Lim--Oprea--Pandharipande in dimensions one and two,
      \item a closed formula for the Nekrasov genus, which gives a compact analogue of Nekrasov's conjecture.
  \end{enumerate}
		As our techniques are orthogonal to the original literature, we make our work independent by proving a new combinatorial identity in \cite{bojko4}. 
	\end{abstract}
 \begin{small}
	\begin{center}
		\begin{tabular}{p{4cm} p{10cm}}
			$\textbf{Hst}$, $\textbf{Top}$  & The $\infty$-categories of higher stacks and topological spaces. \\
			$\textbf{Ho}(-)$  &The homotopy category of the respective $\infty-$category.\\
    $(-)^{\text{top}}$& The topological realization functor.\\
			$\mX^{\pl}$& For a stack $\mX$ with a $[*/\GG_m]$ action $\rho: [*/\GG_m]\times \mX\to \mX$ denotes its \textit{rigidification}\,.\\
			$\Pi^{\pl}$& The natural projection $\Pi^{\pl}: \mM\to \mM^{\pl}$\,.\\
			$\Hom(-,-)$& The derived hom-functor, for which we never explicitly specify that it is derived.
			\\
			$T^\vir$& The virtual normal bundle of a moduli scheme with a Behrend--Fantechi or Oh--Thomas obstruction theory.\\
			$[z^n]\{f(z)\}$& For a power-series $f(z)$ in a variable $z$ denotes the $n$'th coefficient of $z$.
			\\
			$\mC_Y$ & The topological space $\Map(Y^{\an},\BU\times \ZZ )$ where $(-)^{\an}$ is the analytification.\\
			$\mP_Y$ & The space $\mC_Y\times BU\times \ZZ$.
			\\
			$\mathfrak{U}$ & The universal K-theory class on $\BU\times \ZZ$.\\
			$\mathfrak{E}$& The universal K-theory class on $Y\times \mC_Y$\\
			$\mP_Y$& The topological space $(\BU\times \ZZ)\times \mC_Y$.\\
			$E_Y$& A torsion-free sheaf on $Y$ of rank $e$.\\
			$\Hilb$ & Hilbert scheme of $n$ points of $Y$.\\
			 $Q_Y$ & Moduli scheme $\text{Quot}_Y(E_Y,\beta,n)$ of quotients $E_Y\to F$ for a one-dimensional sheaf $F$ with support $\beta$ and Euler characteristic $n$  when $E_Y$, $\beta$ and $n$ are known.\\
			$K^0(Y)$& Topological K-theory of $Y$.
			\\
			$\alpha^{[n]}$ & The tautological class for a K-theory class $\alpha\in K^0(X)$ of rank $a$.\\
			$\llbracket F^\bullet \rrbracket$ & The K-theory class associated to a perfect complex $F^\bullet$. \\
			$\mN^{E_Y}_0$& The stack of pairs $W\otimes E_Y\to F$.\\
			$\Omega^{E_Y}$& The natural map from $(\mN_0^{E_Y})^{\text{top}}$ to $\mP_Y$. \\
				$[\QuotY]_{\vir}$, $[\mM_{np}]_{\inv}$& The virtual fundamental class of a quot-scheme and classes counting zero-dimensional sheaves in $H_*(\mN^{E_Y}_0)/T\big(H_{*-2}(\mN^{E_Y}_0)\big)$.\\
			$\mathscr{Q}_{E_Y,n}$, $\mathscr{M}_{np}$& The classes obtained from $[\QuotY]_{\vir}, [\mM_{np}]_{\inv}$ via the map $\overline{\Omega}^{E_Y}_*: H_*(\mN^{E_Y}_0)/T\big(H_{*-2}(\mN^{E_Y}_0)\big)\to H_*(\mP_Y)/T\big(H_{*-2}(\mP_Y)\big)$ induced by pushforward in homology.
			\\
			$c_i$& The degree $i$ Chern class of $Y$ when $Y$ is understood.\\
			$\mathcal{A}_{E_Y}$&The abelian category of morphisms $\phi:W\otimes E_Y\to F$, where $W$ is a vector space and $F$ is zero-dimensional. \\
			$V_*, \ket{0},T,\mathsf{Y}_{(-)}$& The translation operator $T:V_*\to V_{*+2}$, a vacuum vector and the state-field correspondence of some vertex algebra $V_*$.\\
   $\otimes$&The tensor product over $\ZZ$ unless a different ring is specified.
   \\
			$V_*$& After §3.2, we reserve this notation for the vertex algebra on $H_*(\mN^{E_Y}_0)$.\\
			$P_*$ & The vertex algebra on $H_*(\mP_Y)$.\\
			$B_i,B^\vee_i$& A basis $B_i\subset H^i(Y)$ and its dual 
with respect to the cohomological pairing.\\
$\Lambda_{-y}$& The multiplicative genus associated to the invertible powerseries $1-ye^t$ in $t$.
\end{tabular}
	\end{center}
	\end{small}
\setstretch{1.1}
	
	\section{Introduction}
	\subsection{Background}
	The goal of this paper is to give a unifying treatment of virtual invariants for Grothendieck's quot-schemes of zero-dimensional quotients. More explicitly, we study here the case when the projective variety $Y$ is a smooth curve $C$, a surface $S$ or a Calabi--Yau fourfold $X$.  For such $Y$, we have the projective variety $ \textnormal{Quot}_Y(E_Y,\beta,n)$ of 1-dimensional quotients 
	$$E_Y\twoheadrightarrow F \quad \textnormal{for} \quad \chi(F)=n\,,\quad \textnormal{and} \quad [F]=\beta\in H_2(Y,\ZZ)\,,$$
	where two such quotients are identified if their kernels $I\hookrightarrow E_Y$ are equal as subsheaves. From now on, we will always work with (co)homology theories over rational numbers without mentioning it.

 We will often use the notation $Q_Y:=\textnormal{Quot}_Y(E_Y,\beta,n)$, when $E_Y,\beta$ and $n$ are understood. We will always denote by $e$ the rank of $E_Y$. 
	Defining the projections 
	$$
	Y\xleftarrow{\pi_Y}Y\times Q_Y\xrightarrow{\pi_{Q_Y}}Q_Y
	$$
	we have the universal complex 
	$$
	\mI = \big[\pi^*_Y(E_Y) \longrightarrow \mF\big]
	$$
	for the universal zero-dimensional sheaf $\mF$.
	
	In general, the above moduli spaces are not smooth and singularities appear already in the case when $Y=S$ and $\beta=0$ as long as $e:=\textnormal{rk}(E_Y)\geq 2$ and $n\geq 2$ (see Stark \cite[Prop. 1]{stark2}). In \cite[Conj. 1]{stark2}, these singularities have been conjectured to be rational. However, we choose a different path in our study of their geometries focusing on the virtual invariants. We then need to distinguish between different dimensions, because their virtual geometries have different origins:
	\begin{itemize}
		\item When $Y=C$ and $\beta=0$, $Q_C$ is known to be smooth, hence we can study its fundamental class
		\begin{equation}
		[Q_C]\in H_{2ne}(Q_C)\,.
		\end{equation}
		directly.
		\item When $Y=S$, we rely on the work of Behrend--Fantechi \cite{BF} and Li--Tian \cite{LiTian} to obtain virtual fundamental classes using the natural perfect obstruction theories on $Q_S$. These were constructed by Marian--Oprea--Pandharipande \cite[Lem. 1.1]{MOP1} (see Stark \cite[Prop. 5]{stark2} for a more detailed proof) in the case that $E_S$ is a vector bundle. The most important part of the argument is the vanishing of $\Ext^2(I,F)$ at each $\CC$-point of $\QuotS$ corresponding to the exact sequence $0\to I\to E_S\to F\to 0$. This vanishing remains true even when $E_S$ is a torsion-free sheaf, because 
		$$
		\text{Ext}^2_S(I,F)^* \cong \text{Ext}^0_S(F,I\otimes \omega_S)\hookrightarrow \text{Ext}^0_S(F,E_S\otimes\omega_S) = 0\,.
		$$
		The obstruction theory is given by the dual of its virtual tangent bundle
		\begin{equation}
		\label{eqClF}
		T^{\textnormal{vir}}=\underline{\textnormal{Hom}}_{Q_S}(\mathcal{I},\mathcal{F})\,,\end{equation}
		where 
		$$\uHom_{Q_Y}(-,-) =\pi_{Q_Y,\,*}\big(\Hom(-,-)\big)$$
		and $\Hom$ denotes the global derived hom-functor.
		After setting $\beta =0$, the virtual dimension is given by
		\begin{align*}
		& \textnormal{vd}\big(Q_S\big) = e\cdot n\,.
		\end{align*}

		As a consequence, we obtain virtual fundamental classes 
		\begin{equation}
		\label{eqQvi}
		[Q_S]^{\textnormal{vir}}\in H_{2en}\big(Q_S\big) \,.
		\end{equation}
		which were used with the additional restriction $E_S=\mO^{e}_S$ by Marian--Oprea--Pandharipande \cite{MOP1} to prove Lehn's conjecture \cite{Lehn} for the generating series of tautological invariants on Hilbert schemes of points $\text{Hilb}^n(S)$. More recently, they were studied by Arbesfeld et al \cite{AJLOP}, Johnson et al \cite{JOP}, Lim \cite{Lim} and Oprea--Pandharipande \cite{OP1}.
		\item Following the philosophy of Donaldson \cite{Do83}, who defined invariants counting \textit{anti-self-dual connections} on a real 4-manifold, Donaldson--Thomas proposed a holomorphic version of this construction for a \textit{Calabi--Yau fourfold} $X$ in Donaldson--Thomas \cite{DT96}. Their ideas can be translated to studying half of the obstruction theory of sheaves
  \begin{equation}
		\label{eq:CYob}
		T^{\textnormal{vir}}=\underline{\textnormal{Hom}}_{Q_X}(\mathcal{I},\mathcal{I})\,,\end{equation}
  after using Serre duality $$
\Ext^i(E,E) \cong \Ext^{4-i}(E,E)^*\,,
$$
a rigorous formulation of which was given by Borisov--Joyce \cite{BJ} and Oh--Thomas \cite{OT}.
  If $Y=X$ is a\textit{ strict Calabi--Yau fourfold} (by this we mean that there is a choice of a non-vanishing section of $K_X$ and $H^2(\mathcal{O}_X)=0$) and $E_X$ on $X$ is a simple rigid locally-free sheaf,  we construct in §\ref{secQCY4c} the virtual fundamental classes:
		\begin{equation}
		\label{eqQuoX}
		[\textnormal{Quot}_X(E_X,\beta,n)]^{\vir} \in H_{2en}\big(\textnormal{Quot}_X(E_X,\beta,n)\big)\,.
		\end{equation}
		
		The class \eqref{eqQuoX} is a generalization of the one studied in the author's previous work \cite{bojko2}, where we worked with $E_X = \mO_X$ and $\beta=0$. This 
 simplified it to the virtual fundamental class over the Hilbert scheme of points $\big[\textnormal{Hilb}^n(X)\big]^{\textnormal{vir}}\in H_{2n}\big(\textnormal{Hilb}^n(X)\big)$. Unlike the VFC induced by perfect obstruction theories, we need the additional choice of data called \textit{orientations} which exist by Cao--Gross--Joyce \cite{CGJ} and the author's \cite{bojko}. These lead to a freedom of choice of a sign in \eqref{eqQuoX}.
	\end{itemize}
	\subsection{Our methods}
	\label{Sec: our methods}
	When $Y=S,X$ we use novel methods developed by the author in \cite{bojko2} relying on the wall-crossing framework of Joyce \cite{GJT, JoyceWC}. Our approach to tackling computations with the above virtual fundamental classes is therefore different from the one appearing for example in \cite{AJLOP, JOP, Lim}. As the machinery used to obtain the results is new, we give a brief summary of the approach. 
	
	Starting from the stack $\mN^{E_Y}_0$ of morphisms $W\otimes \mO_Y\to F$, where $F$ is a zero-dimensional sheaf and $U$ a vector space, we construct a vertex algebra (see \cite{Borcherds, LLVA, KacVA} for an introduction to the subject) on its homology $V_* = H_*(\mN^{E_Y}_0)$ by a recipe of Joyce \cite{Joycehall}. A part of the data inherent to a vertex algebra is a translation operator $T: V_*\to V_{*+2} $ and the quotient by its action carries a Lie algebra structure. As we will see, the homology classes 
	$
	[\QuotY]^{\vir}
	$
	can be first pushed forward to $V_*$ and then related inside of the Lie algebra $ V_*/T V_{*-2}$ to the classes 
	\begin{equation}
	\label{Eq: Mnpclasses}
	[M_{np}]_{\inv}\in V_2/TV_{0}
	\end{equation} 
	counting zero-dimensional sheaves supported at $n$ points counted with multiplicities.
	We note that 
	\begin{itemize}
		\item when $Y=S$, we are relying on Theorem \ref{thmCS} which we prove in the Appendix A.
		\item when $Y=X$, we use Theorem \ref{conjecture quot WC} which needs additional work, because of dealing with Borisov--Joyce/Oh--Thomas classes rather than Behrend--Fantechi ones. The proof of this statement will appear in the upcoming work of the author \cite{bojko3} by proving a general wall-crossing statement for Calabi--Yau fourfolds. The arguments which will make it applicable to quot-schemes are the same as in the case of $Y=C,S$ described in Appendix \ref{App: checking of assumptions}.
	\end{itemize}
	As generally is the case with wall-crossing, one needs to know invariants on one side of the wall which in our case are the classes \eqref{Eq: Mnpclasses}.
	We recover them from twisting Joyce's vertex algebra construction on $V_*$ by a global version of the tautological classes $L^{[n]}$ (see the precise definition in \eqref{Eq: alphan}) for any line bundle $L$
	and by wall-crossing instead with the invariants
	$$
	\int_{[\Hilb]^{\vir}}c_{n}(L^{[n]})\,, 
	$$
	where $\Hilb=\text{Quot}_Y(\CC,n)$. 
	To get explicit formulae we need to move to the homology of the topological space
	$$
	\mP_Y = (BU\times \ZZ)\times \mC_Y = (BU\times \ZZ)\times  \text{Map}_{C^0}(Y,BU\times \mathbb{Z})
	$$
replacing the stack $\mN^{E_Y}_0$.
	Taking the K-theory class $
	\llbracket \pi^*_Y(E_Y)\rrbracket - \llbracket \mathcal{F}\rrbracket = \llbracket \mI\rrbracket 
	$ of the universal complex on $Y\times Q_Y$
	induces a natural map from its analytification
	$$
	\Omega_{Q_Y}: \big(\QuotY\big)^{\text{an}}\to \mathcal{C}_Y \,.
	$$
	We may then study the pushforward of the virtual fundamental classes
	$$
	\mathscr{Q}_{E_Y,n}=\Omega_{Q_Y\,*}\Big([Q_Y]^{\text{vir}}\Big) \in H_*(\mathcal{C}_Y)
	$$
	for which we obtain a closed expression in Theorem \ref{thmMain} using Gross's \cite{gross} description of the homology of $\mathcal{C}_Y$.

	The full generating series \begin{equation}
 \label{Eq:QEY}
	\mathscr{Q}_{E_Y}(q) = \sum_{n\geq 0}\mathscr{Q}_{E_Y,n}q^n
	\end{equation}
	will have a closed formula stated in Theorem \ref{thmMain}. These invariants contain the full topological data of the virtual fundamental classes. However, they do not exhibit any interesting obvious symmetries, unlike the invariants that can be extracted from them after integration. We, therefore, view the explicit expressions for \eqref{Eq:QEY} as an auxiliary result similar to the formula for wall-crossing with descendants found in the monograph of T. Mochizuki \cite{mochizuki} which on its own holds little value in proving or making conjectures about numerical invariants without extra work done. This was undertaken for example in Göttsche--Kool \cite{GK1, GK2} in the case of Mochizuki's work. 
	\subsection{Surfaces}
	Note that until recently all known results for $Y=S$ in the above sources are for $E_Y=\mO^e_Y$. The benefit of our approach is that it works essentially whenever perfect obstruction theories exist, allowing us to handle any such $E_Y$.  Every result stated in this subsection is a theorem due to the work of Joyce \cite{JoyceWC} together with Theorem \ref{Thm:PJS} proved in the Appendix \ref{App: checking of assumptions}, where we check that the necessary assumptions are satisfied. 
	
	Our first step towards generalizing beyond trivial vector bundles is a statement about the  independence of the \textit{virtual cobordism class} of $\QuotS$ on $c_1(E_S)$ and $c_2(E_S)$. For this, recall from Shen \cite{Shen} that for a projective scheme $Z$ with a perfect obstruction theory its virtual cobordism class is denoted by
	$$
	\big[Z\big]^{\textnormal{vir}}_{\textnormal{cob}}\in \Omega(\textnormal{pt})\,,
	$$
	which by \cite[Thm. 0.1]{Shen} is uniquely determined by the integrals of the form
	$$
	\int_{\big[Z\big]^{\vir}}P(\textnormal{ch}_1(T^{\textnormal{vir}}),\textnormal{ch}_2(T^{\textnormal{vir}}),\ldots)\,,
	$$
	where $P$ is any element of  $R[x_1,x_2,\cdots]$ over a ring  $R$.
	The virtual Euler characteristics $\chi^\vir(Z)$ as defined by Fantechi--Göttsche \cite{FanGo} can be extracted out of $[Z]^{\vir}_{\cob}$ by integrating the Euler class of the virtual tangent bundle
	$$
	\int_{[Z]^{\vir}}c(T^{\vir})\,.
	$$
	We then have the following statement:
	\begin{theorem}
		\label{thmInd}
		Let $P\in R[x_1,x_2,\ldots]$ be a polynomial over a ring $R$, then 
		$$
		\int_{\big[\Quot\big]^{\vir}}P(\textnormal{ch}_1(T^{\textnormal{vir}}),\textnormal{ch}_2(T^{\textnormal{vir}}),\ldots) = \int_{\big[\textnormal{Quot}_S(\mO^{e}_S,n)\big]^{\textnormal{vir}}}P(\textnormal{ch}_1(T^{\textnormal{vir}}),\textnormal{ch}_2(T^{\textnormal{vir}}),\ldots)
		$$
		for all $E_S$ torsion-free and $n\geq 0$. Equivalently, we obtain
		$$
		\big[\Quot \big]^{\vir}_{\cob} = \big[\QuotCC \big]^{\vir}_{\cob} \in \Omega(\textnormal{pt})\,.
		$$
		As a special case of this, we have
		$$
		\chi^{\vir}\big(\Quot\big) = \chi^{\vir}\big(\QuotCC\big)\,.
		$$
	\end{theorem}
	
	More generally, we will consider multiplicative genera of the tautological classes $\alpha^{[n]}$ which are constructed out of a topological K-theory class $\alpha\in K^0(X)$ of rank $a=\textnormal{rk}(\alpha)$ as
	\begin{equation}
	\label{Eq: alphan}
	\alpha^{[n]}=\pi_{Q_Y\,*}(\pi^*_Y(\alpha)\otimes \mF)\,.
	\end{equation}
	Define the generating series of invariants 
	$$
	Z_{E_Y}(f,g,\alpha;q) = \sum_{n\geq 0}q^n\int_{[\QuotY]^{\vir}} f(\alpha^{[n]})g(T^{\vir})\,,
	$$
	where $f,g$ are multiplicative genera as in the previous work \cite[§5.3]{bojko2}.
	Theorem \ref{thmInd} is a consequence of the following result:
	\begin{theorem}
		\label{thmZEZC}
		We have the identity:
		\begin{equation}
		\label{eqZEZC}
		\frac{Z_{E_S}(f,g,\alpha;q)}{Z_{\mO^{e}_S}(f,g,\alpha;q)} = \prod_{i=1}^e\bigg(\frac{f\big(H_i(q)\big)}{f(0)}\bigg)^{a\mu_S(E_S)}\,,
		\end{equation}
		where $\mu_S(E_S) = c_1(E_S)\cdot c_1(S)/e$ and $H_i(q)$ are the $e$ different Newton--Puiseux solutions of 
		$$
		H^e_i(q) =qf^a\big(H_i(q)\big)g^e\big(H_i(q)\big) \,.
		$$
	\end{theorem}

	To obtain the above theorem, we prove a new\footnote{The experts in this area were not aware of this result.}  combinatorial expression for formal power-series in the companion paper \cite{bojko4} in 2 different ways which rely purely on combinatorics. We will be using the notation 
	$$
	[z^n]\big\{f(z)\big\} = f_n\,,\qquad \text{where} \quad f(z) = \sum_{n\geq 0}f_n z^n\,.
	$$
	\begin{theorem}[\cite{bojko4}]
		\label{thmCtC}
		Let $Q(t)\in \mathbb{K}\llbracket t\rrbracket$  be a power-series over a field $\mathbb{K}$ (e.g. the field of Laurent-series over $\QQ$) independent of $q$ and satisfying $Q(0) \neq 0$ .
		Let  $H_i(q)$ be the $e$ different Newton--Puiseux solutions to $
		H_i^e(q) = qQ(H_i(q))\,.
		$
		Set
		\begin{equation}
		\label{eqGeR}
		G_e(Q)= \exp\bigg[-\sum_{\begin{subarray}a n,m>0\\ j>0\end{subarray}}j\frac{1}{m}[z^{me+j}]\Big\{Q^m(z)\Big\}\frac{1}{n}[z^{ne-j}]\Big\{Q^n(z)\Big\}q^{n+m}\bigg]\,.  
		\end{equation}
		Then the following identity holds:
		\begin{align*}
		\label{eqExp}
		&G_e(Q)=  \prod_{i=1}^e\bigg(\frac{Q(H_i)}{Q(0)}\bigg)\prod_{i=1}^e(-H_i)^{e}
		\cdot \prod_{i_1\neq i_2 }\bigg(\frac{1}{H_{i_1}-H_{i_2}}\bigg)\prod_{i=1}^e\bigg(\frac{Q'(H_i)}{Q(H_i)}-\frac{e}{H_i}\bigg)\,.
		\numberthis
		\end{align*}
	\end{theorem}
	One proof that does not use combinatorial techniques but instead the wall-crossing and previous results using virtual localization techniques appears already in this paper but is unsatisfactory because it would make our powerful approach dependent on previous results.
	
	Using Theorem \ref{thmZEZC}:
	\begin{itemize}
		\item We study the rationality and poles of the K-theoretic descendant series associated with the virtual $\chi_y$-genus (see \eqref{eqPoles2} for the precise definition). Results related to descendants in homology and K-theory in the case $E_S=\mO^{e}_S$ appeared in \cite{AJLOP, JOP, Woonamthesis, Arb22} and we summarize ours in Theorem \ref{Thm:descendents}.
		\item We find a new symmetry of Segre and Verlinde invariants that removes the additional restriction from the analogous one observed in Arbesfeld et. al. \cite[§1.7]{AJLOP} and includes a correspondence with Calabi--Yau fourfolds. This also includes a generalization of the Segre--Verlinde duality observed for $E_S=\mO^{e}_S$ in \cite[Thm. 13]{AJLOP} for curves and surfaces and in the author's work \cite[Thm 5.13]{bojko2} for fourfolds. Our new results are summarized in Theorem \ref{thmSVD} and \ref{thm4FoI}.
	\item In Example \ref{Nekex}, we consider the Nekrasov genus four fourfolds when $E_X$ is a rigid simple vector-bundle and show that it can be expressed in certain cases in terms of the MacMahon function
    $$
    M(q) = \prod_{n>0}(1-q^n)^{-n}
    $$
which is the generating series for counting plane partitions. 
	\end{itemize}
	The virtual Euler characteristic also allows for additional insertions to be integrated in K-theory. For a class $A\in K^0(\QuotS) $, we denote the corresponding virtual integral in K-theory by 
	$$
	\chi^{\vir}(\QuotS, A) = \chi(\QuotS,\mO^\vir\otimes A)\,,
	$$
	where $\mO^\vir$ is the virtual structure sheaf defined for example in Kapranov--Ciocan--Fontanine \cite{KCF} or
 Fantechi--Göttsche \cite{FanGo}.
	\begin{theorem}
 \label{Thm:descendents}
		Consider the generating series for fixed choices of $\alpha_i\in K^0(X)$, $k_i\geq 0$ and $i=1,\ldots,l$:
		\begin{itemize}
			\item of descendant invariants 
			$$Z^{\textnormal{des}}_{E_S}(\alpha_1,\ldots,\alpha_l|k_1,\ldots,k_l)(q)=\sum_{n\geq 0}q^n\int_{\big[\Quot\big]^{\vir}}\textnormal{ch}_{k_1}\big(\alpha_1^{[n]}\big)\cdots\textnormal{ch}_{k_l}\big(\alpha^{[n]}_l\big)c\big(T^{\vir}\big)\,,$$
			\item of K-theoretic invariants $$Z^{\chi_{-y}}_{E_S}(\alpha_1,\ldots,\alpha_l|k_1,\ldots,k_l)(q) = \sum_{n\geq 0}q^n\chi^{\textnormal{vir}}\Big(\Quot,\wedge^{k_1}\alpha_1^{[n]}\otimes\ldots \otimes \wedge^{k_l}\alpha_l^{[n]}\otimes \Lambda_{-y}\Omega^{\vir}\Big)\,,$$
			where $\Omega^\vir = \big(T^{\vir}\big)^\vee$ and $\Lambda_{-y}$ is the multiplicative genus given by the invertible power series $1-ye^t$ in $t$.
		\end{itemize}
		Then the following statements hold:
		\begin{itemize}
			\item $Z^{\chi_{-y}}_{E_S}\big(\alpha_1,\ldots,\alpha_l|k_1,\ldots,k_l\big)(q)$ is a rational function which has poles of order less than or equal to $\sum_{i=1}^l2k_i$ at $q=1$.
			\item $Z^{\textnormal{des}}_{E_S}(\alpha_1,\ldots,\alpha_l|k_1,\ldots,k_l)(q)$ is a rational function.
		\end{itemize}
	\end{theorem}
	
	\subsection{Calabi--Yau fourfolds and the 12-fold correspondence}
	In this subsection, we focus on the virtual invariants of $\QuotX$ using its virtual fundamental class that we construct. All of the results rely on Theorem \ref{conjecture quot WC}, the proof of which is going to follow from our upcoming work \cite{bojko3} and the same arguments as in the appendix \ref{App: checking of assumptions}.	

	Using the Oh-Thomas \cite{OT} \textit{twisted virtual structure sheaf} $\hat{\mO}^{\vir}$, we set for each class $A\in K^0(Q_X)$ the notation
	$$
	\hat{\chi}^{\vir}(A) = \chi\big(Q_X,\hat{\mO}^{\vir}\otimes A\big)\,.
	$$
	We introduce the untwisted K-theoretic invariants by setting 
	$$
	\chi^\vir(A) = \hat{\chi}^\vir\big(A\otimes \mathsf{E}^{\frac{1}{2}}\big)\,,
	$$
 where $\mathsf{E}=\textnormal{det}\Big(\big(E_X^\vee\big)^{[n]}\Big)$.
	These invariants have a nicer behaviour due to the inherent twist by a square root line bundle in the construction of $\hat{\mO}^{\vir}$ in \cite[Def. 5.3]{OT}. To obtain integers, one needs to make up for it with another square root line bundle.

 We moreover observe that they are closely related to invariants on surfaces, by having a universal transformation relating them to the corresponding analogs. This transformation recovers the one from \cite[Def. 4.10]{bojko2} after setting $e=1$:
	\begin{equation}
	\label{eqUnT}
	U_e(f) = \prod_{n}\prod_{k=1}^nf\big((-1)^ee^{\frac{2\pi i k}{n}}q\big)\,.
	\end{equation}
	Let us first state a simple structural result between the generating series. 
	\begin{corollary}[Corollary \ref{Cor:comp}]
		\label{thm42r}
		Let $X$ be a Calabi--Yau fourfold, S a surface such that $c_1(S)^2=0$ and $\alpha_Y\in K^0(Y)$ for $Y=X,S$ satisfying
		$$
		c_1(\alpha_S)\cdot c_1(S) = c_1(\alpha_X)\cdot c_3(X)\,,
		$$
		then 
		$$
		Z_{E_X}(f,g,\alpha_X;q) = U_e\big(Z_{E_S}(f,\tilde{g},\alpha_S;q)\big)\,,
		$$
  where $\tilde{g}(x) = g(x)g(-x)$. Moreover, if 
		$$
		K_{E_Y}(f,\alpha_Y;q) = \sum_{n\geq 0}q^n\chi^{\vir}\big(f(\alpha^{[n]}_Y)\big)\,,
		$$
		where we are identifying $K^0(Q_Y)\cong H^*(Q_Y)$ using the Chern character to make sense of integrating the cohomology class $f(\alpha^{[n]}_Y)$ in K-theory, then
		$$
		K_{E_X}(f,\alpha_X;q) = U_e\big(K_{E_S}(f,\alpha_S;q)\big)\,.
		$$
	\end{corollary}
	In \cite{bojko2}, the author used this relation when $E=\mO_X$ to explain \textit{Segre--Verlinde duality }for Calabi--Yau fourfolds. We now complete the picture that we began studying there by noting down the full list of relations between the \textit{Segre} and\textit{ Verlinde series} \underline{in dimensions 1, 2 and 4}. Recall therefore that we define for $Y=S,X$ the Segre, resp. Verlinde series by:
	\begin{align*}
	S_{E_Y,\alpha_Y}\big(q\big)&= \sum_{n\geq 0}q^n\int_{[Q_Y]^{\vir}}s_n\big(\alpha_Y^{[n]}\big)\,,\\
	V_{E_Y,\alpha_Y}(q) &= \sum_{n\geq 0}q^n\chi^{\vir}\big(Q_Y,\textnormal{det}\big(\alpha_Y^{[n]}\big)\big)\,.
	\end{align*}
	To find further relation to curves, we follow Arbesfeld et al. \cite{AJLOP} in defining for any $C$ and vector bundle $E_C$ on $C$ the class
	$$
	\mathcal{D}_n=\text{det}^{\frac{1}{2}}\Big(\underline{\textnormal{Hom}}_{Q_C}\big(\mathcal{F},\mathcal{F}\otimes\pi_C^*( \Theta)\big)\Big) 
	$$
	as a K-theory class, where $\Theta$ is any choice of theta-characteristic satisfying $\Theta^2 = K_C$. Then following \cite{AJLOP}, the \textit{twisted Verlinde numbers} for curves are defined by 
	$$
	V_{E_C,\alpha}(q) = \sum_{n\geq 0}q^n\chi\big(\det(\alpha^{[n]})\otimes \mathcal{D}_n\big)\,.
	$$

	The next sequence of results is best summarized by the diagram in Figure \ref{fig1} which gathers the statements of Corollary \ref{thm42r} and Theorems \ref{thmSVD} and \ref{thm4FoI}.
	
	\begin{theorem}
		\label{thmSVD}
		For any torsion-free sheaf $E_Y$ on $Y$ and $\alpha_Y\in K^0(Y)$, we have the following Segre--Verlinde duality
		$$
		S_{E_Y,\alpha_Y}\big((-1)^eq\big)=V_{E_Y,\alpha_Y}(q)\,,
		$$
		whenever the invariants are defined.
	\end{theorem}
As already mentioned, the fourfold theory is the holomorphic version of Donaldson's theory which can be also rephrased in terms of virtual fundamental classes when working with complex surfaces (see Morgan \cite{Mo93}). One should therefore expect that previously obtained results for surfaces could be formulated in this case. 
 
	Finally, we obtain the following observation comparing two different Segre/Verlinde series. The authors of \cite{AJLOP} observed a similar symmetry which however required additional restriction on insertions because they focused on studying quotients of trivial vector bundles. We believe this is the correct direction towards answering \cite[Question 16]{AJLOP}, which asks about the geometric interpretation of their version of the symmetry. 
	\begin{theorem}
		\label{thm4FoI}
		Let $E_Y,F_Y\to Y$ be two torsion-free sheaves with ranks $e=\textnormal{rk}(E_Y),f=\textnormal{rk}(F_Y)$, then the following holds:
		\begin{align*}
		S_{E_Y,F_Y}\big((-1)^eq\big) &= S_{F_Y,E_Y}\big((-1)^fq\big)\,.\\
		V_{E_Y,F_Y}(q) &= V_{F_Y,E_Y}(q)\,.
		\end{align*}
		whenever the generating series are defined. 
	\end{theorem}
	\begin{figure}
		\label{fig1}
		\includegraphics[scale=1]{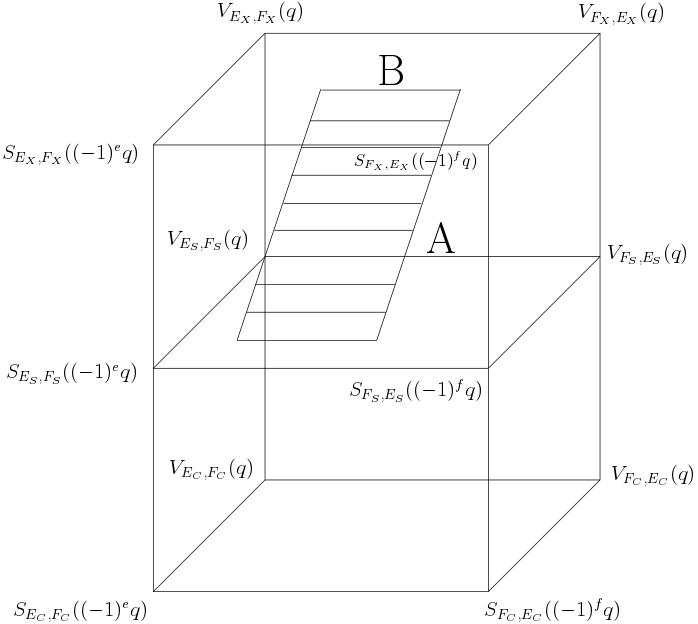}
		\caption{This diagram captures the symmetries of Segre--Verlinde invariants. The letter $B$ is meant to denote that we are specializing to simple rigid locally-free sheaves $E_X,F_X$ for the case $Y=X$. On the other hand, we need to take the ladder A (this is our universal transformation $U_e$ of \eqref{eqUnT} to go from the surfaces to $B$ which puts additional restrictions on geometries of $S$ as described in Theorem \ref{thm42r}.}
		\centering
	\end{figure}
	To conclude, we study an example of the Segre generating series for $Y = X$ and its  K-theoretic refinement studied by the author \cite{bojko2} for the Hilbert schemes of points. These invariants were originally defined by Nekrasov \cite{Nekrasov1} and Nekrasov--Piazzalunga \cite{Nekrasov2}, where they worked equivariantly over $\CC^4$.  
	\begin{example}
		\label{Nekex}
		Let $\alpha_X\in K^0(X)$ be a K-theory class of rank $\textnormal{rk}(\alpha_X) = e$. The invariants 
		$$
		S_{E_X,-\alpha_X}(q) = \sum_{n\geq 0}q^n\int_{[Q_X]^\vir}c_{en}\big(\alpha_X^{[n]}\big)=:I_{E_X,\alpha_X}(q)\,,
		$$
		can be expressed in terms of their K-theoretic refinements. These are defined by
		\begin{equation}
		\label{eqNekI}
		N(\alpha_X,y;q) = \sum_{n\geq 0}q^n\hat{\chi}^{\vir}\Big(\Lambda^{\bullet}_{y^{-1}}(\alpha_X^{[n]})\otimes \textnormal{det}^{-\frac{1}{2}}\big(\alpha_X^{[n]}y^{-1}\big) \Big)\,.
		\end{equation}
		We can recover the former by taking the limit 
		$$
		\textnormal{lim}_{y\to 1^+}(1-y^{-1})^{(a-e)n}[q^n]\Big\{N\big(\alpha_X,y;q\big)\Big\}
		= (-1)^{ne}[q^n]\big\{I_{E_X,\alpha_X}(q)\big\}\
		$$
  as follows from the proof of \cite[Prop. 5.5]{bojko2}.
		Denoting by $M(q) = \prod_{n>0}(1-q^n)^{-n}$ the MacMahon function, we show in Theorem  \ref{thmNek} that
		$$
		N(\alpha_X,y;q)  = M\big(y^{\frac{e}{2}}q\big)M(y^{-\frac{e}{2}}q)^{\frac{1}{2}\big(c_1(\alpha_X)\cdot c_3(X)+\mu_X(E_{X})\big)}\,.
		$$
		By an identical argument as in \cite[(1.10),(1.11)
		]{bojko2} it can be shown to be related to the conjectured formula on $\CC^4$ the proof of which has been announced recently by Kool--Rennemo \cite{KRdraft}.
		Finally, this implies that
		$$
		I_{E_X,\alpha_X}(q) =M\big((-1)^eq \big)^{c_1(\alpha)\cdot c_3(X)+\mu_X(E_X)}\,.
		$$
	\end{example}
	\subsection*{Related works}
	The virtual geometry of $Q_S$ when $E_S = \mO^{e}_S$ is well studied in \cite{OP1, JOP, AJLOP, Lim, Woonamthesis}. However, the extension to a non-trivial vector bundle has only been pursued recently. The work of Stark \cite{stark} also addresses this question. He proves independently  the part of Theorem \ref{thmInd} relating to the virtual Euler characteristic using a universality argument, invariance under the twisting of $E_S$ by a line bundle $L$ and a dimensional reduction to a smooth canonical curve $C$ using cosection localization.

	A shadow of the symmetry described in the first two floors of Figure \ref{fig1} has been observed by Arbesfeld--Johnson--Lim--Oprea--Pandharipande \cite{AJLOP}. However, ours is a genuinely new result, because to recover \cite[Thm. 14]{AJLOP} one would need to divide out by the additional contribution in Theorem \ref{thmZEZC}. As in loc. cit., this leads to forcing additional restrictions which in our case are no longer needed. 
	
	Ricolfi \cite[Conj. 3.5]{ricolfi1} conjectures a formula for the generating series of virtual fundamental classes of some punctual quot-schemes on three-folds.  I was unable to address it  here because there is no virtual fundamental class for 0-dimensional sheaves on a general 3-fold. 
	
	To make our work independent of the previous results in this area found in \cite{AJLOP, JOP, Lim, Woonamthesis} except for the simple example from Oprea--Pandharipande \cite{OP1}, we needed to prove the identity of Theorem \ref{thmCtC} using purely combinatorics. This is accomplished in \cite{bojko4}.
	
	\section*{Acknowledgement}
	I would like to thank Alex Gavrilov and Ira Gessel for their help with the combinatorial identity appearing in this work. Further, I thank Noah Arbesfeld, Dominic Joyce, Martijn Kool, Woonam Lim, Hyeonjun Park, Samuel Stark, Richard Thomas, and Rahul Pandharipande for helpful discussions.

	I was supported by ERC-2017-AdG-786580-MACI. This project has received funding from the European Research Council (ERC) under the European Union Horizon 2020 research and innovation program (grant agreement No 786580). During my stay at the University of Oxford, I was supported by a Clarendon Fund scholarship. 
	
	\section{Minimal geometric input and the construction of $[Q_X]^\vir$}
	We begin by constructing the virtual fundamental classes $[\QuotX]^{\vir}$ for $X$ a strict Calabi--Yau fourfold which we will study only in the case of $\beta=0$ in later sections. Then we explain the minimal external data needed as input into our machinery to obtain all of our results. These are precisely the cap products
	$$
	I_{Y,n}(L)=[\Hilb]^\vir\cap c_n(L^{[n]})
	$$
	for any line bundle $L$, which are known for 
	$$
	\begin{cases}
	Y=S&\textnormal{from Oprea--Pandharipande\cite[(31)]{OP1}}\,,\\
	Y=X& \textnormal{from our previous work \cite[Thm. 3.1]{bojko2} .}
	\end{cases}
	$$
	\subsection{Virtual fundamental classes for Y=X}
	\label{secQCY4c}
	We follow closely the arguments of Ricolfi \cite{ricolfi1} applied to the constructions of Oh--Thomas \cite{OT}, who gave an algebraic construction of virtual fundamental classes for projective moduli schemes of stable sheaves on Calabi--Yau fourfolds. We will continue using the notation $Q_X=\textnormal{Quot}_X(E_X,\beta,n)$. Note that as we increased the dimension of $X$, we allow the support of the quotient $F$ to be in 1 dimension higher than in \cite{ricolfi1}.
	\begin{definition}
		For the rest of the section, we will say that a locally free sheaf $E_X$ is \textit{RS} if it is rigid and simple. Rigid for a perfect complex $E^\bullet$ here means $\Ext^1(E^\bullet,E^\bullet) = 0$ and simple 
  $$
  \Ext^{-i}(E^\bullet, E^\bullet) = 0\,,\quad \text{for}\quad i>0\,,\qquad \text{and}\qquad \Ext^0(E^\bullet,E^\bullet)\cong \CC\,.
  $$
	\end{definition}
	Let 
	$$0\to \mathcal{I}\to \pi_X^*E_X\to \mathcal{F}\to 0$$
	be the universal short exact sequence on $X\times Q_X$, then  Huybrechts--Thomas \cite[Def. 2.6]{HuyTho} define the truncated Atiyah class 
	$$
	\Hom(\mathcal{I},\mathcal{I})[-1]\xrightarrow{\textnormal{At}}\mathbb{L}_{X\times Q_X}\,,
	$$
	which together with $\mathbb{L}_{X\times Q_X}=\pi_X^*\mathbb{L}_X\oplus \pi_{Q_X}^*\mathbb{L}_{Q_X}$ and the natural projection induces 
	$$
	\Hom_{X\times Q_X}(\mathcal{I},\mathcal{I})[-1]\xrightarrow{\textnormal{At}}\pi^*_{Q_X}\mathbb{L}_{Q_X}
	$$
	By Grothendieck--Serre duality, this gives
	$$
\mathbb{E}:=\underline{\textnormal{Hom}}_{Q_X}(\mathcal{I},\mathcal{I})_0[3]\rightarrow \underline{\textnormal{Hom}}_{Q_X}(\mathcal{I},\mathcal{I})[3]\to \mathbb{L}_{Q_X}\,.
	$$
	This is the obstruction theory, which we use to construct the virtual fundamental class using the results below. The majority of the proofs that follow merge the arguments of Ricolfi \cite{ricolfi1} and Oh--Thomas \cite{OT} and the reader is welcome to skip them. Let us also follow the example of Oh--Thomas \cite{OT} and write $a^\bullet$ for a particular representative of $a\in \textnormal{Hom}_{D^b(M)}(\mathcal{E}^\bullet,\mathcal{F}^\bullet)$ which is a true morphism of complexes $\mathcal{E}^\bullet,\mathcal{F}^\bullet\in D^b(M)$. Moreover, we will use $\mathcal{E}_{\bullet}$ to denote the dual of $\mathcal{E}^\bullet$ and similarly for $a_{\bullet}$.
	
	Here we recall that for a fixed class $\alpha\in K^0(X)$ and the moduli space of simple perfect complexes $M_{\alpha}$ (see Inaba \cite[Thm. 0.2]{Inaba} for the proof of its existence), the construction of the virtual fundamental classes always depends on the choice of the orientation on $M_{\alpha}$. This is a $\ZZ_2$ choice for each connected component. Changing the orientations of a connected component changes its virtual fundamental class by a sign. For the reader's benefit, we will avoid going into detail while still recalling some necessary background.
	\begin{definition}
		\label{Def: orientations}
		On the space $\mC_X$ by \cite[(4.13)]{bojko} and  Cao--Gross--Joyce \cite[Def. 3.9]{CGJ}, there is a natural construction of a $\ZZ_2$-bundle denoted by $O^{\omega}$ coming from gauge-theoretic data.  
  Its restrictions to each connected component $\mC_{\alpha}\subset \mC_X$ corresponding to a K-theory class $\alpha\in K^0(X)=\pi_0(\mC_X)$ is denoted by $O^{\omega}_{\alpha}$. In \cite{CGJ} and the auhtor's \cite{bojko}, these bundles are showed to be trivializable 
		thus leading to a $\ZZ_2$-choice of trivializations $$o_\alpha: \ZZ_2\longrightarrow O^\omega_{\alpha}$$
		called \textit{orientations}. These induce orientations on the moduli stack of simple perfect complexes $\mathfrak{M}_{\alpha}=M_{\alpha}\times [*/\GG_m]$ for the algebraic space $M_{\alpha}$.  We pull the orientations back along the natural map 
		$$
		\Gamma_{\alpha}:\mathfrak{M}_{\alpha}\longrightarrow \mC_{\alpha}\,,
		$$
		satisfying $\Gamma^*_{\alpha}(\mathcal{U}_\alpha) = \mathcal{E}_{\alpha}$ for the universal objects $\mathcal{U}_{\alpha}$ on $X\times \mC_{\alpha}$ and $\mathcal{E}_{\alpha}$ on $X\times \mathfrak{M}_{\alpha}$. This way we get orientations on $M_{\alpha}\times *$. In our case, we are working with $\alpha= \llbracket E_X\rrbracket -\llbracket F\rrbracket \,,$ where $[E_X\to F]$ is a point in $Q_X$. 
	\end{definition}
	
	\begin{lemma}
		\label{leRes}
		Let $E_X$ be RS, then the complex $\mathbb{E}$ is perfect of tor-amplitude $[-2,0]$. There is an isomorphism $\theta: \mathbb{E}\xrightarrow{\sim}\mathbb{E}^{\vee}[2]$. In particular, it admits a  three-term locally free resolution
		$$
		\mathcal{E}^\bullet =(\mathcal{T}\to \mathcal{E}\to \mathcal{T}^*)\xrightarrow{a^\bullet} \mathbb{E}
		$$
		which moreover comes with an isomorphism $\delta^\bullet: \mathcal{E}^{\bullet} \to \mathcal{E}_{\bullet}[2]$, such that $\delta^\bullet =  a_\bullet[2]\circ \theta\circ a^\bullet$. After picking an orientation $o_{\alpha}:\mathbb{Z}_2\xrightarrow{\sim}O_{\alpha}$ as in Theorem \cite[Thm. 5.4]{bojko}, we obtain a natural $SO(\mE,\mathbb{C})$ structure on $\mE$. 
	\end{lemma}
	\begin{proof}
		For a point $[E_X\to F]\in Q_X$, we will denote the associated short exact sequence by
		\begin{equation}
		\label{eqSes}
		¨0\to I\to E_X\to F\to 0\,.
		\end{equation}
		We first note down two useful identities:
		\begin{equation}
		\textnormal{Ext}^0(I,E_X)\cong \textnormal{Ext}^0(E_X,E_X)\cong \mathbb{C}\,,\qquad \textnormal{Ext}^1(I,E_X) =0\,,
		\end{equation}
		which follow from applying $\textnormal{Ext}^0(-,E_X)$ to \eqref{eqSes} and using that $E_X$ is locally-free. Using $\textnormal{Ext}^0(I,I) =H^0(X,\mathcal{O}_X)\oplus \textnormal{Ext}^0(I,I)_{0}$, the inclusion
		$$
		0\to \textnormal{Ext}^0(I,I)\hookrightarrow \textnormal{Ext}^0(I,E_X)\cong \CC
		$$
		induces 
		\begin{equation}
		\label{eqIsocon}
		\textnormal{Ext}^0(I,I)\cong \textnormal{Ext}^0(I,E_X)\cong \mathbb{C}.
		\end{equation}
		we therefore obtain $\textnormal{Ext}^i(I,I)_0 = 0$ for $i=0,4$ by Serre duality. Applying Proposition Oh--Thomas \cite[Prop. 4.1]{OT} and that $Q_X$ is projective, we obtain the resolution $\mathcal{E}^{\bullet}$ together with $\delta^\bullet$ and $a^\bullet$ satisfying the above conditions. Then applying \cite[Prop. 4.2.]{OT}, we obtain the choice of the $SO(E_X,\mathbb{C})$-structure.
	\end{proof}
	We now construct as in \cite[§4.3]{OT}(see also Behrend--Fantechi \cite[§ 4]{BF}) the cone $\mathcal{C}_{\mathcal{E}^{\bullet}}\subset \mE$ by using $\mE\to \mT^*$ as the two term perfect obstruction theory. 
	\begin{proposition}
		Let $E_X$ be RS then the cone $C_{\mathcal{E}^{\bullet}}\subset \mT$ is isotropic of dimension $\text{rk}(\mT)$. 
	\end{proposition}
	\begin{proof}
		The idea is to use the construction for coherent sheaves of Oh--Thomas \cite{OT}. This can be done, because of the following:
		We claim that there is a map $\phi: Q_X\to M_{\llbracket E_X\rrbracket-\llbracket F\rrbracket}$ and prove that $\phi$ is an open embedding by checking that the argument of Ricolfi \cite[§2.1 -§2.3]{ricolfi1} still apply.
		
		Note that it is injective by the definition of $Q_X$. Moreover, by applying $\textnormal{Ext}^0(I,-)$ to \eqref{eqSes}, we obtain 
		\begin{equation*}
		\begin{tikzcd}
		0\arrow[r]&\textnormal{Ext}^0(I,I) \arrow[r]&\textnormal{Ext}^0(I,E_X)\arrow[r,"u"]&\textnormal{Ext}^0(I,F)\\
		\arrow[r,"\partial"]& \textnormal{Ext}^1(I,I)\arrow[r]&0\arrow[r]&\textnormal{Ext}^1(I,F)\\
		\arrow[r,"\partial"]&\textnormal{Ext}^2(I,I)\,,
		\end{tikzcd}
		\end{equation*}
		where $u$ must be zero by \eqref{eqIsocon}. Therefore, we have 
		\begin{equation}
		\label{eqObs}
		\textnormal{Ext}^0(I,F)\xrightarrow{\sim} \textnormal{Ext}^1(I,I)\,,\qquad \textnormal{Ext}^1(I,F)\hookrightarrow \textnormal{Ext}^2(I,I)\,.
		\end{equation}
		By the proof of Ricolfi's \cite[Prop. 2.8]{ricolfi1} together with Serre's \cite[Rem. 2.3.8]{Ser} it follows from \eqref{eqObs} that $\phi$ is locally étale thus an open immersion. \\
		
		For the rest of the proof, we follow the arguments of Oh--Thomas \cite[p. 22, p. 32]{OT}. Let $\boldsymbol{\mathcal{M}}_X=\textbf{Map}(X,\textbf{\textnormal{Perf}})$ be the derived mapping stack for $\textbf{Perf}_{\mathbb{C}}$ the derived stack of perfect complexes. Let, $\mathcal{M}_{\llbracket E_X\rrbracket-\llbracket F\rrbracket} = [*/\mathbb{G}_m]\times M_{\llbracket E_X\rrbracket-\llbracket F\rrbracket}$. We have the natural $\tilde{m}: \mathcal{M}_{\llbracket E_X\rrbracket -\llbracket F\rrbracket}\rightarrow\boldsymbol{\mathcal{M}}$ together with $q:Q_X\to \boldsymbol{\mathcal{M}}$ induced by the respective universal complexes. Using the section $s: M_{\llbracket E_X\rrbracket - \llbracket F\rrbracket }\to  M_{\llbracket E_X\rrbracket - \llbracket F\rrbracket }\times [*/\mathbb{G}_m]$ and its composition $m =\tilde{m}\circ s$ gives the commutative diagram:
		$$
		\begin{tikzcd}
		Q_X\arrow[d,"\phi"]\arrow[dr,"q"]&\\
		M_{\llbracket E_X\rrbracket-\llbracket F\rrbracket}\arrow[r,"m"]&\boldsymbol{\mathcal{M}_X}
		\end{tikzcd}
		$$
		In particular, because $\phi$ is an open immersion we can reduce the constructions in loc. cit. to  $Q_X$. We give a short summary:
		
		Let $\mathfrak{q}: \uHom_{Q_X}(\mathcal{I},\mathcal{I})[3]= q^*\mathbb{L}_{\boldsymbol{\mathcal{M}}_X} \to \mathbb{L}_{Q_X}$ be the map induced by $q$, then we have the following commutative diagram:
		\begin{equation*}
		\begin{tikzcd}
		\arrow[d,equal]\underline{\textnormal{Hom}}_{Q_X}(\mathcal{I},\mathcal{I})_0[3]\arrow[r,"\textnormal{i}_1"] &\underline{\textnormal{Hom}}_{Q_X}(\mathcal{I},\mathcal{I})[3]\arrow[d,equal]\arrow[r,"\textnormal{At}"]&\mathbb{L}_{Q_X}\arrow[d,equal]\\
		\underline{\textnormal{Hom}}_{Q_X}(\mathcal{I},\mathcal{I})_0[3]\arrow[r,"\textnormal{i}_2"] &\underline{\textnormal{Hom}}_{Q_X}(\mathcal{I},\mathcal{I})[3]\arrow[r,"\mathfrak{q}"]&\mathbb{L}_{Q_X}
		\end{tikzcd}\,.
		\end{equation*}
		The map $\textnormal{i}_2$ is induced by the splitting $\boldsymbol{\mathcal{M}}_{\textnormal{det}(E_X)}\times \textbf{Pic}(X) =\boldsymbol{\mathcal{M}}_X $ étale locally around $[S]$. The square diagram on the right commutes by Schürg--Toën--Vezzosi \cite[Rem. A.1]{STV} and the one on the left does also because in both cases the map is obtained as the kernel of the trace map. The proof then follows by identical arguments as in Oh--Thomas \cite[Prop. 4.1, §4.4]{OT}. The main point is that we rely on a neighbourhood $\boldsymbol{U}\to \boldsymbol{\mathcal{M}}_X$ of $[I]$ from \eqref{eqSes}, where the map is of relative dimension 1 and induces an étale neighbourhood $t_0(\boldsymbol{U})\to Q_X$. We used $t_0(-)$ to denote the classical truncation of the derived affine scheme $\boldsymbol{U}$. Proving isotropy then relies on comparing a choice of a resolution on $Q_X$ and on $t_0(\boldsymbol{U})$, where the former has a cone which is canonically isotropic by the work of Bassat--Brav--Bussi--Joyce \cite[Thm. 2.10]{BBBJ} and Brav--Bussi--Joyce \cite[Ex. 5.16]{BBJ}.
	\end{proof}
	
	We are now able to construct the virtual fundamental classes.
	\begin{definition}
		Let $E_X$ be RS and let $\sqrt{0^!_{\mathcal{E}^\bullet}}:A_*(C_{\mE^{\bullet}})\to A_{*-\frac{\text{rk}(\mE)}{2}}(Q_X)$ be the square-root Gysin map of Oh--Thomas \cite[Def. 3.3]{OT} for a fixed $SO(\mathcal{E},\mathbb{C})$ structure as in Lemma \ref{leRes}, then we define
		$$
		[Q_X]^{\vir} = \sqrt{0^!_{\mathcal{E}^\bullet}}[C_{\mathcal{E}^\bullet}]\in A_{e\cdot n}(Q_X)\,.
		$$
		It is independent of the choice of $\mathcal{E}^{\bullet}$ by \cite[p. 28-29]{OT}. 
	\end{definition}

	\subsection{Minimal input data}
	As we already explained in § \ref{Sec: our methods}, by using the ideas in \cite{bojko2}, we only need a minuscule amount of knowledge about the virtual fundamental class of $\text{Quot}_S(\mO_S,n) =\text{Hilb}^n(Y)$ to recover the information about all descendants of the class $[\QuotY]^\vir$. We use this section to note down the necessary invariants for later use. 
	
	Let $Y=C,S, X$.  We only need one ingredient for all the results discussed above, . Let us denote
	$$
	I_{Y,L}(q) = 1+\sum_{n>0}q^n I_{Y,n}(L)\,, 
	$$
	for the invariants 
	$$
	I_{Y, n}(L) = \int_{[Q_Y]^\vir}s\big((-L)^{[n]}\big)\,.
	$$
	When $E_Y = \mO_Y$, $\beta=0$ we have  $\textnormal{Quot}_Y(\mathbb{C}^1,n)= Y^{[n]}$ and the virtual fundamental classes get identified with
	\begin{align*}
	\big[\textnormal{Quot}_S(\mathbb{C},n)\big]^{\textnormal{vir}} &= \big[S^{[n]}\big]\cap c_n(K_S^{[n]}\,^\vee)\,,\\
	\big[\textnormal{Quot}_C(\mathbb{C},n)\big]^\vir &= \big[C^{[n]}\big]\,,\\
	\big[\textnormal{Quot}_X(\mathbb{C},n)\big]^{\textnormal{vir}} &= \big[X^{[n]}\big]^{\vir}
	\end{align*}
	using that $S^{[n]}, C^{[n]}$ are smooth and taking the first equality from Oprea--Pandharipande \cite[(31)]{OP1}. 
	
	From now on we will denote $c_i(Y)$ by $c_i$ for Chern-classes when $Y$ is understood. From \cite[Cor. 15]{OP1}, we obtain 
	\begin{equation}
	\label{eqrk1}
	I_{S,L}(q)= A(q)^{c_1(L)\cdot c_1}B^{c_1^2}\,.
	\end{equation}
	where 
	$$A(q)=1-t\qquad B=-\Big(\frac{q}{t}\Big)^2\Big(\frac{dq}{dt}\Big)^{-1}\qquad q=-t(1-t)^{-1}\,.$$
	A simple computations shows that 
	\begin{equation}
	\label{eqABs}
	A(q)=\frac{1}{1-q}\,,\qquad B=1 \,.
	\end{equation}
	
	%We have then \begin{align*}
	%[\textnormal{Quot}_S(\mathbb{C}^1,n)]^{\textnormal{vir}}\cap c_n(L^{[n]}) &= [\textnormal{Hilb}^n(S)]\cap c_n\big((K_S^{[n]})^\vee\big)\cap c_n\big(L^{[n]}\big)\\
	%&=(-1)^n[\textnormal{Hilb}^n(S)]\cap c_{2n}(K^{[n]}_S\oplus L^{[n]})\,,\end{align*} 
	%so we can use \cite[eq. (18)]{MOPhigher}, which tells us
	%\begin{equation}
	%\label{eqILq}
	%I(L,q) = \Big(1-q\Big)^{c_1(L)\cdot c_1}\,.
	%\end{
	We will avoid addressing the curve case using wall-crossing, because it is not a fruitful approach while for $Y=X$ the result is parallel but depends on the additional choices of orientations as in \cite[§3.1]{bojko2}. As we already did compute the virtual fundamental classes of $0$-dimensional sheaves in \cite[Thm. 3.10]{bojko2}, we only recall their formula in Lemma \ref{lemNnp} below.
	
	\section{Vertex algebras and wall-crossing for quot-schemes}
	\label{SecVACQ}
	The main machinery used for computing interesting invariants are the vertex algebras of Joyce \cite{Joycehall} coming from a geometric construction. We explain the necessary topological data needed to construct them and how they lead to Lie algebras, where the wall-crossing takes place. We finish by discussing the resulting wall-crossing formulae in the algebro-geometric setting and the induced formulae in the topological setting. 
	
	\subsection{Vertex algebras}
	\label{VAinAG}
	Let us recall first the definition of vertex algebras focusing on \textit{graded lattice vertex algebras}. For background literature, we recommend  \cite{Borcherds, KacVA, BZVA, FLM, LLVA, gebert}, with Borcherds \cite{Borcherds} being most concise.
	
	\begin{definition}
		A \textit{$\mathbb{Z}$-graded vertex algebra} over the field $\mathbb{Q}$ is a collection of data $(V_*,T,\ket{0}, \mathsf{Y})$, where $V_*$ is a $\ZZ$-graded vector space, $T: V_*\to V_{*+2}$ is graded linear, $\ket{0}\in V_0$, $\mathsf{Y}: V_*\to \textnormal{End}(V_*)\llbracket z,z^{-1}\rrbracket$ is graded linear after setting $\textnormal{deg}(z)=-2$,  satisfying the following set of assumptions.
		Let $u,v,w\in V_*$, then 
		\renewcommand{\theenumi}{\roman{enumi}}
		\begin{enumerate}
			\item We always have $\mathsf{Y}(u,z)v\in V_*(\!(z)\!)$,
			\item $\mathsf{Y}(\ket{0},z)v = v$, 
			\item $\mathsf{Y}(v,z)\ket{0}=e^{zT}v$, 
			\item Let $\delta(z)=\sum_{n\in \ZZ}z^n\in\mathbb{Q}\llbracket z,z^{-1}\rrbracket$ \begin{align*}
			z_2^{-1}\delta\Big(\frac{z_1-z_0}{z_2}\Big)\mathsf{Y}&\big(\mathsf{Y}(u,z_0)v,z_2)w = z_0^{-1}\delta\Big(\frac{z_1-z_2}{z_0}\Big)\mathsf{Y}(u,z_1)\mathsf{Y}(v,z_2)w \\&- (-1)^{\textnormal{deg}(u)\textnormal{deg}(v)}z_0^{-1}\delta\Big(\frac{z_2-z_1}{-z_0}\Big)\mathsf{Y}(v,z_2)\mathsf{Y}(u,z_1)w\,,
			\end{align*}
			where $z_i$ are different variables for $i=0,1,2$.
		\end{enumerate}
	\end{definition}
	By Borcherds \cite{Borcherds}, the graded vector-space $V_{*+2}/T(V_*)$ carries a graded Lie algebra structure determined by
	\begin{equation}
	\label{eqLia}
	[\bar{u},\bar{v}]=[z^{-1}]\mathsf{Y}(u,z)v\,,  
	\end{equation}
	where $\overline{(-)}$ denotes the equivalence class in the quotient.
	Let $A^{\pm}$ be abelian groups and $\chi^{\pm}: A^{\pm}\times A^{\pm}\to \ZZ$ be symmetric bi-additive maps. Let us denote $\mathfrak{h}^{\pm} = A^{\pm}\otimes \mathbb{Q}$ and fix a basis $B^{\pm}$ of $\mathfrak{h}^{\pm}$ with the total basis being 
	$$
	B=B^+\sqcup B^-\,.
	$$
	For $(A^\bullet,\chi^\bullet)$, where $A^\bullet =A^+\oplus A^-$ and $\chi^{\bullet} = \chi^+\oplus \chi^-$ and a choice of a group 2-cocycle $\epsilon: A^+\times A^+\to \ZZ_2$ satisfying
	\begin{equation}
	\label{eqSig}
	\epsilon_{\alpha,\beta} = (-1)^{\chi^+(\alpha,\beta)+\chi^+(\alpha,\alpha)\chi^+(\beta,\beta)}\epsilon_{\beta,\alpha}   \,,\quad  \forall \,\,\alpha,\beta\in A^+
	\end{equation}
	there is a natural graded vertex algebra on
	\begin{align*}
	V_*&=\QQ[A^+]\otimes_{\QQ}\textnormal{SSym}_{\QQ}\llbracket u_{v,i}, v\in B,i>0 \rrbracket\\
	&\cong\mathbb{Q}[A^+]\otimes _{\mathbb{Q}}\textnormal{Sym}_{\QQ}\big(A^+\otimes t^2\QQ[t^2]\big)\otimes_{\QQ}\Lambda_{\QQ}\big(A^-\otimes t\QQ[t^2]\big)\,,    
	\end{align*}
	where $u_{v,i}$ are formal variables of degree 
	$$
	\begin{cases}
	2i&\text{if}\quad v\in B^+\\
	2i-1&\text{if}\quad v\in B^-
	\end{cases}\,.
	$$
	The vertex algebra structure is called the \textit{lattice vertex algebra} $(V_*,T,\ket{0},\mathsf{Y})$ associated to $(A^\bullet,\chi^\bullet)$.  Its fields can be described explicitly as follows. Let $\alpha\in A^+$, such that $\alpha=\sum_{v\in B^+}\alpha_v v$. We use $e^{\alpha}$ to denote the corresponding element in $\QQ[A^+]$. For $K\in \textnormal{SSym}_{\QQ}\llbracket u_{v,i}, v\in B,i>0 \rrbracket$ any $\alpha,\beta\in A^+$ and $v\in \mathfrak{h}^{\pm}$ we have
	\begin{align*}
	\label{fields}
	\mathsf{Y}(e^{0}\otimes u_{v,1},z)e^{\beta}\otimes K &= e^\beta\otimes \Big\{\sum_{k>0}u_{v,k}\cdot Kz^{k-1} \\
	+ &\sum_{k>0}\sum_{w\in B}k^{(1-\delta_v)}\chi^{\bullet}(v,w)\frac{dK}{d u_{w,k}}z^{-k-1+\delta_v} + \chi^\bullet(v,\beta)z^{-1}\Big\}\,,\\
	\mathsf{Y}(e^{\alpha}\otimes 1,z)e^\beta\otimes K &= \epsilon_{\alpha,\beta}z^{\chi^+(\alpha,\beta)}e^{\alpha+\beta}\otimes \textnormal{exp}\Big[\sum_{k>0}\sum_{v\in B^+}\frac{\alpha_v}{k}u_{v,k}z^k\Big]\\
	&\textnormal{exp}\Big[-\sum_{k>0}\sum_{v\in B^+}\chi^+(\alpha,v)\frac{d}{du_{v,k}}z^{-k}\Big]K\,,
	\numberthis
	\end{align*}
 where 
 $$
 \delta_v =\begin{cases}
     0&\text{if}\quad v\in \mathfrak{h}^+\\
    1&\text{if}\quad v\in \mathfrak{h}^-
 \end{cases}\,.
 $$
	Note that by the reconstruction lemma in Lepowsky--Li \cite[Thm. 5.7.1]{LLVA}, Ben--Zvi \cite[Thm. 4.4.1]{BZVA} and Kac \cite[Thm. 4.5]{KacVA}, these formulae are sufficient for determining all fields. These structures were observed to appear in algebraic geometry and topology by Joyce \cite{Joycehall}. We now give a simplified description. 
 
 When $\pi_{I}:\prod_{k\in K}Z_k\to \prod_{i\in I} Z_i$ is the projection to factors labeled by $i\in I\subset K$, we will use the notation 
			$$
			(\Theta)_{I}= \pi_{I}^*(\Theta)\,,
			$$
			when $\Theta$ is a perfect complex, K-theory class, or cohomology class on $\prod_{i\in I} Z_i$.
	\begin{definition}
		\label{DeftopVA}
		Let $(\mathcal{C},\mu, 0)$ be an H-space (see Hatcher \cite{Hatcher})with a $\mathbb{C}\mathbb{P}^{\infty}$ action $\Phi:\mathbb{C}\mathbb{P}^{\infty}\times \mathcal{C}\to \mathcal{C}$ which is an $H$-map with respect to the multiplication $\mu$ an identity $0$. Let $\theta\in K^0(\mathcal{C}\times \mC)  =[\mathcal{C}\times \mC,BU\times\ZZ]$ be a K-theory class satisfying $\sigma^*(\theta) = \theta^\vee$, where $\sigma:\mC\times\mC\to \mC\times \mC$ permutes the two factors, together with
		\begin{align*}
		(\mu\times \textnormal{id}_{\mathcal{C}})^*(\theta)= (\theta)_{1,3}+ (\theta)_{2,3}\,,&\qquad 
		(\textnormal{id}_{\mathcal{C}}\times \mu)^*(\theta)=(\theta)_{1,2}+ (\theta)_{1,3}\,,\\
		(\Phi\times \textnormal{id}_{\mathcal{C}})^*(\theta) = \mV_1\boxtimes \theta\,,&\qquad
		(\textnormal{id}_{\mathcal{C}}\times \Phi)^*(\theta) = \mV_1 ^*\boxtimes \theta\,,
		\end{align*}
		where $\mV_1\to \mathbb{C}\mathbb{P}^{\infty}$ is the universal line bundle.
		
		Let $\pi_0(\mathcal{C})\to K$ be a morphism of commutative monoids. Denote $\mathcal{C}_\alpha$ to be the open and closed subset of $\mathcal{C}$ which is the union of connected components of $\mathcal{C}$ mapped to $\alpha\in K$. We write  $\theta_{\alpha,\beta} = \theta|_{\mathcal{C}_\alpha\times \mathcal{C}_\beta}$, and $\chi^+(\alpha,\beta) = \textnormal{rk}(\theta_{\alpha,\beta})$ must be a symmetric bi-additive form on $K$. Let $\epsilon: K\times K\to  \{-1,1\}$ satisfying
		\begin{align*}
		\label{signs}
		\epsilon_{\alpha,\beta} &= (-1)^{\chi^+(\alpha,\beta)+\chi^+(\alpha,\alpha)\chi^+(\beta,\beta)}\epsilon_{\beta,\alpha}   \,,\quad  \forall \,\,\alpha,\beta\in A^+\\
		\epsilon_{\alpha,0} &= \epsilon_{0,\alpha}=1\,,\qquad
\epsilon_{\alpha,\beta}\epsilon_{\alpha+\beta,\gamma}=\epsilon_{\beta,\gamma}\epsilon_{\alpha,\beta+\gamma}\,,
		\numberthis
		\end{align*}
		be a group 2-cocycle and $\hat{H}_a(\mathcal{C}_\alpha) = H_{a-\chi(\alpha,\alpha)}(\mathcal{C}_\alpha)$. Then we denote by $$(\hat{H}_*(\mathcal{C}), \ket{0},T, \mathsf{Y})$$ the vertex algebra on the graded $\QQ$-vector space $\hat H_*(\mathcal{C}) =\bigoplus_{\alpha\in K}\hat H_*(\mathcal{C}_{\alpha})$ defined from the data $$(\mathcal{C}, K(\mathcal{C}),\Phi,\mu,0,\theta, \epsilon)$$ by 
		\begin{itemize}
			\item $\ket{0} =0_*(*)$, where $*$ denotes the generator of $H_0(\text{pt})$ and $T(u) = \Phi_*(t\boxtimes u)$, where $t$ is the dual of $c_1(\mV_1)$,
			\item the state to field correspondence $Y$ is given by
			\begin{align*}
			\mathsf{Y}(u,z)v = \epsilon_{\alpha,\beta} (-1)^{a\chi^+(\beta,\beta)}z^{\chi^+(\alpha,\beta)}\mu_*(e^{zT}\otimes \textnormal{id})\big((u\boxtimes v)\cap c_{z^{-1}}(\theta_{\alpha,\beta})\big)\,,
			\end{align*}
			for all $u\in \hat{H}_a(\mathcal{C}_\alpha)$, $v\in \hat{H}_b(\mathcal{C}_\beta)$. 
		\end{itemize}
		\label{definition haltop}
	\end{definition}
	\begin{remark}
 The complex $\theta$ has a geometric interpretation which is useful to keep in mind when constructing these vertex algebras in relation to wall-crossing. For this, we fix $\mC$ to be a scheme of perfect complexes on $Y$ which admits an obstruction theory 
		$$
		\EE = \uHom_{\mC}(\mE,\mE)^\vee[-1] \longrightarrow \LL_{\mC}\,,
		$$
		where $\mE$ is the universal perfect complex on $Y\times \mC$. If
		$$
		\mu: \mC\times \mC\longrightarrow \mC
		$$
		were an embedding induced by taking direct sums of perfect complexes, then the virtual co-normal bundle would be given by
  \footnote{We work in the next equation with the factors $Y\times \mC\times \mC$ in this specified order.}
		$$
		\theta[1]= \uHom_{\mC\times \mC}\big(\mE_{1,2},\mE_{1,3}\big)^\vee\oplus \uHom_{\mC\times \mC}\big(\mE_{1,3},\mE_{1,2}\big)^\vee\,.
		$$
		In particular, if we pull back the first summand along the diagonal $\Delta: \mC\to \mC\times \mC$ then we recover $\EE[1]$. The general set up for quot-schemes is similar, but we need to adapt the class $\theta$ to the pair obstruction theory when $Y=S,C$.
	\end{remark}
	
	The wall-crossing formulae in Joyce \cite{JoyceWC}, Gross--Joyce--Tanaka \cite{GJT} are expressed in terms of a Lie algebra defined by Borcherds \cite{Borcherds}. Let $(\hat{H}_*(\mathcal{C}), \ket{0},T, \mathsf{Y})$ be the vertex algebra from Definition \ref{definition haltop} and define
	$$\Pi_{*+2}:\hat{H}_{*+2}(\mathcal{C})\longrightarrow \check{H}_*(\mathcal{C}) = \hat{H}_{*+2}(\mathcal{C})/T\big(\hat{H}_{*}(\mathcal{C})\big)\,,$$
	then this homology has a natural Lie algebra structure given by
	\begin{equation}
	\label{lie algebra}
	[\bar{u},\bar{v}]=[z^{-1}]\{\mathsf{Y}(u,z)v\}\,, 
	\end{equation}
	where $\bar{U}$ denotes the equivalence class in $\hat{H}_{*+2}(\mathcal{C})$.
	
	The following remark discusses briefly, how to reduce working with (higher) stacks and perfect complexes to topological spaces and K-theory classes, but a reader not familiar with this machinery can pretend  that any stack with a perfect complex is a topological space with a K-theory class. 
	
	\begin{remark}
 \label{rem:toprel}
		For a higher stack $\mathcal{S}$, we denote by $H_*(\mathcal{S}) = H_*(\mathcal{S}^{\textnormal{top}})$, $H^*(\mathcal{S}) = H^*(\mathcal{S}^{\textnormal{top}})$ its Betti (co)homology as in Joyce \cite{Joycehall}, Gross \cite{gross} using topological realization functor of Blanc \cite{Blanc}. Note that we will always treat $H_*(T)$ as a direct sum and $H^*(T)$ as a product over all degrees. Following May--Ponto \cite[§24.1]{MayPonto} define the topological K-theory of $\mathcal{S}$ to be
		$$
		K^0(\mathcal{S}) =[\mathcal{S}^{\textnormal{top}}, BU\times\ZZ]\,,
		$$
		where $[X,Y] =\pi_{0}\big(\textnormal{Map}_{C^0}(X,Y) \big)$. For any perfect complex $\mathcal{E}$ on $\mathcal{S}$, there is a unique map $\phi_{\mathcal{E}}: \mathcal{S}\to \textnormal{Perf}_{\CC}$ in $\textbf{Ho}(\textbf{HSt})$ satisfying $\phi_{\mathcal{E}}^*(\mU) = \mE$ for the universal complex $\mU$ on $\textnormal{Perf}_{\CC}$. Using Blanc \cite[§4.1]{Blanc}, this gives
		$$
		\llbracket \mathcal{E}\rrbracket: \mathcal{S}^{\textnormal{top}}\longrightarrow BU\times \ZZ\,.
		$$
		in $\textbf{Ho}(\textbf{Top})$. We then have a well defined map assigning to each perfect complex $\mathcal{E}$ its class $\llbracket \mathcal{E}\rrbracket\in K^0(\mathcal{S})$.  
		Using these constructions we will construct the data in Definition \ref{DeftopVA} out of the corresponding lifts to algebraic geometry.
	\end{remark}
	\subsection{Vertex algebra construction}
	This section builds examples of the construction in Definition \ref{DeftopVA} by describing the precise topological data needed to work with quot-schemes parametrizing the morphisms
	\begin{equation}
	\label{Eq: EYtoF}
	W\otimes E_Y\longrightarrow F\,.
	\end{equation}
	Here, $E_Y$ is fixed, $W$ is a vector space of a fixed dimension giving a point in $BU$ and $F$ a zero-dimensional sheaf. After doing so for the stack of morphisms of the form \eqref{Eq: EYtoF} the construction of which was given by Joyce  \cite[Def. 5.5]{JoyceWC} when $E_Y = \mO_Y$, we address its purely topological analog on $\mP_Y = (\BU\times \ZZ)\times \mC_Y$. To compare the two, one should view $BU\times \ZZ$ as keeping track of the vector space $U$ and $\mC_Y$ seeing the K-theory class of $F$. These explicit descriptions are useful for describing the appropriate data of the vertex algebra, especially the complexes governing the obstruction theory.

	\begin{definition}
		\label{defAMTH}
		\begin{enumerate}[(i)]
			\item For $Y=X, S, C$ we consider the abelian category $\mathcal{A}_{E_Y}$ of  triples $(F,U,\phi)$, where $U$ is a vector space $\phi: W\otimes E_Y\rightarrow F$ and $F$ is a zero-dimensional sheaf. The morphisms are the pairs $(v,f):(F_1,U_1,\phi_1)\to (F_2,U_2,\phi_2)$, where $v:U_1\to U_2$ is a linear map and $f:F_1\to F_2$ a morphism of $\mO_X$-modules satisfying $\phi_2\circ v = f\circ \phi_1$.
			\item We define $K(\mathcal{A}_{E_Y}) = \ZZ\times \ZZ$ and the cone $C(\mathcal{A}_{E_Y}) = \{(n,d)\in\ZZ\times \ZZ|\,n,d\geq 0 \quad\textnormal{and}\quad (d,n)\neq 0  \}$. The integer $d$ keeps track of the rank of $W$ while $n$ keeps track of the number of points in the quotient. 
			\item Let $\mathcal{E}_{np}$ be the universal sheaf on $Y\times \mathcal{M}_{np}$, where $\mathcal{M}_{np}$ is the moduli stack of 0-dimensional sheaves in class $np$ and $p=\llbracket \mathcal{O}_y\rrbracket$ for some $y\in Y$. From now on, we leave out $p$ from the subscript in the notation where possible.
   
   The moduli stack $\mathcal{N}^{E_Y}_0$ is constructed as in Definition \cite[Def 2.12]{bojko2} (see also \cite[Def. 4.5.17]{PhD}), except in the first bullet point of  \cite[Def 2.12]{bojko2}, we take the total space of the vector bundle 
			$$\pi_{n,d}: \pi_{2\,*}\big(\pi_1^*(E^*_Y)^{}\otimes\mathcal{E}_{n}\big)\boxtimes \mathcal{V}^*_d\to \mathcal{M}_{n}\times [*/\textnormal{GL}(d,\CC)]\,,$$
   where $\mV_d$ is the universal vector bundle on $[*/\GL(d,\CC)]$. We then set
			$$
			\mathcal{N}^{E_Y}_0 = \bigsqcup_{\begin{subarray}a (n,d)\\
   n,d\geq 0\end{subarray}}\mathcal{N}^{E_Y}_{(n,d)} \,,
			$$
			where 
			$$
			\mathcal{N}^{E_Y}_{(n,d)} = \begin{cases}\textnormal{Tot}\big(\pi_{n,d}\big)&\textnormal{if}\quad n,d\neq 0\,,\\
			\mM_{n}&\textnormal{if}\quad  d=0,n\neq 0\,,\\
			[*/\textnormal{GL}(d,\CC)]&\textnormal{if}\quad d\neq 0,n=0\,.
			\end{cases}
			$$
			and empty otherwise. 
			\item For the stack $\mN^{E_Y}_0$, we will construct the perfect complex $\Theta^{E_Y,\textnormal{pa}}$ on $\mN^{E_Y}_0\times \mN^{E_Y}_0$ such that 
			$$
			\Delta^*(\Theta^{E_Y,\textnormal{pa}})[-1] 
			$$
			is the obstruction theory of the stack which at each point $[W\otimes E_Y\to F]$ is given by
			$$
			\begin{cases} \text{Ext}^\bullet(W\otimes E_Y\to F,F)&\text{if}\quad Y=C,S\\
			\text{Ext}^\bullet(W\otimes E_Y\to F, W\otimes E_Y\to F) &\text{if}\quad Y=X\,.
			\end{cases}
			$$
			Fixing the order $Y\times \mM_{n_1}\times \mM_{n_2}$ of products, let $$\Theta_{n_1,n_2}=\underline{\textnormal{Hom}}_{\mathcal{M}_{n_1}\times \mathcal{M}_{n_2}}\Big((\mathcal{E}_{n_1})_{1,2},(\mathcal{E}_{n_2})_{1,3}\Big)^\vee\,.$$
			We define $\Theta^{E_Y,\textnormal{ob}}$, which captures the obstruction theory of pairs \eqref{eqClF} and \eqref{eq:CYob}, by
			\begin{equation*} 
			\Theta^{E_Y,\textnormal{ob}}_{(n_1,d_1),(n_2,d_2)} =\begin{cases} (\pi_{n_1,d_1}\times \pi_{n_2,d_2})^*\\
			\Big\{(\Theta_{n_1,n_2})_{1,3}\oplus \Big( \mathcal{V}_{d_1} \boxtimes\pi_{2\,*}\big(\pi^*_1(E^*_Y)\otimes  \mathcal{E}_{n_2}\big)^\vee\Big)_{2,3}[1]\Big\}&\textnormal{if} \quad Y=S,C\\
			&
			\\
			(\pi_{n_1,d_1}\times \pi_{n_2,d_2})^*&\\
			\Big\{(\Theta_{n_1 ,n_2 })_{1,3} \oplus \Big((V_{d_1}\boxtimes V^*_{d_2})^{\oplus \chi(E^*_Y\otimes E_Y)}\Big)_{2,4}&\\
			\oplus  \Big(\mathcal{V}_{d_1}\boxtimes \pi_{2\,*}\big(\pi^*_1(E^*_Y)\otimes\mathcal{E}_{n_2 }\big)^\vee\Big)_{2,3}[1]
			&\\\oplus \Big( \pi_{2\, *}\big( \pi^*_1(E^*_Y)\otimes \mathcal{E}_{n_1 }\big)^\vee\boxtimes\mathcal{V}_{d_2}\Big)_{1,4} [1]\Big\}\,.&\textnormal{if}\quad Y=X
			\end{cases}
			\numberthis
			\end{equation*}
   	Working with $Y=C,S$, the correct pair vertex algebra structure requires the symmetrization of $\Theta^{E_Y,\textnormal{ob}}$: 
 $$\Theta^{E_Y,\textnormal{pa}} = \Theta^{E_Y,\textnormal{ob}}\oplus\sigma^*\big(\Theta^{E_Y,\textnormal{ob}}\big)^\vee\,.$$
 	In the case of $Y=X$, the K-theory class is already self-dual, so we follow \cite{bojko2} by setting $$\Theta^{E_Y,\pa} = \Theta^{E_Y,\ob}\,.$$
  A vague idea of why this is allowed is because for the obstruction theory
	$$
	\EE = \Delta^*(\Theta^{E_Y,\pa})[-1]
	$$
	we should think of it as splitting in terms of some Behrend--Fantechi \cite{BF} obstruction theory $\FF$:
	$$
	\EE =\FF\oplus \FF^\vee[2]\,.
	$$
	This is known to happen on certain moduli schemes  (see for example Diaconescu--Sheshmani--Yau \cite[Prop. 4.2]{DYS}). This suggests that hidden in the class $\llbracket \Theta^{E_Y,\pa}\rrbracket$ is a symmetrized expression of the form in \eqref{Eq: pairvertexalgebra}. Moreover, it is really only this one half $\FF$ of $\EE$ that is used by Oh--Thomas \cite{OT} to construct the virtual fundamental class and by me \cite{bojko3} to prove the wall-crossing.
 \item	The form is computed by taking ranks \begin{equation*}
			\chi^{E_Y,\textnormal{pa}}\big((n_1,d_1),(n_2,d_2)\big)= \textnormal{rk}\Big(\Theta^{E_Y,\textnormal{ob}}_{(n_1,d_1),(n_2,d_2)}\Big)=\begin{cases}
			-e(d_1n_2 +d_2n_1)&\textnormal{if}\quad Y=C,S\\&
			\\
			+ \chi(E^*_Y\otimes E_Y) d_1d_2&\\
			- e(d_1n_2 + d_2n_1)&\textnormal{if}\quad Y=X
			\end{cases}\,.
			\end{equation*}
   \end{enumerate}
The rest of the data from Definition \cite[Def. 2.12]{bojko2} has an obvious modification, which we do not mention here as we do not use it explicitly.
	\end{definition}

 Thus the correct data for constructing vertex algebras using the notation of Definition \ref{DeftopVA} and topological realization of Remark \ref{rem:toprel}  is
	\begin{equation}
	\label{Eq: pairvertexalgebra}
	\big((\mathcal{N}^{E_Y}_0)^{\textnormal{top}},\ZZ\times \ZZ,\mu_{\mathcal{N}^{E_Y}_0}^{\textnormal{top}},\Psi_{\mathcal{N}^{E_Y}_0}^{\textnormal{top}}, 0^{\textnormal{top}}, \llbracket\Theta^{E_Y,\textnormal{pa}}\rrbracket,\epsilon^{E_Y,\pa}\big)\,,
	\end{equation}
	where $\epsilon^{E_Y,\textnormal{pa}}_{(n_1,d_1),(n_2,d_2)} = (-1)^{ed_1n_2}$ for $Y=C,S$.
	The signs that were used for $Y=X$ are important as they are meant to make up for the freedom to choose orientations. They are given in terms of
	$$
	\epsilon^{E_Y,\textnormal{pa}}_{(n_1,d_1),(n_2,d_2)}= \epsilon_{n_1p-d_1\llbracket E_Y\rrbracket,n_2p -d_2\llbracket E_Y\rrbracket}\,,
	$$
	where $\epsilon_{\alpha,\beta}$ for all $\alpha,\beta\in K^0(X)$ are meant to compare the orientations on each connected component in Definition \ref{Def: orientations} under direct sums. They are constructed in the next theorem.  
	\begin{theorem}[Cao--Gross--Joyce \cite{CGJ}, B.\cite{bojko}]
		\label{Thm: orientations}
		Let $\mC_{\alpha}$ be the connected component labeled by $\alpha\in \pi_0(\mathcal{C}_X) = K^0(X)$ such that $O^{\omega}_\alpha = O^{\omega}|_{\mathcal{C}_\alpha}$ as in Definition \ref{Def: orientations}, then there are natural isomorphisms
		$$
		\phi^{\omega}_{\alpha,\beta}: O^{\omega}_{\alpha}\boxtimes O^{\omega}_{\beta}\to \mu^*(O^{\omega}_{\alpha+\beta})\,,
		$$
		where $\mu: \mC_Y\times \mC_Y\to \mC_Y$ corresponds to taking sums of K-theory classes. Fixing any collection of orientations $\{o_{\alpha}: \ZZ\to O^\omega_{\alpha}\}$, there exist unique signs $\epsilon_{\alpha,\beta}$ for all $\alpha,\beta\in K^0(X)$ satisfying \eqref{signs}, such that
		\begin{equation}
		\label{Eq: signcomparison}
		\phi^{\omega}_{\alpha,\beta}(o^\omega_\alpha\boxtimes o^\omega_\beta) = \epsilon_{\alpha,\beta}o^\omega_{\alpha+\beta}\,.
		\end{equation}
		
	\end{theorem}
	\begin{definition}
		For all $Y=C,S,X$, we denote the vertex algebras above by $V_* = \hat{H}_*(\mN^{E_Y}_0)$.
	\end{definition}
 The geometric realization of the Lie algebra of Borcherds in \eqref{eqLia} already observed by Joyce \cite{Joycehall} is the shifted homology 
 $$
\check{H}_*\big((\mN^{E_Y}_0)^{\pl}\big) = H_{*+2-2\chi^{E_Y,\pa}}\big((\mN^{E_Y}_0)^{\pl}\big)\,,
 $$
 where $\chi^{E_Y,\pa}$ denotes the shift by $\chi^{E_Y,\pa}\big((n.d),(n,d)\big)$  on each component of $\mN^{E_Y}_0$ labeled by $(n,d
)$. This is because, by \cite{Joycehall}, there is an isomorphism 
\begin{equation}
\label{Eq:rigcheck}
\check{H}_{*}(\mathcal{N}^{E_Y}_0) \cong \check{H}_*\big((\mN^{E_Y}_0)^{\pl}\big)
\end{equation}
outside of the zero component.

	\subsection{Twisted vertex algebras and vertex algebras of topological pairs}
	An essential ingredient for obtaining explicit results already used in \cite{bojko2} was the $L$-\textit{twisted vertex algebra} for a line bundle $L\to Y$. It was introduced to allow us to compute out of the invariants 
	$$
	I_{Y,n}(L) = \int_{[\Hilb]^{\vir} }c_n(L^{[n]}) 
	$$
	the invariants counting zero-dimensional sheaves on $Y$, because these satisfy a wall-crossing formula in the Lie algebra associated with this twisted vertex algebra. We recall its construction and the resulting wall-crossing formula, where we only apply it to the case when $E_Y = \mO_Y$. We suppress $\mO_Y$ in the notation of the previous subsection unless strictly necessary.
	\begin{definition}[{\cite[Def. 3.2]{bojko2}}]
		\label{Def: twistedVOA}
		For all $(n_i,d_i)\in C(\mA_{\mO_Y})$, $i=1,2$, we define
		$$
		\mathcal{L}^{[n_1,n_2]}_{d_1,d_2}\to \mathcal{N}_{n_1,d_1}\times \mathcal{N}_{n_2,d_2}
		$$
		by
		$$
		\mathcal{L}^{[n_1,n_2]}_{d_1,d_2}= (\pi_{n_1,d_1}\times \pi_{n_2,d_2})^*\Big(\mathcal{V}^*_{d_1}\boxtimes {\pi_2}_*\big(\pi_X^*(L)\otimes \mathcal{E}_0\big)\Big)_{2,3}\,,
		$$
		where $\mathcal{E}_0$ is the universal sheaf on $Y\times\mathcal{M}_0$ the moduli stack of 0-dimensional sheaves. It is a vector bundle of rank $d_1n_2$. We define
		$$
		\mathcal{L}^{[-,-]}|_{\mathcal{N}_{n_1,d_1}\times \mathcal{N}_{n_2,d_2}} = \mathcal{L}^{[n_1,n_2]}_{d_1,d_2}\,.
		$$
		We modify the data \eqref{Eq: pairvertexalgebra} out of which the vertex algebra is constructed by replacing 
		\begin{align*}
		\Theta^{\pa} &\mapsto \Theta^{\pa} + \mathcal{L}^{[-,-]}\\
		\epsilon^{\textnormal{pa}}_{(n_1,d_1),(n_2,d_2)} = (-1)^{d_1n_2}&\mapsto (-1)^{d_1n_2}\epsilon^{\textnormal{pa}}_{(n_1,d_1),(n_2,d_2)} = 1\,.
		\end{align*}
		\sloppy This leads to a vertex algebra $V_*^L = (\tilde{H}_*(\mathcal{N}_0),\ket{0},T, \mathsf{Y}^{L})$ and its associated Lie algebra $\big(\mathring{H}_*(\mathcal{N}_0), [-,-]^L\big)$. 
	\end{definition}
	The consequence of Theorem \cite[Thm. 2.12]{GJT} of Gross--Joyce--Tanaka is that these vertex algebras are related by  cap product.
	\begin{lemma}[{\cite[Lem. 3.3]{bojko2}}]
		\label{lemma Lconditions}
		We have the morphism $$(-)\cap c_{\textnormal{top}}\big(\mathcal{L}\big):\big(\check{H}_*(\mathcal{N}_0), [-,-]\big)\to \big(\mathring{H}_*(\mathcal{N}_0), [-,-]^L\big)$$ induced by the morphism of vertex algebras
		$$(-)\cap c_{\textnormal{top}}\big(\mathcal{L}\big): V_* \to V_*^L \,.$$
	\end{lemma}

	We construct the vertex algebra of topological pairs $P_* = \hat{H}_*(\mP_Y)$ and the \textit{$L$-twisted} vertex algebra on it. The homology of the space
	$$
	\mP_Y = \mC_Y\times \BU\times \ZZ
	$$
	is going to be the underlying graded vector space of $P_*$. Our reason for doing so is to avoid having to deal with the homology of the stack $\mN^{E_Y}_0$ which does not have a known explicit description. Denoting by $\mM_Y$ the higher stack of perfect complexes constructed by Toën--Vaquié \cite{TVaq} we obtain the map
	\begin{equation}
	\Omega^{E_Y}=( \Gamma\times\textnormal{id}_{BU\times \ZZ})\circ (\Sigma_{E_Y})^{\textnormal{top}}:(\mathcal{N}_0^{E_Y})^{\textnormal{top}}\to  \mathcal{M}_Y^{\textnormal{top}}\times BU\times\ZZ\to \mathcal{C}_Y \times BU\times\ZZ\,,
	\end{equation} 
	where $\Sigma_{E_Y}: \mathcal{N}_0^{E_Y}\to \mathcal{M}_Y\times \textnormal{Perf}_{\CC}$ maps $[F,W,\phi]$ to $[F,W]$, $(-)^{\textnormal{top}}$ denotes the topological realization of Remark \ref{rem:toprel}, $\Gamma: \mM_Y^{\textnormal{top}}\to \mC_Y$ is the  natural map identifying universal objects, and we used $\text{Perf}_{\CC}^{\,\text{top}} = BU\times\ZZ$ from Blanc \cite[§4.1]{Blanc}. We will formulate the construction of $P_*$ in such a way that the pushforward $\Omega^{E_Y}_*$ induces a morphism of vertex algebras $V_*\to P_*$. This will require replacing $\Theta^{E_Y,\pa}$ by its topological version, which I will denote in the definition below by $\theta^{E_Y,\pa}$.
	
	To formulate the analogue of the twisted vertex algebra of Definition \ref{Def: twistedVOA} in our topological set-up, we will need a K-theory class $\theta_L$ on $\mP_Y\times \mP_Y$ which should satisfy
	$$
	(\Omega\times\Omega)^*\theta_L = \Theta^{\pa} + \mathcal{L}^{[-,-]}\,.
	$$

	\begin{definition}
		\label{def twistedbylquot}
		Let $L\to Y$ be a line bundle. Define the data $\big(\mathcal{P}_Y, K(\mathcal{P}_Y), \Phi_{\mathcal{P}_Y}, \mu_{\mathcal{P}_Y}, 0, \theta_{L}, \tilde{\epsilon}^{L}\big)$, $\big(\mathcal{P}_Y, K(\mathcal{P}_Y), \Phi_{\mathcal{P}_Y}, \mu_{\mathcal{P}_Y}, 0, \theta^{E_Y,\pa }, \tilde{\epsilon}^{E_Y}\big)$ as follows:
		
		\begin{itemize}
			\item $K(\mathcal{P}_Y) = K^0(Y)\times \ZZ$.
			\item Set $\mathfrak{L} = \pi_{2\,*}(\pi_Y^*(L)\otimes  \mathfrak{E})\in K^0(\mathcal{C}_Y)$ and $\theta = \pi_{2,3\,*}\big( \mathfrak{E}^\vee_{1,2}\otimes \mathfrak{E}_{1,3} \big)^\vee$ and let $Y=C,S$  then on $\mathcal{P}_Y\times \mathcal{P}_Y$ we define
			$
			\theta^{E_Y, \textnormal{ob}}=(\theta)_{1,3}-\Big(\mathfrak{U}\boxtimes \pi_{2\,*}(\pi^*_1(E^*_Y)\otimes \mathfrak{E})^\vee\Big)_{2,3}
			$ and
			\begin{align*}
			&\theta^{E_Y,\pa} =  \theta^{E_Y,\textnormal{ob}} +\sigma^*(\theta^{E_Y,\textnormal{ob}})^\vee\\
			& \theta^{L,\pa} =\theta^{\pa}_{1,3}+  \Big(\mathfrak{U}\boxtimes \mathfrak{L}^\vee\Big)_{2,3}+\Big(\mathfrak{L}\boxtimes \mathfrak{U}^\vee\Big)_{1,4}\,,
			\end{align*}
			If $Y=X$, we define 
			\begin{align*}
			\theta^{E_Y,\pa }& = \theta^{E_Y,\ob} = (\theta)_{1,3}+\chi(E^*_Y\otimes E_Y)(\mathfrak{U}\boxtimes \mathfrak{U}^\vee)_{2,4}\\
			&- \big(\mathfrak{U}\boxtimes \pi_{2\,*}(\pi^*_1(E^*_Y)\cdot\mathfrak{E})^\vee\big)_{2,3} - \big(\pi_{2\,*}(\pi^*_1(E^*_Y)\cdot \mathfrak{E})\boxtimes \mathfrak{U}^{\vee}\big)_{1,4}\,.
lkl			\end{align*}
			\item The symmetric forms $\tilde{\chi}^{E_Y,\pa}: (K^0(Y)\times \ZZ)\times (K^0(Y)\times \ZZ)\to \ZZ$, $\tilde{\chi}_L: (K^0(Y)\times \ZZ)\times (K^0(Y)\times \ZZ)\to \ZZ $ are given in the case that $Y=C,S$ by 
			\hspace{-20mm}\begin{align*}
			\label{eqChi}
			\tilde{\chi}^{E_Y,\pa}\big((\alpha_1,d_1),(\alpha_2,d_2)\big)& = \chi(\alpha_1,\alpha_2)+\chi(\alpha_2,\alpha_1)-d_1\chi(E_Y, \alpha_2)-d_2\chi(E_Y,\alpha_1)\,,\\
			\tilde{\chi}_L\big((\alpha_1,d_1),(\alpha_2,d_2)\big)&=\chi(\alpha_1,\alpha_2)  -d_1\big(\chi(\alpha_2)-\chi(\alpha_2\cdot L)\big) \\
			&-d_2\big(\chi(\alpha_1) - \chi(\alpha_1\cdot L)\big)\,,
			\numberthis
			\end{align*}
			 and when $Y=X$ by
			\begin{align*}
			\label{ptop}
			\tilde{\chi}^{E_X,\pa}\big((\alpha_1,d_1),(\alpha_2,d_2)\big) &= \chi(\alpha_1,\alpha_2) + \chi(E^*_X\otimes E_X) d_1d_2 -d_2\chi\big(\alpha_1\cdot E^*_X\big)-d_1\chi\big(\alpha_2\cdot E^*_X\big) \,.
			\numberthis{}
			\end{align*}
			\item The signs are defined by 
			$$\tilde{\epsilon}_{E_Y,(\alpha_1,d_1),(\alpha_2,d_2)} =\begin{cases}(-1)^{\chi(\alpha_1,\alpha_2)+d_1\chi(E_Y,\alpha_2)}&\textnormal{if}\quad Y=C,S\\
			&
			\\
			\epsilon_{\alpha_1-d_1\llbracket E_Y \rrbracket,\alpha_2-d_2\llbracket E_Y\rrbracket}&\textnormal{if}\quad Y=X\,,
			
		\end{cases}
		$$
		
		where $\epsilon_{\alpha,\beta}$ are as before taken from Theorem \ref{Thm: orientations}
		
		The signs for the twisted vertex algebra need to be adapted by taking  $$\tilde{\epsilon}^L_{(\alpha_1,d_1),(\alpha_2,d_2)} =(-1)^{\chi(\alpha_1,\alpha_2)+d_1\big(\chi(\alpha_2)-\chi(L\cdot\alpha_2)\big)}$$ when $Y=S,C$. 
	\end{itemize}
	We denote by $P_* = (\hat{H}_*(\mathcal{P}_Y),\ket{0}, T, \mathsf{Y})$, resp. $P_*^L = (\tilde{H}_*(\mathcal{P}_Y),\ket{0}, T, \mathsf{Y}^L)$ the vertex algebras associated to this data and $(\check{H}_*(\mathcal{P}_Y),[-,-])$, resp. $(\mathring{H}_*(\mathcal{P}_Y),[-,-]^L)$ the corresponding Lie algebras. 
\end{definition}
We summarize at least for the case $E_Y =\mO_Y$ the above by noting that we have the diagram
$$
\begin{tikzcd}
\arrow[d,"\cap \,c_{\textnormal{top}}(\mathcal{L})"']\big(\check{H}_*(\mathcal{N}_0),[-,-]\big)\arrow[r,"\bar{\Omega}_*"]&\big(\check{H}_*(\mathcal{P}_X),[-,-]\big)\\
\big(\mathring{H}_*(\mathcal{N}_0),[-,-]^L\big)\arrow[r,"\bar{\Omega}_*"]&\big(\mathring{H}_*(\mathcal{P}_X),[-,-]^L\big)
\end{tikzcd}
$$
of morphisms of Lie algebras, where $\bar{\Omega}_*$ are the morphisms induced by $\Omega_*$ and we are missing the right vertical arrow.
\subsection{Wall-crossing formulae}
To relate the above constructions to the virtual fundamental classes $\big[\QuotY\big]^{\textnormal{vir}}$, we use the universal family $\pi_1^*(E_Y)\to \mathcal{F}$ on $Y\times Q_Y$ giving us
$$
\begin{tikzcd}
&\mathcal{N}^{E_Y}_0\arrow[d,"\Pi^{\pl}"]\\
Q_Y \arrow[ur,"\iota_{n,E_Y}"]\arrow[r,"\iota^{\pl}_n"]&(\mathcal{N}^{E_Y}_0)^{\pl}
\end{tikzcd}\,,
$$
where $\Pi^{\pl}$ is the projection to the rigidification.
The class $[Q_Y]_{\textnormal{vir}} \in H_*(\mathcal{N}^{E_Y}_0)$ is defined as the pushforward
$
[Q_Y]_{\textnormal{vir}}=i_{n,E_Y\,*}\big([Q_Y]^{\textnormal{vir}}\big).
$
Using the identification \eqref{Eq:rigcheck}, this leads to classes in $\check{H}_*(\mN^{E_Y}_0)$. 

In Appendix \ref{App: checking of assumptions} we give a short argument proving the next theorem. Recall that we do not write $np$ in the subscript, and this also applies in the case of $[\mathcal{M}_{np}^{\textnormal{ss}}]_{\textnormal{inv}}=[\mathcal{M}_{n}^{\textnormal{ss}}]_{\textnormal{inv}}$.
\begin{theorem}[Joyce \cite{JoyceWC}, Appendix \ref{Thm:PJS}]
	\label{thmCS}
	Let $Y=S$ or $C$. In $\check{H}_{*}\big(\mathcal{N}^{E_Y}_0\big)$  we have for any torsion-free sheaf $E_Y$ and $n\geq 0$ the formula
	\begin{equation}
	\label{EqQuEY}
	[\QuotY]_{\textnormal{vir}}=\sum_{\begin{subarray}
		a k>0,n_1,\ldots,n_k\\
		n_1+\ldots +n_k=n
		\end{subarray}}\frac{(-1)^k}{k!}\big[\big[\ldots \big[[\textnormal{Quot}_Y(E_Y,0)]_{\textnormal{vir}},[\mathcal{M}_{n_1}^{\textnormal{ss}}]_{\textnormal{inv}}\big],\ldots \big],[\mathcal{M}_{n_k}^{\textnormal{ss}}]_{\textnormal{inv}}\big]
	\end{equation}
	for the invariants  $[\mathcal{M}_{n}^{\textnormal{ss}}]_{\textnormal{inv}}\in \check{H}_0(\mathcal{N}^{E_Y}_0)$ constructed by Joyce \cite{JoyceWC}. Note that with our grading convention, we get $$[Q_Y]_{\textnormal{vir}}\in \check{H}_{\chi(\mO_Y)-2}(\mathcal{N}^{E_Y}_0)\,.$$ 
	Moreover, in $\big(\mathring{H}_*(\mathcal{N}_0), [-,-]^L\big)$, we have the formula
	\begin{equation}
	\label{wcfhilbtwistL}
	I_n(L)e^{(n,1)}=  \sum_{\begin{subarray}a
		k\geq 1,n_1,\ldots,n_k>0\\
		n_1+\cdots+n_k = n
		\end{subarray}}\frac{(-1)^k}{k!}\big[\big[\ldots \big[[\mathcal{M}_{(0,1)}]_{\textnormal{inv}},[\mathcal{M}_{n_1}]_{\textnormal{inv}}\big]^L,\ldots \big]^L,[\mathcal{M}_{n_k}\big]_{\textnormal{inv}}\big]^L\,.
	\end{equation}
\end{theorem}

The next claim requires additional work and the proof of it will appear in a more general setting in the author's \cite{bojko4} which proves the conjecture of Gross--Joyce--Tanaka \cite[§4.4]{GJT}.
\begin{theorem}[\cite{bojko4}]
	\label{conjecture quot WC}
	Let $Y=X$. In $\check{H}_{*}\big(\mathcal{N}^{E_Y}_0\big)$  we have for any locally free $E_Y$ on $Y$ which is additionally RS and $n\geq 0$ that
	$$
[\QuotY]_{\textnormal{vir}}=\sum_{\begin{subarray}
		a k>0,n_1,\ldots,n_k\\
		n_1+\ldots +n_k=n
		\end{subarray}}\frac{(-1)^k}{k!}\big[\big[\ldots \big[[\textnormal{Quot}_Y(E_Y,0)]_{\textnormal{vir}},[\mathcal{M}_{n_1}^{\textnormal{ss}}]_{\textnormal{inv}}\big],\ldots \big],[\mathcal{M}_{n_k}^{\textnormal{ss}}]_{\textnormal{inv}}\big]
	$$
	for classes $[\mathcal{M}_{n}^{\textnormal{ss}}]_{\textnormal{inv}}\in \check{H}_0(\mathcal{N}^{E_Y}_0)$ defined in \cite{bojko3} depending on fixed choices of orientations as in \cite[§3.1]{bojko2}. Note that with our grading convention, we get
	$[Q_Y]_{\textnormal{vir}}\in \check{H}_{\chi(\mO_X)-2}(\mathcal{N}^{E_Y}_0)$. 
\end{theorem}

\subsection{Explicit expression for the vertex algebra}

Now that we have defined the vertex algebras on $P_*$, we may ask for their explicit description in terms of the lattice vertex algebras introduced in \eqref{fields}. After transporting the wall-crossing formulae \eqref{EqQuEY} and \eqref{wcfhilbtwistL} to the associated Lie algebra by $\overline{\Omega}^{E_Y}_*$ induced by the push-forward $\Omega^{E_Y}_*$, we will use this description to make explicit computations. 

Setting $m=\textnormal{dim}_{\CC}(Y)$ for any complex variety $Y$, we will recall some notation now which will be used for computations:
\begin{definition}
	Let us write $(0,1)$ for the generator of $\ZZ$ in $K^0(Y)\oplus \ZZ$. Let $\textnormal{ch}: \big(K^0(Y)\oplus K^1(Y)\big)\otimes \QQ \to H^*(Y)$ be the Chern character. For each $0\leq i\leq 2m$ choose a subset 
	$$B_i\subset \big(K^0(Y)\oplus K^1(Y)\big)\otimes \QQ\,,$$
	such that $\textnormal{ch}(B_i)$ is a basis of $H^{i}(Y)$. We take $B_0 = \{\llbracket\mathcal{O}_Y\rrbracket\}$ and $B_{2m} = \{p\}$.  Then we write $B = \bigsqcup_i B_i $ and $\mathbb{B} = B\sqcup\{(0,1)\}$. Let $K_*(Y)$ denote the topological K-homology of $Y$. Let $\textnormal{ch}^\vee:K_*(Y)\otimes \QQ\to H_*(Y)$ be defined by commutativity of the following diagram:
	\begin{equation*}
	\begin{tikzcd}
	\arrow[d] K_*(Y)\otimes \QQ \otimes K^*(Y)\otimes \QQ\arrow[r,"\textnormal{ch}^\vee\otimes \textnormal{ch}"]&H_*(Y)\otimes H^*(Y)\arrow[d]\\
	\QQ\arrow[r,"\textnormal{id}"]&\QQ
	\end{tikzcd}\,
	\end{equation*}
	Then choose 
	$B^\vee\subset K_*(Y)\otimes \QQ$ such that $B^\vee$ is a dual basis of $B$, we also write $\mathbb{B}^\vee = B^\vee\sqcup \{(0,1)\}$, where $(0,1)$ is the natural generator of $\ZZ$ in $K_0(Y)\oplus K_1(Y)\oplus \ZZ$. The dual of $\sigma\in \mathbb{B}$ will be denoted by $\sigma^\vee\in \mathbb{B}^\vee$. For each $\sigma\in \mathbb{B}$, $(\alpha,d)\in K^0(Y)\times \ZZ$ and $i\geq 0$ we define 
	\begin{equation}
	\label{mudef}
	e^{(\alpha,d)}\otimes \mu_{\sigma,i} =\textnormal{ch}_i(\mathfrak{E}_{(\alpha,d)}/\sigma^\vee)\,,
	\end{equation}
	using the slant product $/: K^i(Y\times Z)\times K_j(Y)\to K^{i-j}(Z)$ and the restriction of $\mathfrak{E}$ to the connected component labeled by $(\alpha,d)$. We have a natural inclusion $\iota_{\mathcal{C},\mathcal{P}}: \mathcal{C}_Y\to \mathcal{P}_Y$: $x\mapsto (x,1,0)\in \mathcal{C}_Y\times  BU\times \ZZ$, so we identify $H_*(\mathcal{C}_Y)$ with the image of $(\iota_{\mathcal{C},\mathcal{P}})_*$, which in turn corresponds to $H_*(\mathcal{C}_Y)\boxtimes 1\subset H_*(\mathcal{C}_Y)\boxtimes H_*(BU\times \ZZ) =H_*(\mathcal{P}_Y)$.
	The universal K-theory class $\mathfrak{E}_{\mathcal{P}}$ on $(Y\sqcup *)\times (\mathcal{P}_Y)$ restricts to $\mathfrak{E}\boxtimes 1$ on $(Y\times \mathcal{C}_Y)\times BU\times \ZZ$ and   $1\boxtimes \mathfrak{U}$ on $*\times \mathcal{C}_Y\times (BU\times \ZZ)$. 
\end{definition}

\begin{proposition}[{\cite[Prop. 2.21]{bojko2}}]
	The cohomology ring $H^*(\mathcal{P}_Y)$ is generated by $$\{e^{(\alpha,d)}\otimes \mu_{\sigma,i}\}_{(\alpha,d)\in K^0(Y)\times \ZZ, \sigma\in \mathbb{B},i\geq 1}\,.$$
	Moreover, there is a natural isomorphism of rings
	\begin{equation}
	\label{explcitcoh}
	H^*(\mathcal{P}_Y)\cong \QQ[K^0(Y)\oplus\ZZ]\otimes_{\QQ}\textnormal{SSym}_{\QQ}\llbracket\mu_{\sigma,i}, \sigma\in \mathbb{B},i>0\rrbracket\,.
	\end{equation}
\end{proposition}

The dual of \eqref{explcitcoh} gives us an isomorphism 
\begin{equation}
\label{isohomol}
H_*(\mathcal{P}_Y)\cong \QQ[K^0(Y)\times \ZZ]\otimes_{\QQ}\textnormal{SSym}_{\QQ}\llbracket u_{\sigma,i}, \sigma\in \mathbb{B},i>0\rrbracket \,, 
\end{equation}
where we use the normalization which leads to 
\begin{equation}
\label{Eq: pairing}
\mu_{\sigma,i} =\frac{1}{(i-1)!}\frac{\partial}{\partial u_{\sigma,i}}
\end{equation}
with the additional requirement of supercommutativity of the variables $u_{\sigma,i}$ and the derivatives with respect to them.
When $\sigma =(v,0)$ or $\sigma = (0,1)$ we will shorten the notation to $$\mu_{\sigma,i} =\mu_{v,i}\,,\quad u_{\sigma,i} = u_{v,i}\quad \textnormal{or}\quad\mu_{\sigma,i} = \beta_{i}\,,\quad u_{\sigma,i} = b_{i}.$$ Setting $\beta_i=0$, $b_i=0$ and only considering $\QQ[K^0(Y)]\subset \QQ[K^0(Y)\oplus\ZZ]$ gives us the (co)homology of $H^*(\mathcal{C}_Y)$, $H_*(\mathcal{C}_Y)$ up to a canonical isomorphism.

We have the following obvious result which was shown also in \cite[Prop. 2.16, Prop. 2.23, Prop. 3.5]{bojko2} and later on in B.--Moreira--Lim \cite[Thm. 4.7]{BML22}). It corrects some mistakes in \cite{gross}\footnote{The cited article states that the resulting pairing $\chi^-$ is antisymmetric and notes down the state-field correspondence \eqref{fields} incorrectly for elements coming from $\Lambda_{\QQ}\big(A^-\otimes t\QQ[t^2]\big)$.}:
\begin{proposition}
	\label{propmorphismLquot}
	Let $\QQ[K^0(Y)\times \ZZ]\otimes_{\QQ}\textnormal{SSym}_{Q}\llbracket u_{\sigma,i}, \sigma\in \mathbb{B},i>0\rrbracket$ be the generalized lattice vertex algebra associated to $\big((K^0(Y)\oplus\ZZ)\oplus K^1(Y),(\tilde{\chi}_L)^\bullet\big)$, resp. $\big((K^0(Y)\oplus\ZZ)\oplus K^1(Y),(\tilde{\chi}^{E_Y,\pa})^\bullet\big)$, where $(\tilde{\chi}^{E_Y,\pa})^\bullet = \tilde{\chi}^{E_Y,\pa}\oplus \chi^-$, $(\tilde{\chi}_L)^\bullet = \tilde{\chi}_L\oplus \chi^-$  for $\tilde{\chi}_L$ from \eqref{eqChi} and 
	\begin{align*}
	\chi^-:& K^1(Y)\times K^1(Y)\to \ZZ\,,\\ \chi^-(\alpha,\beta)& =\begin{cases} \int_Y\textnormal{ch}(\alpha)^\vee\textnormal{ch}(\beta)\textnormal{Td}(Y)+\int_Y\textnormal{ch}(\beta)^\vee\textnormal{ch}(\alpha)\textnormal{Td}(Y)&\textnormal{if}\quad Y=C,S\\
	&
	\\
	\int_Y\textnormal{ch}(\alpha)^\vee\textnormal{ch}(\beta)\textnormal{Td}(Y)&\textnormal{if}\quad Y=X\
	\end{cases}\,.
	\end{align*}
	Here we use
	$$
	\Gamma^\vee = (-1)^{\floor*{ \frac{\deg \Gamma}{2}}}\Gamma\,,
	$$
	for any $\Gamma\in H^*(Y)$, to extend the dual to odd degrees.
	The isomorphism  \eqref{isohomol} induces an isomorphism of graded vertex algebras for all $E_Y$ on $Y$ and line bundles $L\to Y$ 
	\begin{align*}
	\hat{H}_*(\mathcal{P}_Y)\cong \QQ[K^0(Y)\times \ZZ]\otimes_{\QQ}\textnormal{SSym}_{Q}\llbracket u_{\sigma,i}, \sigma\in \mathbb{B},i>0\rrbracket\,,\\
	\tilde{H}_*(\mathcal{P}_Y)\cong \QQ[K^0(Y)\times \ZZ]\otimes_{\QQ}\textnormal{SSym}_{Q}\llbracket u_{\sigma,i}, \sigma\in \mathbb{B},i>0\rrbracket\,,
	\end{align*}
	where the degrees are given appropriately (see e.g. Proposition \cite[Prop. 3.8, Thm. 4.5]{bojko2}).
	\sloppy The map $(\Omega^{E_Y})_*: H_*(\mathcal{N}^{E_Y}_0)\to H_*(\mathcal{P}_Y)$ induces morphisms of graded vertex algebras $(\hat{H}_*(\mathcal{N}^{E_Y}_0),\ket{0},T, \mathsf{Y})\to (\hat{H}_*(\mathcal{P}_Y),\ket{0},T, \mathsf{Y})$, $(\tilde{H}_*(\mathcal{N}^{E_Y}_0),\ket{0},T, \mathsf{Y}^L)\to (\tilde{H}_*(\mathcal{P}_Y),\ket{0},T, \mathsf{Y}^L)$ and of graded Lie algebras
	\begin{align*}
	\bar{\Omega}^{E_Y}_* &:\big(\check{H}_*(\mathcal{N}^{E_Y}_0),[-,-]\big)\longrightarrow \big(\check{H}_*(\mathcal{P}_Y),[-,-]\big)\,,\\
	\bar{\Omega}_* &:\big(\mathring{H}_*(\mathcal{N}_0),[-,-]^L\big)\longrightarrow \big(\mathring{H}_*(\mathcal{P}_Y),[-,-]^L\big)\,.
	\end{align*}
\end{proposition}
\section{Computing invariants}
We use the notation and results of §\ref{SecVACQ} to do explicit computations. The first two subsections are focused on computing the general invariants corresponding to the topological data of $[\QuotY]_{\vir}$. Afterwards, we integrate multiplicative genera of classes of the form $\alpha^{[n]}$ and $T^\vir$ to obtain closed formulae in this general setting.

\subsection{Computing the classes counting zero-dimensional sheaves}

We will be using the notation
$$
\mathscr{Q}_{E_Y,n}=\bar{\Omega}^{E_Y}_*\Big(\big[\QuotY\big]_{\textnormal{vir}}\Big)\,,\quad\textnormal{and}\quad  \mathscr{M}_{n}= \bar{\Omega}^{E_Y}_*\big([\mathcal{M}_{n}]_{\textnormal{inv}}\big)
$$
for our main classes of interest.  For $Y=S$, after we apply wall-crossing in the Lie algebra $(\mathring{H}_*(\mathcal{P}_Y),[-,-]^L)$ to recover $\mathscr{M}_{n}$ from $I_Y(L;q)$, we will compute a closed formula for their generating series. We then compute the full generating series $$
\mathscr{Q}_{E_Y}(q) = \sum_{n\geq 0}\mathscr{Q}_{E_Y,n}q^n
$$
by wall-crossing back inside of 
$
(\check{H}_*(\mathcal{P}_Y),[-,-])\,.
$
Most importantly, we will be doing all of our computations under the assumption that $b_1(S) = 0$, but we explain how to remove this restriction in Remark \ref{remark b1}.

For any $Y=C,S$, we will always choose the basis $B^\vee_i\in H^i(Y)$, such that $B_{i}=B_{\textnormal{2dim}(Y) -i}^{\vee}$ and $B^\vee$ is dual to $B$ with respect to the pairing  
$$
\int_Yv\cdot w \,,\qquad v\in H^i(Y),w\in H^{2\dim(Y)-i}(Y)\,. 
$$
Thus we have elements $v^\vee \in B^\vee_i$ such that $\int_Y w\cdot v^\vee =\delta_{v,w},$ whenever  $w\in B_i.$
For any class $h\in H^*(Y)$ we  will also write $h_v=h\cdot v^\vee$.

The next definition is meant to address the freedom of the choice of trivialization $o_{\alpha}: \ZZ_2\to O_{\alpha}$ for any zero-dimensional sheaf $F$ as explained in Definition \ref{Def: orientations}.
They are similar to the point-canonical ones  we used in \cite[Def. 3.2]{bojko2} with the exception that the structure sheaf is now replaced by $E_X$.

\begin{definition}
	\label{DefEpoint}
	As in Joyce--Tanaka--Upmeier \cite[Thm. 2.27]{JTU} or \cite[Thm. 5.4]{bojko}, choose the order of generators of $\{N\llbracket E_X\rrbracket + np\}_{n,N\in \ZZ}\subset K^0(X)$ such that
	$
	\llbracket E_X\rrbracket< p.
	$
	We now fix the orientations $o_{\llbracket E_X\rrbracket}, o_p$ such that  $o_p$ induces the Oh-Thomas/ Borisov--Joyce virtual fundamental class $[M_p]^{\textnormal{vir}} =\textnormal{Pd}(c_3(X))$ and $o_{\llbracket E_Y\rrbracket}$ induces the virtual fundamental class $[\{E_X\}]^{\textnormal{vir}}=1\in H_0(\textnormal{pt})$ by Definition \ref{Def: orientations}. We will denote these choices of orientations $o^{\textnormal{can}}_p$ and $o^{\textnormal{can}}_{\llbracket E_X\rrbracket}$ respectively. Using the construction in \cite[Thm. 2.27]{JTU} or \cite[Thm. 5.4]{bojko}, these determine canonical orientations for all $\alpha=N\llbracket E_X\rrbracket + np$. We call such orientation $E_X$-\textit{point-canonical}. Explicitly, these orientations can be described by setting 
	$$
	\epsilon_{N\llbracket E_X\rrbracket+kp,M\llbracket E_X\rrbracket +lp} = (-1)^{Mek}
	$$
	and applying  \eqref{Eq: signcomparison} to define the orientations $o_{N\llbracket E\rrbracket+ lp}$ for all $N,l\in\ZZ$.
\end{definition}

We now restrict the number of parameters needed to describe the classes $\msM_{n}$. Define the subgroups of K-theory $N_{\leq k}(Y)´\subset K^0(Y)$ to be maximal such that
$$
\textnormal{ch}\big(N_{\leq k}(Y)\big) \subset \bigoplus_{i\geq 2m -2k}H^i(Y)\,,
$$
where we continues using $m=\dim_{\CC}(Y)$.
The following result generalizes \cite[Lem. 3.6]{bojko2}.
\begin{lemma}
	\label{lemDeg}
	Let $\alpha\in N_{\leq k}$ and $\mathcal{M}_{\alpha}$ a moduli stack of sheaves on $Y$ with class $\alpha$. Consider the map $i_{\alpha}:(\mathcal{M}_{\alpha})^{\textnormal{top}} \to \mathcal{C}_{\alpha}$ induced by the universal sheaf $\mathcal{F}_{\alpha}\to Y\times \mathcal{M}_{\alpha}$, then $i_{\alpha\,*}\big(H_2(\mathcal{M}_{\alpha})\big)\subset H_2(\mathcal{C}_{\alpha})$ is contained in
	$$
	\bigoplus_{i\geq d}H^{2i}(Y)\oplus \Lambda^2\big(\bigoplus_{i\geq d}H^{2i+1}\big)\,.
	$$
	where $d = m - (k+1)$
	and we used the isomorphism 
	\begin{equation}
	¨\label{eqIso}
	H_2(\mathcal{C}_{\alpha}) = H^{\textnormal{even}}(Y)\oplus \Lambda^2H^{\textnormal{odd}}(Y)\,.  
	\end{equation}
\end{lemma}
\begin{proof}
	Recall from \cite[Lem. 3.6 ]{bojko2}, that we have the
	\begin{equation}
	\label{pullchernofE}
	\textnormal{ch}(\mathfrak{E}_{\alpha})=\sum_{\begin{subarray}a v\in B\\
		i\geq 0\end{subarray}}v\boxtimes e^{\alpha}\otimes\mu_{v,i}\,,
	\end{equation}
	where $\mathfrak{E}$ is the universal K-theory class on $Y\times \mathcal{C}_Y$. 
	Let $\textnormal{deg}(v) = j$, then
	$$
	\textnormal{deg}\Big(v\boxtimes i_{\alpha}^*\big(e^{\alpha}\otimes \mu_{v,1}\big)\Big) = \begin{cases}
	j+1&\textnormal{if}\quad j\,\textnormal{ is odd}\\
	j+2&\textnormal{if}\quad j\,\textnormal{ is even}
	\end{cases}\,.
	$$
	Therefore $\big(\textnormal{id}_Y\times i_{\alpha}\big)^*\Big(v\boxtimes (e^{\alpha}\otimes \mu_{v,1}) \Big)= 0$ whenever 
	$$
	2m - 2(k+1)\geq \begin{cases}j+1&\textnormal{if}\quad j\,\textnormal{ is odd}\\
	j+2&\textnormal{if}\quad j\,\textnormal{ is even}
	\end{cases}
	$$
	by dimension arguments.
	In particular, we conclude that $i_{\alpha}
	^*(e^\alpha\otimes \mu_{v,1}) = 0$, whenever $v\in B_{<2d}$. Thus the result follows by using the isomorphism \eqref{eqIso}. 
\end{proof}

From now on we fix the $E_X$-point-canonical orientations for $Y=X$.  When $\Gamma\in H^*(Y)$ is a cohomology class, we will denote its coefficients with respect to a fixed basis $B\subset H^*(Y)$ by 
$$
\Gamma = \sum_{v\in B}\Gamma_v v\,.
$$
\begin{lemma}
	\label{lemNnp}
	If Conjecture \ref{conjecture quot WC} holds, then 
	$$
	\mathscr{M}_{n} = e^{(n,0)}\otimes 1\cdot \mathscr{N}_{n} + \QQ T(e^{(n,0)}\otimes 1)\in H_2(\mC_Y)/\big(TH_0(\mC_Y)\big)\,,
	$$
	where for the series $\mathscr{N}(q) = \sum_{n>0}\mathscr{N}_{n}q^n$ we have
	$$
	\textnormal{exp}\big(\mathscr{N}(q)\big) = \begin{cases}
	\Big[1-e^pq\Big]^{\Big(-\sum_{v\in B_2}c_{1,v}u_{v,1}\Big)}&\textnormal{if}\quad  Y=S\,,\\
	\\
	M(e^pq)^{\Big(\sum_{v\in B_6}c_3(X)_vu_{v,1}\Big)}&\textnormal{if}\quad Y=X\,.
	\end{cases}
	$$
	
\end{lemma}
\begin{proof}
	We begin by noting that from \cite[Lem. 2.2]{bojko2}, we recover the action of the translation operator
	$$
	T(e^{np,0}) = e^{(n,0)}\otimes nu_{p,1}
	$$
	which implies together with Lemma \ref{lemDeg} that (assuming $b_1(Y) = 0$)
	$$
	\mathscr{N}_{n} = \begin{cases}
	\sum_{v\in B_2}a_v(n)u_{v,1} + \QQ u_{p,1}&\text{if}\qquad Y=S\,,\\
	\sum_{v\in B_6}a_v(n)u_{v,1} + \QQ u_{p,1}&\text{if}\qquad Y=X\,.
	\end{cases}
	$$

	To compute the coefficients, we start from \eqref{wcfhilbtwistL}
	which we may push forward to $\big(\mathring{H}(\mP_Y), [-,-]^L\big)$ where it becomes 
	\begin{equation}
	\label{wcfhilb2twistL}
	I_{Y,n}(L)e^{(np,1)}\otimes 1=  \sum_{\begin{subarray}a
		k\geq 1,n_1,\ldots,n_k>0\\
		n_1+\cdots+n_k = n
		\end{subarray}}\frac{(-1)^k}{k!}\big[\big[\ldots \big[e^{(0,1)}\otimes 1,\mathscr{M}_{n_1}\big]^L,\ldots \big]^L,\mathscr{M}_{n_k}\big]^L\,.
	\end{equation}
	We now only focus on the case $Y=S$, for which each iterated bracket reduces to computing 
	$$
	[e^{(mp,1)}\otimes 1, e^{(n,0)}\otimes \mathscr{N}_{n}]^L =-e^{(m+n)p,1}\otimes \sum_{v\in B_2}\int_Sc_1(L)\textnormal{ch}(v)a_{v}(n)\,.
	$$
	This is obtained from \eqref{fields} and taking residues. It is useful to rewrite the term on the left-hand side as
	$$
	[q^n]\Big\{\exp\Big(\sum_{k>0}\frac{1}{k}q^k\Big)\Big\}\,.
	$$
	using \eqref{eqrk1} combined with \eqref{eqABs}. By the same induction argument as in \cite[Thm. 4.9]{bojko2} but without having to consider the sign compatibility, we obtain $a_v(n) = -\frac{1}{n}c_{1,v}$ thus
	$$\mathscr{N}(n) =-
	\frac{1}{n}\sum_{v\in B_2}c_{1,v}u_{v,1}\,.\footnote{Note that the result of \eqref{eqrk1} is originally stated for algebraic line bundles. For our purposes, we need the result for a basis of $H^2(S)$. This can be obtained, as the arguments of Oprea--Pandharipande \cite{OP1} are purely topological.}$$
	The case $Y=X$ has been proved in \cite[Thm. 3.10]{bojko2}. 
\end{proof}

\subsection{Computing the VFCs of quot-schemes}
We can now compute the generating series $\mathscr{Q}_{E_Y}(q) = 1+\sum_{n>0}q^n\frac{\mathscr{Q}_{E_Y,n}}{e^{(np,1)}}$. In the following, we use the notation 
$$[z^n]\big\{f(z)\big\}=f_n\quad \textnormal{where}\quad  f(z) = \sum_{n\geq 0}f_nz^n\,.$$
For the class of a point $p$, we will also always use $y_k$ to denote  $u_{p,k}$.
\begin{theorem}
	\label{thmMain}
	Let $Y=S$ and
	\begin{align*}
	\mathscr{H}_{e,n}&=[z^{ne}]\bigg\{\sum_{v\in B_2}\frac{c_{1,v}}{n}\sum_{k>0}u_{v,k}z^{k}\textnormal{exp}\Big[\sum_{k>0}\frac{ny_k}{k}z^k\Big]\bigg\}\,,\\
	L_{e,n}&=[z^{ne}]\bigg\{\textnormal{exp}\Big[\sum_{k>0}\frac{ny_k}{k}z^k\Big]\bigg\}\,,\\
	Q_{e, n,j}&=\frac{1}{n}[z^{ne+j}]\bigg\{\textnormal{exp}\Big[\sum_{k>0}\frac{ny_k}{k}z^k\Big]\bigg\}\,.
	\end{align*}
	Then
	\begin{equation}
	\label{eqQug}
	\mathscr{Q}_{E_S}(q) = \textnormal{exp}\bigg[\sum_{n>0}\Big(\mathscr{H}_{e,n}+\Big(\frac{ec_1^2}{2n}-\frac{c_1(E_S)\cdot c_1}{n}\Big)L_{e,n}\Big)q^n-c_1^2\sum_{\begin{subarray}a n,m>0\\
		j>0\end{subarray}}jQ_{e,n,-j}Q_{e,m,j}q^{n+m}\bigg] \,.
	\end{equation}
	Let $Y=X$ and 
	$$
	\mathscr{H}_{e,n}=[z^{ne}]\bigg\{\sum_{v\in B_6}\sum_{l|n}\frac{n}{l^2}c_{3,v}\sum_{k>0}u_{\llbracket\mathcal{O}_C\rrbracket,k}z^{k}\textnormal{exp}\Big[\sum_{k>0}\frac{ny_k}{k}z^k\Big]\bigg\}\,,
	$$
	then
	\begin{equation}
	\label{eqQugX}
	\mathscr{Q}_{E_X}(q) = \textnormal{exp}\Big[\sum_{n>0}(-1)^{ne}\Big(\mathscr{H}_{e,n}+c_1(E_X)\cdot c_3\sum_{l|n}\frac{n}{l^2}L_{e,n}\Big)q^n\Big] \,.
	\end{equation}

\end{theorem}
\begin{proof}
	Using Proposition \ref{propmorphismLquot}, we may pushforward the identity \eqref{EqQuEY} to the Lie algebra $\check{H}(\mP_Y)$ where it becomes 
	\begin{equation}
	\mathscr{Q}_{E_Y,n}=\sum_{\begin{subarray}
		a k>0,n_1,\ldots,n_k\\
		n_1+\ldots +n_k=n
		\end{subarray}}\frac{(-1)^k}{k!}\big[\big[\ldots\big[e^{(0,1)},\mathscr{M}_{n_1}\big],\ldots \big],\mathscr{M}_{n_k}\big]\,.
	\end{equation}
	Reordering it, we may write 
	$$
	\msQ_{E_Y,n} =  \sum_{\begin{subarray}a
		k\geq 1,n_1,\ldots,n_k\\
		\sum n_i = n
		\end{subarray}}\frac{1}{k!}\big[e^{(n_1p,0)}\otimes \mN(n_1) ,\big[\ldots,\big[e^{(n_kp,0)}\otimes \mN(n_k),e^{(0,1)}\big]\ldots\big]\big]\,.
	$$
	For now, let us restrict to the case $Y=S$ and only sketch the computation here. We will use multiple times that 
	\begin{align*}
	\tilde{\chi}^{E_S,\pa}(v,w) &= \chi(v,w)+\chi(w,v) = -2v\cdot w\,,\quad v,w\in H^2(S)\,,\\
	\tilde{\chi}^{E_S,\pa}\big((n,0),(0,1) \big)&=-en\,,\\
	\tilde{\chi}^{E_s,\pa}\big((v,0),(m,1)\big)&=-\chi(E_S,v)=-e \frac{c_1\cdot v}{2}+c_1(E_S)\cdot v\,,\quad v\in H^2(S)\,.
	\end{align*}
	
	Starting from a polynomial $K=K(u_{v,j},y_k)$ for $v\in B_2$ and $j,k\geq  1$ we act on it with the field $\msY(\msM_{n},z)$. Using \eqref{fields}, this amounts to
	\begin{align*}
	&\mathsf{Y}\big(e^{(np,0)}\otimes \sum_{v\in B_2}\frac{c_{1,v}}{n}u_{v,1},z\big)e^{(mp,1)}\otimes K=-\sum_{v\in B_2}\frac{c_{1,v}}{n}:\mathsf{Y}(e^0\otimes u_{v,1},z)\mathsf{Y}(e^{(np,0)}\otimes 1,z):e^{(mp,0)}\otimes K \\
	&= -\sum_{v\in B_2}\frac{c_{1,v}}{n}z^{-ne}e^{\big(1,(n+m)p\big)}\otimes \bigg\{\sum_{k>0}u_{v,k}z^{k-1}\textnormal{exp}\Big[\sum_{k>0}\frac{ny_k}{k}z^k\Big]-\textnormal{exp}\Big[\sum_{k>0}\frac{ny_k}{k}z^k\Big]\\
	&\sum_{k>0}\sum_{w\in B_2}k2v\cdot w \frac{d}{du_{w,k}}z^{-k-1}\bigg\}K-\Big(\frac{ec_1^2}{2n}-\frac{c_1(E_S)\cdot c_1}{n}\Big)z^{-1}K
	\end{align*}
	Taking residues, we will use the next set of identities
	\begin{align*}
	\frac{d}{du_{w,k}}\mathscr{H}_{e,n_1}=c_{1,w}Q_{e,n_1,-k}  \,,\qquad \sum_{v\in B_2}c_{1,v}c_{1,v^\vee}=c_{1}^2\,,\qquad \sum_{v\in B_2}c_1(E_S)_{v^\vee}c_{1,v}=c_1(E_S)\cdot c_1(S)
	\end{align*}
	to express the iterated Lie bracket.
	By induction and using that the number of ways to form pairs out of $2k$ elements is $\frac{2k!}{2^k}$, a simple combinatorial argument shows that for any partition of the form
	$$n=n_1+n_2+\ldots+n_l+n_{l+1}+\ldots+n_{2k+l}$$ the coefficient of $$\prod_{i=1}^l\Big(\mathscr{H}_{n_i}-\big(\frac{ec_1^2}{2n_i}-c_1(E_S)\cdot c_1\big)L_{n_i}\Big)\sum_{j>0}j(-c_1^2)Q_{n_{l+1},j}Q_{n_{l+2},j}\ldots \sum_{j'>0}j'(-c_1^2)Q_{n_{l+2k-1},j'}Q_{n_{l+2k},j'}$$ in 
	$$
	\sum_{n_1,\ldots,n_k}\frac{1}{k!} \big[\mathscr{M}_{n_k}, \ldots\big[\mathscr{M}_{n_1},e^{(0,1)}\otimes 1\big]\ldots \big]\big]
	$$
	is
	\begin{align*}
	&\frac{2^k}{(2k+l)!}{2k+l\choose l}\frac{2k!}{k!2^k}C\Big(\begin{subarray}
	a n_1,\ldots, n_l\\
	n_{l+1}, \ldots n_{l+2k-1}\\
	n_{l+2}, \ldots n_{l+2k}
	\end{subarray}\Big)\\
	&=\frac{1}{l!k!}C\Big(\begin{subarray}
	a n_1,\ldots, n_l\\
	n_{l+1}, \ldots n_{l+2k-1}\\
	n_{l+2}, \ldots n_{l+2k}
	\end{subarray}\Big)\,,
	\end{align*}
	where
	$
	C\Big(\begin{subarray}
	a n_1,\ldots, n_l\\
	n_{l+1}, \ldots n_{l+2k-1}\\
	n_{l+2}, \ldots n_{l+2k}
	\end{subarray}\Big)\
	$
	are some additional combinatorial coefficients that we avoid specifying explicitly except that the right-hand side is precisely the corresponding coefficient of the same term in \eqref{eqQug}. \\

	Finally, the computation for $Y=X$ is the simplest and is closer to the one in \cite[Prop. 5.1]{bojko2}. 
	We again rephrase the necessary Euler pairings in cohomological terms
	\begin{align*}
	\label{EqCYpa}
	\tilde{\chi}\big((n,0),\sigma\big)={\begin{cases}ne&\textnormal{if }\sigma = (\llbracket \mathcal{O}_X\rrbracket,0)\\
		-ne&\textnormal{if }\sigma = (0,1)\\
		0&\textnormal{otherwise}
		\end{cases}}\quad
	\tilde{\chi}\big((v,0),(m,1)\big) = {\begin{cases}
		c_{1}(E_X)\cdot v&\textnormal{if }v\in B_6\\
		-em&\textnormal{if }v=p
		\end{cases}}
	\numberthis
	\end{align*}
	This time, we need to be more careful with signs by recalling from \ref{DefEpoint} that 
	$$
	\epsilon_{N\llbracket E_X\rrbracket+kp,M\llbracket E_X\rrbracket +lp} = (-1)^{Mek}\,.
	$$ 
	So acting with the field $\mathsf{Y}(\mM_{n},z)$ on  $K=K(u_{w,k})$ for $w\in B_{6,8}$. we recover
	\begin{align*}
	&Y\big(e^{(np,0)}\otimes \sum_{v\in B_2}\frac{n}{l^2}c_{3,v}u_{v,1},z\big)e^{(mp,1)}\otimes K\\
	=&\sum_{v\in B_2}\sum_{l|n}\frac{n}{l^2}c_{3,v}:\mathsf{Y}(e^0\otimes u_{v,1},z)\mathsf{Y}(e^{np}\otimes 1,z):e^{(mp,0)}\otimes K \\
	=& (-1)^{ne}\sum_{v\in B_2}\sum_{l|n}\frac{n}{l^2}c_{3,v}z^{-ne}e^{\big(1,(n+m)p\big)}\otimes \bigg\{\sum_{k>0}u_{v,k}z^{k-1}\textnormal{exp}\Big[\sum_{k>0}\frac{ny_k}{k}z^k\Big]\\
	&+c_1(E_X)\cdot c_3\,\,z^{-1}\textnormal{exp}\Big[\sum_{k>0}\frac{ny_k}{k}z^k\Big]\bigg\}K\,.
	\end{align*}
	
	Taking residues, the result follows by an easy computation.
\end{proof}

\begin{remark}
	\label{remark b1}
	Going through the above computation without the assumption $b_1(S)=0$ one can check that under the projection $\Pi_{\textnormal{even}}:\check{H}_*(\mathcal{P}_S)\to \check{H}_{\textnormal{even}}(\mathcal{P}_S)$ we still obtain the same results. This is sufficient for us because we never integrate odd cohomology classes, except when integrating polynomials in $\textnormal{ch}_k(T^{vir})$, but as the only terms $\mu_{v,k}$ contained in $\textnormal{ch}_k(T^{\vir})$ (see \eqref{eqThOb}) for $v\in B_{\textnormal{odd}}$ are given for $v\in B_3$, each such integral will contain a factor of $\chi^-(v,w)=0$ for $v,w\in B_3$.
	
	The condition $b_1(C)=0$ is not important to us, as we only use geometric arguments to compute integrals over $[\QuotC]$ when $C\subset S$ is a smooth canonical curve. 
	
	Finally, we can drop the condition $b_1(X)=0$, because this would only introduce terms in $u_{v,k}$, where $v\in B_7$ and then again $\chi(v,w)=0$ when $w\in B_7$ also thus the only possible non-zero term coming from \eqref{eqChk} vanishes.
\end{remark}
\begin{remark}
	It is not reasonable to apply the above to curves. Firstly, the computation itself becomes overly complicated and secondly one can use a different approach because we are only interested in Segre and twisted Verlinde numbers. These can be addressed by only working with
	$$
	\textnormal{Quot}_{\mathbb{P}^1}\big(\bigoplus_i\mO(d_i),n\big)\,,
	$$
	where one can apply the classical localization of Atiyah--Bott and apply dimensional reduction from a blow-up of a K3 surface in a point. We do this in §\ref{secTwF}
\end{remark}

\subsection{Tautological integrals}
While in the previous section, we studied the data of the virtual fundamental class corresponding to all descendant integrals, we now focus on more revealing expressions which satisfy compelling symmetries. For this, we first address integrals  of multiplicative genera of tautological classes. The computation for $Y=S$ is the biggest challenge, so we focus on it first.

The corresponding topological analogues of 
$$
\alpha^{[n]}\,, \llbracket T^\vir\rrbracket ^\vee \in K^0\big(\QuotS\big)\qquad \text{are}\qquad \mathfrak{T}(\alpha) =   \pi_{2\,*}\big(\pi_S^*(\alpha)\otimes  \mathfrak{E}\big)\boxtimes \mathfrak{U}^\vee\,, -\Delta^*\theta^{E_S,\textnormal{ob}}\in  K^0(\mathcal{P}_S)\,,
$$
where $\theta^{E_S,\ob}$ was described in Definition \ref{def twistedbylquot}. They satisfy
\begin{equation}
\label{Eq: alphaTvirpullback}
(i_{E_S,n})^*\Big(\Omega^{E_S}\Big)^*\mathfrak{T}(\alpha) = \alpha^{[n]}\,,\qquad -(i_{E_S,n})^*\Big(\Omega^{E_S}\Big)^*\Delta^*\theta^{E_S,\text{ob}} = \llbracket T^\vir\rrbracket ^\vee\,.
\end{equation}

In \cite[Lem. 5.7]{bojko2}, we expressed the Chern characters of $\alpha^{[n]}$ by using Atiyah--Hirzebruch--Riemann--Roch \cite{dold} of Dold:
$$
\textnormal{ch}_k\big(\mathfrak{T}(\alpha) \big) = \sum_{\begin{subarray}a v\in B_{\textnormal{even}}\\
	i+j=k
	\end{subarray}}(-1)^i\chi(\alpha^\vee,v)\beta_i\cdot \mu_{v,j}
$$
We use it to compute also Chern characters of $\theta^{E_S,\text{ob}}$, where we use $b_1(S) = 0$ to avoid odd degrees based on Remark \ref{remark b1}:
\begin{align*}
\label{eqThOb}
\textnormal{ch}_k(\theta^{E_S,\textnormal{ob}}) &= 
1\boxtimes\Big(\int_S \ch(\mathfrak{E})^\vee\cdot\ch(\mathfrak{E})\cdot \td(S)\Big)^\vee - \Big(\int_S\ch(E_S)^\vee \cdot\ch(\mathfrak{U})^\vee\boxtimes \ch(\mathfrak{E})\cdot \td(S)\Big)^\vee
\\&=\sum_{\begin{subarray}a v,w\in B\\
	i+j=k
	\end{subarray}}\int_S v^\vee\cdot w\cdot \td(S)(-1)^j\mu_{v,i}\mu_{w,j} - \sum_{\begin{subarray}a v\in B\\
	i+j=k
	\end{subarray}}\int_S \ch(E_S)^\vee \cdot v\cdot \td(S) (-1)^j\beta_i\mu_{v,j}
\\&=\sum_{\begin{subarray}a i+j=k
	\\
	v\in B \end{subarray}}(-1)^j\chi(v,w)\mu_{v,i}\boxtimes \mu_{w,j} -\sum_{\begin{subarray}a v\in B\\
	i+j=k
	\end{subarray}}d(-1)^j\beta_i\boxtimes\chi(E_S,v)\mu_{v,j}\,.
\numberthis
\end{align*}

When $Y=X$, we instead get 
$$
\llbracket T^{\textnormal{vir}}\rrbracket^\vee =-(\omega^{E_X}\circ i_{E_X,n})^*\big(\theta^{E_X,\text{ob}}\big)-\chi(\mO_X)\,,
$$
We also have (after neglecting $B_{\textnormal{odd}}$ contributions which we are allowed to do by Remark \ref{remark b1})
\begin{equation}
\label{eqChk}
\textnormal{ch}_k\big(\theta^{E_X,\ob}\big) = \sum_{\begin{subarray}a i+j=k\\
	\sigma,\tau\in \BB
	\end{subarray}}(-1)^j\tilde{\chi}^{E_X}(\sigma,\tau)\mu_{\sigma,i}\boxtimes \mu_{\tau,j}\,.
\end{equation}
Acting with \eqref{eqChk} on $\eqref{eqQugX}$ simplifies the expression:
\begin{align*}
\ch_k\big(\theta^{E_X,\ob}\big)\cap \msQ_{E_Y,X}  = \bigg(c_1(E_X)_{v^\vee}\Big(1+(-1)^k\Big)\mu_{v,k} -en \Big(1+(-1)^k\Big)\mu_{p,k}\bigg) \cap \mathscr{Q}_{E_X,n}
\end{align*}
by a similar computation as in \cite[Lem. 5.71]{bojko2}, which boils down to noting that only $\mu_{v,k}$ with $v\in B_{6,8}$ can contribute and then expressing the Euler pairings as in \eqref{EqCYpa}

Let $g,f$ be multiplicative genera as in \cite[§5.3]{bojko2}. By the definition of $\msQ_{E_Y,n}$, we have 
\begin{align*}
\label{eqFaGT}
\int_{[\Quot]^{\textnormal{vir}}} f(\alpha^{[n]})g(T^{\textnormal{vir}}) &= \int_{\mathscr{Q}_{E_S,n}} \textnormal{exp}\Big[\sum_{\begin{subarray}a k\geq 0\\
	v\in B_{6,8}\end{subarray}}f_k\big(c_1(\alpha)\cdot v+v\cdot \frac{c_1}{2}a\big)\mu_{v,k}z^k+\sum_{k>0}a\mu_{p,k}f_kz^k\\
&+\sum_{k>0}\Big(-\sum_{\begin{subarray}a i+j=k\\
	v,w\in B_2\end{subarray}}(-1)^iv\cdot w \mu_{v,i}\mu_{w,j}+\sum_{v\in B_2}\Big(e\frac{c_1\cdot v}{2}-c_1(E_X)\cdot v\Big)\mu_{v,k}g_kz^k\\
&+e\sum_{k>0}\mu_{p,k}g_kz^k
\Big]\\
\int_{[\QuotX]^{\vir}}f(\alpha^{[n]})g(T^{\vir})=&\int_{\mathscr{Q}_{E_X,n}}\exp\Big[\sum_{v\in B_6}\Big(c_1(\alpha)\cdot v f_k-c_1(E_X)\cdot v(g_k+(-1)^kg_k)\Big)\mu_{v,k}\\+&\sum_{k>0}\Big(af_k -e(g_k+(-1)^kg_k)\Big)\mu_{p,k} \Big]
\\ \textnormal{where}\quad 
&\sum_{k\geq 0}f_k\frac{z^k}{k!} = \textnormal{log}\big(f(z)\big)\,,\qquad \sum_{k\geq 0}g_k\frac{z^k}{k!} = \textnormal{log}\big(g(z)\big) \,.
\numberthis
\end{align*}
To obtain the above, we also used for $Y=S$ $$\chi(\alpha^\vee,v)=\begin{cases}
c_1(\alpha)\cdot v + v\cdot \frac{c_1}{2}a&v\in B_2\\
a=\textnormal{rk}(\alpha)&v=p
\end{cases}
$$
and
$$
\chi(E_S,v)=\begin{cases}-
c_1(E_S)\cdot v + v\cdot \frac{c_1}{2}e&v\in B_2\\
e&v=p
\end{cases}\,.
$$
together with $\mu_{v,0} = 0$ whenever $v\in B_2$ which follows from \cite[Lem. 3.9. (i)]{bojko2}.\\

The next lemma summarizes some identities regarding the action of exponentials of differential operators which will be needed in evaluating \eqref{eqFaGT} while using the formula for the cap product \eqref{Eq: pairing}.
\begin{lemma}
	\label{lemma SecDer}
	Let $c_i,c_{i,j},a_i\in R[q]$,  for $i,j\in I$, where $I$ is some indexing set, then the following identities hold whenever both sides are well-defined:
	\begin{align}
	&\exp\Big[\sum_{i\in I}c_i\frac{\partial^2}{(\partial x_i)^2}\Big]\textnormal{exp}\Big[\sum_{i\in I}a_ix_i\Big]\Big|_{x_i=0} =\exp\Big[\sum_{i\in I}c_ia_i^2\Big]\label{eqDoDe1}\\
	&\exp\Big[\sum_{i\neq j\in I}c_{i,j}\frac{\partial}{\partial x_i}\frac{\partial}{\partial x_j}\Big]\exp\Big[\sum_{i\in I}a_ix_i\Big]\Big|_{x_i=0}=\exp\Big[\sum_{i\neq j\in I}c_{i,j}a_ia_j\Big]\label{eqDoDe2}\\
	&\exp\Big[\sum_{i,j\in I}c_{i,j}\frac{\partial}{\partial x_i}\frac{\partial}{\partial x_j} + \sum_{i\in I}c_i\frac{\partial}{\partial x_i}\Big]\exp\Big[\sum_{i\in I}a_kx_k\Big]\Big|_{x_i=0} = \exp\Big[\sum_{i,j\in I}c_{i,j}a_ia_j+\sum_{i\in I}c_ia_i\Big]\,
	\label{eqCoDe}
	\end{align}
\end{lemma}
\begin{proof}
	For the first identity, we can use that 
	\begin{align*}
	\exp\Big[c\frac{\partial^2}{(\partial x)^2 }\Big]\exp\Big[ax\Big]\Big|_{x=0} &=\bigg[\sum_{n\geq 0}\frac{c^n}{n!}\frac{\partial^{2n}}{(\partial x)^{2n}}\bigg]\bigg[\sum_{k\geq 0}\frac{(ax)^k}{k!}\bigg]\Big|_{x=0} \\
	&=\sum_{n\geq 0}\frac{c^na^{2n}}{n!}=\exp\big[ca^2\big]\label{eqCoDe}\,.
	\end{align*}
	Using the standard result 
	\begin{equation}
	\label{eqTra}
	\textnormal{exp}[a\partial/\partial x]f(x) = f(x+a)\,,\end{equation} we obtain:
	$$
	\exp\Big[c\frac{\partial}{\partial x}\frac{\partial}{\partial y}\Big]\exp\Big[ax +by \Big]=\exp\big[ax+by\big]\exp\big[abc\big]\,,
	$$
	by successive application of \eqref{eqTra}. Finally, \eqref{eqCoDe} follows by the same arguments. 
\end{proof}

Recall from \eqref{Eq: pairing} that for any two power-series $F$ and $G$, we have
\begin{equation}
\label{eqCap}
G(\mu_{v,k})\cap F(u_{w,i}) = G\Big(\frac{1}{(k-1)!}\frac{\partial}{\partial u_{v,k}}\Big) F(u_{w,i})\,.
\end{equation}
 Moreover, by evaluating at $u_{w,i}=0$, we obtain the action of integrating cohomology classes $G(\mu_{v,k})$.  Applying this together with Lemma \ref{lemma SecDer}  and \eqref{eqFaGT} we have the following result for $Y=S$:
\begin{align*}
\label{eqQug}
&Z_{E_S}(f,g,\alpha;q)\\
&=
\exp\bigg[c_1^2\sum_{k>0}\frac{g_{k}}{k!}\Big(\sum_{\begin{subarray}a i+j=k\\i,j>0
	\end{subarray}}(-1)^i\frac{1}{(i-1)!(j-1)!}\Big(\sum_{n>0}\frac{1}{n}[z^{en-i}]\Big\{f(z)^{an}g(z)^{en}\Big\}\Big)\\
&\Big(\sum_{n>0}\frac{1}{n}[z^{en-j}]\Big\{f(z)^{an}g(z)^{en}\Big\}\Big)\bigg]\\
&\textnormal{exp}\bigg[\sum_{n>0}\frac{1}{n}\Big(c_1(\alpha)\cdot c_1+
\frac{c_1^2}{2}a\Big)[z^{ne}]\Big\{z\frac{d}{dz}\textnormal{log}\big(f(z)-f(0)\big)f(z)^{an}g(z)^{en}\Big\}q^n\\
&+\sum_{n>0}\frac{1}{n}\Big(
\frac{c_1^2}{2}a- c_1(E_S)\cdot c_1\Big)[z^{ne}]\Big\{z\frac{d}{dz}\textnormal{log}\big(g(z)-g(0)\big)f(z)^{an}g(z)^{en}\Big\}q^n
\\
&-\Big(c_1^2\frac{e}{2n}-c_1(E_S)\cdot c_1\Big)[z^{ne}]\Big\{f(z)^{an}g(z)^{en}\Big\}q^n\\
&-c_1^2\sum_{\begin{subarray}a n,m>0\\ j>0\end{subarray}}\frac{j}{n\cdot m}[z^{ne-j}]\Big\{f(z)^{an}g(z)^{en}\Big\}[z^{me+j}]\Big\{f(z)^{am}g(z)^{em}\Big\}q^{n+m}
\bigg]\,,
\numberthis
\end{align*}
and for $Y=X$:
\begin{align*}
Z_{E_X}(f,g,\alpha;q)
=&\exp\bigg[\sum_{\begin{subarray}a n>0\\ l|n\end{subarray}} (-1)^{ne}\frac{n}{l^2}[z^{ne}]\Big\{\bigg(c_1(\alpha)\cdot c_3\,z\frac{d}{dz}\log\big(f(z)-f(0)\big)\\
-&c_1(E_X)\cdot c_3\,z\frac{d}{dz}\Big(\log\big(g(z)-g(0)\big)+\log\big(g(-z)-g(0)\big)\Big)\bigg)\\
\cdot & f(z)^{an}g(z)^{en}g(-z)^{en} +f(z)^{an}g(z)^{en}g(-z)^{en}\Big\}\bigg]\,.
\end{align*}",
We now apply the following corollary of the Lagrange inversion:
\begin{corollary}[\cite{bojko3}]
	\label{CorMG}
 Let $Q(t)\in \mathbb{K}\llbracket t\rrbracket$  be a power-series independent of $q$ with a constant term being $1$ over a field $\mathbb{K}$ (e.g. the field of Laurent-series over $\QQ$).
		Let  $H_i(q)$ be the $e$ different Newton--Puiseux solutions to \begin{equation}
	\label{eqGSo}\big(H_i(q)\big)^e=qQ\big(H_i(q)\big)\,,\end{equation}
	then for any Laurent series $\phi(t)$, we have
	\begin{align}
	\label{eqHLG}
	\sum_{k=1}^e\phi\big(H_i(q)\big) &= e\phi_0+ [t^{-1}]\big\{\phi'(t)\textnormal{log}\big(Q(t)\big\}+ \sum_{n\neq 0}\frac{1}{n} [t^{ne-1}]\Big\{\phi'(t)Q(t)^n\Big\}x^n\,,\\
	\textnormal{log}\Big(\Big(\prod_{i=1}^e H_i(q)\Big)/q\Big) &= \sum_{m>0}\frac{1}{m}[t^{me}]\Big\{Q(t)^m\Big\}q^m\,.
	\end{align}
\end{corollary}
Then setting $Q(t) = f(t)^{a}g(t)^e$ we first simplify the first term under the exponential in \eqref{eqQug}
by using $[z^{ne-i}]\{\cdots \} = \frac{1}{i}[z^{ne}]\{\frac{d}{dz}z^i\cdots\}$ and $\phi=z^i$ to get
\begin{align*}
&\exp\bigg[c_1^2\sum_{k>0}\frac{g_{k}}{k!}\Big(\sum_{\begin{subarray}a i+j=k\\i,j>0
	\end{subarray}}(-1)^i\frac{1}{(i-1)!(j-1)!}\Big(\sum_{n>0}\frac{1}{n}[z^{en-i}]\Big\{f(z)^{an}g(z)^{en}\Big\}\Big)\\
&\Big(\sum_{n>0}\frac{1}{n}[z^{en-j}]\Big\{f(z)^{an}g(z)^{en}\Big\}\Big)\bigg] =  \exp\bigg[c_1^2\sum_{k>0}\frac{g_{k}}{k!}\sum_{s,t=1}^e\sum_{\begin{subarray}a i+j=k\\i,j>0
	\end{subarray}}{k\choose j}(-1)^iH^i_s(q)H^j_t(q)\bigg] \,.
\end{align*}
Using Corollary \ref{CorMG} for the rest of the terms, we see that 
\begin{align*}
\label{eqfagT2}
&Z_{E_S}(f,g,\alpha;q) =\exp\bigg[c_1^2\sum_{k>0}\frac{g_{k}}{k!}\sum_{s,t=1}^e\sum_{\begin{subarray}a i+j=k\\i,j>0
	\end{subarray}}{k\choose j}(-1)^iH^i_s(q)H^j_t(q)\bigg]\\
&\prod_{i=1}^e\Big(\frac{f\big(H_i(q)\big)}{f(0)}\Big)^{c_1(\alpha)\cdot c_1 +a\frac{c_1^2}{2}}\prod_{i=1}^e\Big(\frac{g\big(H_i(q)\big)}{g(0)}\Big)^{e\frac{c_1^2}{2}-c_1(E_S)\cdot c_1 }\Big(\frac{1}{q}\prod_{i=1}^e H_i(q)\Big)^{e\frac{c_1^2}{2}-c_1(E_S)\cdot c_1}\\
&\cdot \exp\Big[-\sum_{\begin{subarray}a n,m>0\\ j>0\end{subarray}}\frac{j}{nm}[z^{ne-j}]\Big\{f(z)^{an}g(z)^{en}\Big\}[z^{me+j}]\Big\{f(z)^{am}g(z)^{em}\Big\}q^{n+m}\Big]^{c_1^2}
\numberthis
\end{align*}
\label{Eq:R1req1}
where $H_i(q)$ are the solutions to
\begin{equation}
\label{eqVCh}
H_i^e = qf^a(H_i)g^e(H_i)\,,\qquad H_i=\omega^i_eq^{\frac{1}{e}}f^{\frac{a}{e}}(H_i)g(H_i).
\end{equation}
Using \eqref{eqVCh} and \eqref{eqFaGT}, we obtain
\begin{align*}
\label{Eq:ZESexp}
&Z_{E_S}(f,g,\alpha;q)=\\
&\prod_{i\neq j}g(H_i-H_j)^{c_1^2}\prod_{i=1}^e\big(g(H_i)g(-H_i)\big)^{-ec_1^2}\prod_{i=1}^e\bigg(\frac{f(H_i)}{f(0)}\bigg)^{c_1(\alpha)\cdot c_1}\prod_{i=1}^e\bigg(\frac{f(H_i)}{f(0)}\bigg)^{a\mu_S(E_S)}\\
&\cdot \exp\bigg[-\sum_{\begin{subarray}a n,m>0\\ j>0\end{subarray}}\frac{j}{nm}[z^{ne-j}]\Big\{f(z)^{an}g(z)^{en}\Big\}[z^{me+j}]\Big\{f(z)^{am}g(z)^{em}\Big\}q^{n+m}\bigg]^{c_1^2}\,,
\numberthis
\end{align*}
where $\mu_S(E_S) = \frac{c_1(E_S)\cdot c_1}{e}$. This notation is chosen so that for $S$ a del Pezzo surface this coincides with the slope of $E_S$ with respect to the ample divisor class $c_1$.
We only need to deal with the last term
\begin{equation}
\label{eqGeR}
G_e(Q)= \exp\bigg[-\sum_{\begin{subarray}a n,m>0\\ j>0\end{subarray}}j\frac{1}{m}[z^{me+j}]\Big\{Q^m(z)\Big\}\frac{1}{n}[z^{ne-j}]\Big\{Q^n(z)\Big\}q^{n+m}\bigg]\,,   
\end{equation}
for which we proved Theorem \ref{thmCtC} in \cite{bojko3}.

\begin{remark}
	\label{remWhy}
	Knuth \cite{knuthconpol} studies composition polynomials $\{F_n(a)\}_{n\in \mathbb{Z}_{\geq 0}}$ which can be always obtained as
	$
	F_n(a) =[z^n]\Big\{f^{an}(z)\Big\}
	$
	for some generating series $f(z)$. These satisfy useful identities like
	$$
	\sum_{j\in \mathbb{Z}}F_{n+j}(a)F_{m-j}(a) = F_{n+m}(a)\,,
	$$
	where we assume that $F_{-k}(a)=0$ for $k>0$.
	Note that the sum in \eqref{eqGeR} only goes over $j>0$. This complicated the proof of the above identity substantially.
\end{remark}
Applying \eqref{eqExp} to \eqref{Eq:ZESexp}, we obtain the following result:
\begin{align*}
\label{eqfagT3}
&Z(f,g,\alpha;q)= \prod_{i=1}^e\bigg(\frac{f(H_i)}{f(0)}\bigg)^{c_1(\alpha)\cdot c_1}\prod_{i=1}^e\bigg(\frac{f(H_i)}{f(0)}\bigg)^{a\mu_S(E_S)}\\
&\prod_{i=1}^e\bigg(\frac{f(H_i)}{f(0)}\bigg)^{ac_1^2}\prod_{i=1}^e\bigg(\frac{-H_i}{g(-H_i)}\bigg)^{ec_1^2}
\cdot \prod_{i\neq j }\bigg(\frac{g(H_{i}-H_j)}{(H_{i}-H_{j})}\bigg)^{c_1^2}\prod_{i=1}^e\bigg(a\frac{f'(H_i)}{f(H_i)}+e\frac{g'(H_i)}{g(H_i)}-\frac{e}{H_i}\bigg)^{c_1^2}\,.
\numberthis
\end{align*}
In the case $Y=X$ the computation is much simpler. In \cite[(5.9)]{bojko2}, we defined a universal transformation as it was useful in expressing generating series. We now generalize it. Define 
$U_e: \big(R\llbracket t\rrbracket\big)_1\to \big(R\llbracket t\rrbracket\big)_1$\footnote{The subscript $1$ represent that we are looking at power-series with constant term 1.} by 
\begin{equation}
\label{Qmap}
f(t)\mapsto \prod_{n>0}\prod_{k=1}^nf((-1)^ee^{\frac{2\pi i k}{n}}t)^{n}\,.
\end{equation}
Then by a familiar computation (see proof of \cite[Prop. 4.13]{bojko2}) together with Corollary \ref{CorMG}, we obtain
\begin{align*}
\label{eqfaCY}
&Z_{E_X}(f,g,\alpha;q)\\
= & U_e\bigg(\prod_{i=1}^ef\big(H_i(q)\big)^{c_1(\alpha)\cdot c_3}\prod_{i=1}^e\Big(g\big(H_i(q)\big)g\big(-H_i(q)\big)\Big)^{-c_1(E_X)\cdot c_3}\bigg(\frac{\prod_{i=1}^eH_i(q)}{q}\bigg)^{c_1(E_X)\cdot c_3}\bigg)\\
=&U_e\bigg(\prod_{i=1}^ef\big(H_i(q)\big)^{c_1(\alpha)\cdot c_3}\prod_{i=1}^ef\big(H_i(q)\big)^{a\mu_X(E_X)}\bigg)\,,
\numberthis
\end{align*}
where $\mu_X(E_X)=\frac{c_1(E_X)\cdot c_3(X)}{e}$ and $H_i$ solve
$$
H_i^e(q) = qf^a(H_i)g(H_i)g(-H_i)\,,\qquad H_i=\omega^i_eq^{\frac{1}{e}}f^{\frac{a}{e}}(H_i)g(H_i)g(-H_i)\,.
$$
We used this change of variables in the final step and neglected writing $f(0)$ and $g(0)$ in the computation. They should always be included whenever the corresponding $f(H_i)$ or $g(H_i)$ term appears.
\section{Symmetries}
We now focus on special cases of the tautological invariants for fixed choices of multiplicative genera $g$ and $f$. The examples which we study are the Segre and Verlinde numbers, K-theoretic and cohomological descendant series and the  $\chi_y$-genus in the case $Y=S$ and additionally the Nekrasov genus when $Y=X$. 
\subsection{Dependence on $c_1(E),c_2(E)$}
In this section, we focus on comparing our computation to the previous results when $E_S= \mO^{e}_S$. Let us, therefore, recall the result of Arbesfeld--Johnson--Lim--Oprea--Pandharipande \cite{AJLOP} by setting $\beta=0$ in \cite[eq. (13)]{AJLOP}. This gives us
\begin{align*}
\label{eqAJLO}
&\sum_{n=0}^{\infty}q^n\int_{\textnormal{Quot}_S(\mathbb{C}^e,n)}f(\alpha^{[n]})g(T^{\textnormal{vir}}) = \prod_{i=1}^e\bigg(\frac{f(H_i)}{f(0)}\bigg)^{c_1(\alpha)\cdot c_1}\prod_{i=1}^e\bigg(\frac{f(H_i)}{f(0)}\bigg)^{ac_1^2}\\
&\prod_{i=1}^e\bigg(\frac{-H_i}{g(-H_i)}\bigg)^{ec_1^2}
\cdot \prod_{i\neq j }\bigg(\frac{g(H_{i}-H_j)}{(H_{i}-H_{j})}\bigg)^{c_1^2}\prod_{i=1}^e\bigg(a\frac{f'(H_i)}{f(H_i)}+e\frac{g'(H_i)}{g(H_i)}-\frac{e}{H_i}\bigg)^{c_1^2}\,,
\numberthis
\end{align*}
where
$$
H_i^e = qf^a(H_i)g^e(H_i)\,,\qquad H_i=\omega^i_eq^{\frac{1}{e}}f^{\frac{a}{e}}(H_i)g(H_i)\,.
$$
We note that our $H_i$ differs from the source by a sign.
By comparing \eqref{eqAJLO} and \eqref{eqfagT3}, we obtain the following multiple results:
\begin{theorem}
	\label{thmZEZO}
	We have the following identity:
	$$
	\frac{Z_{E_S}(f,g,\alpha;q)}{Z_{\mO^{e}_S}(f,g,\alpha;q)} = \prod_{i=1}^e\bigg(\frac{f(H_i)}{f(0)}\bigg)^{a\mu_S(E_S)}\,.
	$$
\end{theorem}
\begin{proof}
	This follows immediately from \eqref{eqfagT3} and \eqref{eqAJLO}.
\end{proof}
\begin{theorem}
	\label{thmZEZCm}
	For any  $P\in R[[x_1,x_2,\ldots]]$, where $R$ is a ring, we have the following equality: 
	$$
	\int_{[\Quot]^{\vir}}P(\textnormal{ch}_1(T^{\textnormal{vir}}),\textnormal{ch}_2(T^{\textnormal{vir}}),\ldots) = \int_{\textnormal{Quot}_S(\mO^{e}_S,n)}P(\textnormal{ch}_1(T^{\textnormal{vir}}),\textnormal{ch}_2(T^{\textnormal{vir}}),\ldots)
	$$
	for all $E_S$ and $n$. By work of Shen \cite{Shen}, this is equivalent to the identification
	$$
	\big[\Quot\big]^{\textnormal{vir}}_{\textnormal{cob}} = \big[\QuotCC\big]^{\textnormal{vir}}_{\textnormal{cob}}\,.
	$$
\end{theorem}
\begin{proof}
	Setting $f=1$ and 
	$$
	g(\T)=\prod_{j=1}^Mc_{u_j}(T^{\textnormal{vir}})\,,
	$$
	in Theorem \ref{thmZEZO}, we obtain
	\begin{align*}
	&\sum_{\begin{subarray}a l_1,\ldots, l_M\\
		\sum_{j=1}^Ml_i = en
		\end{subarray}}\int_{[\Quot]^{\vir}}c_{l_1}(\T)\ldots c_{l_M}(T^{\textnormal{vir}})u_{1}^{l_1}\ldots u_{M}^{l_M}\\
	=& \sum_{\begin{subarray}a l_1,\ldots, l_M\\
		\sum_{j=1}^Ml_i = en
		\end{subarray}}\int_{[\QuotCC]^{\vir}}c_{l_1}(\T)\ldots c_{l_M}(T^{\textnormal{vir}})u_{1}^{l_1}\ldots u_{M}^{l_M}  
	\end{align*}
	In particular, as these polynomials in $u_i$ coincide for any value of $u_i$, we obtain  
	\begin{equation}
	\label{eqter}
	\int_{[\Quot]^{\vir}}c_{l_1}(\T)\ldots c_{l_M}(T^{\textnormal{vir}})
	= \int_{[\QuotCC]^{\vir}}c_{l_1}(\T)\ldots c_{l_M}(T^{\textnormal{vir}})\,.
	\end{equation}
	As each 
	$$
	\int_{[\Quot]^{\vir}}P\big(\textnormal{ch}_1(\T),\ldots,\textnormal{ch}_2(\T)\big)
	$$
	is expressed as a linear combination of \eqref{eqter}, we recover the statement.
\end{proof}

\subsection{Segre--Verlinde duality and rationality for $Y=S$}
\label{sec: surfaces}
We address here the question of rationality of generating series. This has been studied in the case when $E =\mO^e$ in \cite{AJLOP, OP1, JOP} and \cite{Lim}. The result of this section that was hardest to obtain was the computation of the poles of the K-theoretic descendants with the additional inclusion of the $\chi_y$-genus. The rationality of K-theoretic descendants has been addressed by \cite{AJLOP} without the $\chi_y$-genus and in the simplest case $E_Y = \mO_Y$. The higher rank case poses new difficulties of having to work with Newton--Puiseux generating series

Thus let $E$ be any torsion-free sheaf and $[\Quot]^{\vir}$ the virtual fundamental classes from above. Recall from the introduction that we set $\Lambda_{-y}$ to be the multiplicative genus given by the invertible power series $g(t) = 1-ye^t$. We will want to prove that the following power-series give rational functions for any $E$:
\begin{itemize}
	\item descendant invariants 
	\begin{equation}
	\label{eqPoles1}
	Z^{\textnormal{des}}_{E_S}(\alpha_1,\ldots,\alpha_l|k_1,\ldots,k_l)(q)=\sum_{n\geq 0}q^n\int_{[\Quot]^{\vir}}\textnormal{ch}_{k_1}\big(\alpha_1^{[n]}\big)\cdots\textnormal{ch}_{k_l}\big(\alpha^{[n]}_l\big)c\big(T^{\vir}\big)  
	\end{equation}
	\item K-theoretic invariants 
	\begin{equation}
	\label{eqPoles2}
	Z^{\chi_{-y}}_{E_S}(\alpha_1,\ldots,\alpha_l|k_1,\ldots,k_l) = \sum_{n\geq 0}q^n\chi^{\textnormal{vir}}\Big(\Quot,\wedge^{k_1}\alpha_1^{[n]}\otimes\ldots \otimes \wedge^{k_l}\alpha_l^{n}\otimes \Lambda_{-y}\Omega^{\vir}\Big)\,,
	\end{equation}
\end{itemize}
and study their poles. .

This generalizes the work of \cite{JOP, AJLOP, Lim} and \cite{Woonamthesis}, where it was shown for $E_S=\mathbb{C}^e$ and the poles were computed for \eqref{eqPoles2} when
$$
e=1 \quad\textnormal{and}\quad y=0\,.
$$
We observe that their result on the order of the pole at $q=1$ is unchanged by relaxing the condition on $E_S$ and $y$. However, we obtain a new contribution coming from the pole at $q(1+y)=1$ telling us that there exists a polynomial
$
P(q,y;k_1,\ldots,k_n)\,,
$
such that 
$$
\frac{ P(q,y;k_1,\ldots,k_n)}{(1-q)^{2\sum_ik_i}(1-(1-y)^eq)^{e +\sum_{i}k_i}} =  Z^{\chi_{-y}}_{E_S}(\alpha_1,\ldots,\alpha_l|k_1,\ldots,k_l)(q)\,.
$$

Secondly, for any $E_S$ let us define the Segre and Verlinde series as follows:
\begin{align*}
S_{E_S,\alpha}\big(q\big)&= \sum_{n\geq 0}q^n\int_{[\Quot]^{\vir}}s_{en}(\alpha^{[n]})\,,\\
V_{E_S,\alpha}(q) &= \sum_{n\geq 0}q^n\chi^{\vir}\big(\Quot,\textnormal{det}\big(\alpha^{[n]}\big)\big)\,.
\end{align*}
Then we show that the following Segre--Verlinde duality generalizing \cite[Thm. 13 (i) ]{AJLOP} holds:
$$
S_{E_S,\alpha}\big((-1)^eq\big) = V_{E_S,\alpha}(q)\,.
$$
Finally, the most interesting and enlightening result of this section is stated in Theorem \ref{thm4fo}. We give a more natural formulation of the fourfold correspondence in \cite[§1.9]{AJLOP}. The extension to  CY-fourfolds is going to be given in §\ref{secCY}. We further include curves to complete the lower floor of the diagram \ref{fig1} in §\ref{secTwF}.
\begin{theorem}
	The following statements hold:
	\begin{itemize}
		\item $Z^{\chi_{-y}}_{E_S}(\alpha_1,\ldots,\alpha_l|k_1,\ldots,k_l)(q)$ is a rational function and has poles of order less than or equal to $\sum_{i=1}^l2k_i$ at $q=1$ and less than or equal to $e +\sum_{i=1}^lk_i$ at $q(1+y)^e=1$.
		\item $Z^{\textnormal{des}}_{E_S}(\alpha_1,\ldots,\alpha_l|k_1,\ldots,k_l)(q)$ is rational. 
	\end{itemize}
\end{theorem}
\begin{proof}
	We begin with 
	\begin{equation}
	\label{eqDeK}
	Z^{\chi_{-y}}_{E_S}(\alpha_1,\ldots,\alpha_l|k_1,\ldots,k_l)(q)=\frac{1}{k_1!\ldots k_l!}\frac{d^{k_1}}{(dx_1)^{k_1}}\cdots \frac{d^{k_l}}{(dx_l)^{k_l}}Z^{\chi_{-y}}_{\wedge}(\alpha_1,\ldots,\alpha_l | x_1,\ldots,x_l)|_{x_i=0}\,,\end{equation}
	where
	$$
	Z^{\chi_{-y}}_{\wedge}(\alpha|x)(q) = Z(f,g,\alpha;q)
	$$
	for 
	$$
	f(z) = 1+ xe^z\,,\qquad g(z)=\frac{z(1-ye^{-z})}{(1-e^{-z})}\,.
	$$
	For simplicity, we reduced to the case where $l=1$, $x_1=x$, $\alpha_1=\alpha$. For $l>1$ the arguments are identical.
	Using \eqref{eqfagT3}, we may conclude that 
	$$
	Z_{E_X}(f,g,\alpha;q)= A(q)^{c_1^2}B(q)^{c_1(\alpha)\cdot c_1+a\mu_S(E_S)}\,, $$
	where
	\begin{align*}
	\label{eqAqBq}
	A(q) &= \prod_{i=1}^e\Big(\frac{1+xe^{H_i}}{1+x}\Big)^a\prod_{i=1}^e\Big(\frac{1-e^{H_i}}{1-ye^{H_i}}\Big)^e\prod_{i\neq j}\frac{1-ye^{-H_i+H_j}}{1-e^{-H_i+H_j}}\prod_{i=1}^e\Big(a\frac{xe^{H_i}}{1+xe^{H_i}}+\frac{ye^{-H_i}}{1-ye^{-H_i}}-\frac{e^{-H_i}}{1-e^{-H_i}}\Big)\,,\\
	B_s(q) &= \prod_{i=1}^e\Big(\frac{1+xe^{H_i}}{1+x}\Big)\,,
	\numberthis
	\end{align*}
	and
	$$
	\Big(\frac{1-e^{-H_i}}{1-ye^{-H_i}}\Big)^e\frac{1}{(1+xe^{H_i})^{a}}=q\,.
	$$
	Using the change of variables 
	$$
	t_i=\frac{1-e^{-H_i}}{1-ye^{-H_i}}\,,\qquad e^{-H_i}=\frac{t_i-1}{t_iy-1}\,,
	$$
	this can be expressed as 
	\begin{align*}
	\label{eqAqBq2}
	A(q) &=\prod_{i=1}^e\frac{\Big(1+x\frac{t_iy-1}{t_i-1}\Big)^{a}}{(1+x)^a} \prod_{i=1}^e\Big(\frac{t_i}{t_i(1+y)-1}\Big)^e\prod_{i\neq j}\frac{t_it_jy - (1+y)t_j+1}{t_i-t_j}\\
	&\cdot \prod_{i=1}^e\Big(a\frac{x(t_iy-1)}{(t_i-1)+x(t_iy-1)}+\frac{(t_i-1)(t_iy-1)}{(y+1)t_i}\Big)\\ B_s(q)& = \prod_{i=1}^e\frac{\Big(1+x\frac{t_iy-1}{t_i-1}\Big)}{(1+x)}\,,
	\numberthis
	\end{align*}
	and 
	\begin{equation}
	\label{eqztR}
	t_i^e\Big(\frac{t_i-1}{t_i-1+x(t_iy-1)}\Big)^a =q\,.
	\end{equation}
	
	We now prove the rationality and the bound on the order of poles at $q=1$ and $q(1+y)^e = 1$ by structural analysis of each ingredient.  Set $R(x,t_i,y)$ to be one of the factors in \eqref{eqAqBq2}, then we may express 
	\begin{align*}
	\Big(\frac{d}{dx}\Big)^kR(x,t_i,y) = \sum_{\begin{subarray}am_i, r_{i,1},\cdots r_{i,m_i},p\geq 0\\
		\sum_{i,j}r_{i,j}+p = k
		\end{subarray}}C_p&\begin{pmatrix}m_1, r_{1,1},&\dots &,r_{1,m_1}\\
	\vdots&&\vdots\\
	m_e,r_{e,1},&\dots&,r_{e,m_e}
	\end{pmatrix}\\&\cdot
	\Big(\frac{\partial}{\partial x}\Big)^p\prod_{i=1}^e\Big(\frac{\partial}{\partial t_i }\Big)^{m_i}R(x,t_i,y)\prod_{j=1}^{m_i}\Big(\frac{\partial}{\partial x}\Big)^{r_{i,j}}t_i\,.
	\end{align*}
	Note that in the above, we may replace
	\begin{equation}
	\label{eqCmir}
	C(m_i,r_{i,j})= \Big(\frac{\partial}{\partial x}\Big)^p\prod_{i=1}^e\Big(\frac{\partial}{\partial t_i}\Big)^{m_i}R(x,t_i,y)\prod_{j=1}^{m_{i}}\Big(\frac{\partial}{\partial x}\Big)^{r_{i,j}}t_i
	\end{equation}
	by
	\begin{equation}
 \label{Eq:symrat}
	\frac{1}{e!}\sum_{\sigma\in S_e}\Big(\frac{\partial}{\partial x}\Big)^p\prod_{i=1}^e\Big(\frac{\partial}{\partial t_i}\Big)^{m_{\sigma(i)}}R(x,t_i,y)\prod_{j=1}^{m_{\sigma(i)}}\Big(\frac{\partial}{\partial x}\Big)^{r_{\sigma(i),j}}t_i\,.
	\end{equation}
	Here $C(m_i,r_{i,j})$ is a function in variables 
	$$
	q^{i}_e =\omega^i_eq^{\frac{1}{e}}\,,
	$$
	where we continue using the notation $e^{\frac{2\pi i k}{e}} = \omega^k_e$.
	Because, the expression \eqref{Eq:symrat} is symmetric in $t_i$'s and therefore in $q_i^e$'s, it is sufficient for the pole at $q=1$ to show that evaluating \eqref{eqCmir} at $x=0$ can be written as a sum of rational functions with poles of order $\leq c_i$ at $q^i_e = 1$ such that 
	$
	\sum_{i=1}^e c_i \leq 2k\,.
	$

	We first address $\Big(\frac{\partial}{\partial x}\Big)^{r_{i,j}}t_i|_{x=0}$, which we claim has a pole of order $\leq 2 r_{i,j}-1$ at $q^i_e=1$. In the following, we use $F(t_i,x)$ to denote the left-hand side of \eqref{eqztR}.
	
	We begin our inductive argument by using 
	$$
	\frac{d t_i}{d x} =-\frac{\partial F(t_i,x)}{\partial x}\Big(\frac{\partial F(t_i,x)}{\partial t_i}\Big)^{-1}
	$$
	An easy computation shows that
	$$
	\frac{\partial t_i}{\partial x}\Big|_{x=0} = -\frac{a}{e}\frac{yq^i_e-1}{q^i_e-1}q^i_e
	$$
	where we used that
	$
	t_i|_{x=0} = q^{i}_e\,.
	$
	We may express the higher-order derivatives using
	$$
	\Big(\frac{d}{dx}\Big)^kF(t_i,x)=\frac{\partial F(t_i,x)}{\partial t_i}\Big(\frac{d}{dx}\Big)^kt_i+\quad \textnormal{terms of the form }\quad \Big(\frac{\partial}{\partial x}\Big)^n\frac{\partial^{m_i}}{\partial t_i^{m_i}}F(t_i,x)\prod_{j=1}^{m_i}\Big(\frac{\partial}{\partial x}\Big)^{d_{i,j}}t_i\,,
	$$
	where  $n+\sum_{j}d_{i,j}= k$. 
	We use that
	$$
	\Big(\frac{\partial}{\partial x}\Big)^n\Big(\frac{\partial}{\partial t_i}\Big)^{m_i}F(t_i,x)|_{x=0}\quad 
	\textnormal{has poles of order}\quad \leq n +m_i  \quad \textnormal{at}\quad q^i_e = 1\,,
	$$
	and $\frac{\partial F(t_i,x)}{\partial t_i}|_{x=0}=et_i^{e-1}$.  Both of these are checked by a direct computation to conclude by the induction assumption that $\Big(\frac{d}{dx}\Big)^kt_i|_{x=0}$ has poles of order
	\begin{align*}
 \leq \sum_{j=1}^{m_i}(2d_{i,j}-1) + m_i+n= 2k - n\leq 2k-1 \quad \textnormal{at}\quad q^i_e = 1\,.
	\end{align*}
	Note that for this argument to work, we need to treat the case $n=0$ separately by noting that  $(\frac{d}{dt_i})^kF(t_i,x)\Big|_{x=0}$ has no poles at $q^i_e = 1$ nor at $q^i_e(1+y) = 1$. We will ignore the contributions to the pole at $q=0$, as we are working with power-series.
	
	Finally, we observe that except for the second factor of $A(q)$ in \eqref{eqAqBq2}, the term
	$$
	\Big(\frac{\partial}{\partial x}\Big)^p\prod_{i=1}^e\Big(\frac{\partial}{\partial t_i}\Big)^{m_i}R(x,t_i,y)|_{x=0} \,.
	$$
	has poles at $q^i_e=1$ of orders which add up to $p+\sum_{i=1}^e m_i$ and no poles at $q(1+y)=1$. The combined order of poles in \eqref{eqCmir} is given by
	$$
	p+\sum_{i=1}^e\Big(m_i + \sum_{j=1}^{m_i}(2r_{i,j}-1)\Big)\leq 2k\,.
	$$
 We left the term 
 $$
 R(x,t_i,y)=\Big(\frac{t_i}{t_i(1+y)-1}\Big)^e
 $$
 to be addressed last. In this, case,
 $$
 \prod_{i=1}^e\Big(\frac{\partial}{\partial t_i}\Big)^{m_i}R(x,t_i,y)|_{x=0}
 $$
 has  pole of order $e +m_i$ at $q_e^i(1+y) = 1$. Due to
 $$
 \prod_{i=1}\frac{1}{1-(1+y)q^i_e} = \frac{1}{1-(1+y)^eq}\,,
 $$
 we have thus shown that the order of the pole at $(1+y)^eq = 1$ is bounded by $e+\sum_i m_i\leq e+k$.
	Moving onto the descendant series, we may replace Chern characters with Chern classes as in the proof of Johnson--Oprea--Pandharipande \cite[Thm. 2]{JOP} to work with
	$$
	Z^{\textnormal{des}}_{S,E}(\alpha_1,\ldots,\alpha_l|k_1,\ldots,k_l) = \frac{1}{k_1!}\cdots\frac{1}{k_l!}\frac{\partial^{k_1}}{(\partial t_1)^{k_1}}\cdots \frac{\partial^{k_l}}{(\partial t_1)^{k_l}}W(q)|_{t_i=0}\,,
	$$
	where 
	$$
	W(q) = \sum_{n\geq 0}q^n\int_{[\Quot]^{\vir}}c_{t_1}(\alpha_1^{[n]})\cdots c_{t_l}(\alpha^{[n]}_{l})\cdot c\big(T^{\vir}\big)\,.
	$$
	The proof then follows from the proof of \cite[Thm. 2]{JOP} or \cite[Lem. 24]{Woonamthesis} or by the same arguments above. We leave it to the reader to determine the poles. 
\end{proof}
Perhaps the most interesting and enlightening result of this section follows now. It is meant to give a better explanation of the surprising symmetry observed in \cite[§1.9]{AJLOP}, which required additional restrictions because the invariants were only available when $E_S=\mO^{e}_S$. We include also the natural generalization of the Segre--Verlinde duality which was already observed for trivial vector bundles in \cite{AJLOP}.

\begin{theorem}

\label{thm4fo}
	Let $E,F\to S$ be two torsion-free sheaves with ranks $e=\textnormal{rk}(E),f=\textnormal{rk}(F)$ and $\alpha\in K^0(S)$ then the following holds:
	\begin{align*}
	S_{E,\alpha}\big((-1)^eq\big)&=V_{E,\alpha}(q)\,,\\
	S_{E,F}\big((-1)^eq\big) &= S_{F,E}\big((-1)^fq\big)\,.\\
	V_{E,F}(q) &= V_{F,E}(q)\,.
	\end{align*}
\end{theorem}
\begin{proof}
	The first equality follows immediately from the computations in \cite[§5.2, §5.3]{AJLOP} and \eqref{eqfagT3}. 
	For the other two equalities, we express the series explicitly and we need to do some extra work:
	\begin{align*}
		S_{E,F}\big((-1)^eq\big) &= A^{c_1^2}_{(e,f)}\big((-1)^eq\big)B^{\mu_S(F)f}_{(e,f)}\big((-1)^eq\big)B^{\mu_S(E_S)f}_{(e,f)}\big((-1)^eq\big)\,.\\
  S_{F,E}\big((-1)^fq\big) &= A^{c_1^2}_{(f,e)}\big((-1)^fq\big)B^{\mu_S(E_S)e}_{(f,e)}\big((-1)^fq\big)B^{\mu_S(F)e}_{(f,e)}\big((-1)^fq\big)chan
	\end{align*}
	where 
	\begin{align*}
	A_{(e,f)}(q) &= \prod_{i=1}^e\Big(\frac{1}{1+H_i}\Big)^f\prod_{i=1}^e(-H_i)^e\prod_{i\neq j}\frac{1}{H_{i}-H_{j}}\prod_{i=1}^e\bigg(-f\frac{1}{1+H_i}-\frac{e}{H_i}\bigg)\,,\\
	B_{(e,f)}(q)&=\prod_{i=1}^e\frac{1}{1+H_i}\,,\\
	& \textnormal{for the variable transformation}\qquad H_{i}^e\prod_{i=1}^e\Big(1+H_i\Big)^f =q\,.
	\end{align*}
	Following \cite{AJLOP}, we use the notation \begin{align*}
	& M_{(e,f)}(q)=A_{(e,f)}\big((-1)^eq\big)\,,  \qquad N_{(e,f)}(q)=B^e_{(f,e)}\big((-1)^eq\big) \,.
	\end{align*}

	It was shown in the proof of \cite[Thm. 14]{AJLOP} that 
	$$
	M_{(f,e)}(q) = M_{(e,f)}(q)\,,\qquad N_{(f,e)}(q) = N_{(e,f)}(q)\,.
	$$
	The claim then follows from
	\begin{align*}
	S_{F,E}\big((-1)^fq\big) &= M^{c_1^2}_{(f,e)}\big(q\big)N^{\mu_S(E_S)}_{(f,e)}\big(q\big)N^{\mu_S(F)}_{(f,e)}\big(q\big)\,,\\
	S_{E,F}\big((-1)^eq\big) &= M^{c_1^2}_{(e,f)}\big(q\big)N^{\mu_S(F)}_{(e,f)}\big(q\big)N^{\mu_S(E_S)}_{(e,f)}\big(q\big)\,.
	\end{align*}
\end{proof}
\subsection{Nekrasov genus and Segre--Verlinde duality for $Y=X$}
\label{secCY}
After we have addressed the invariants of surfaces, we now move on to a parallel study for Calabi--Yau fourfolds. Unlike the surface case, there is an additional invariant which has a conjectural formula by Nekrasov and Nekrasov--Piazzalunga \cite{Nekrasov1, Nekrasov2} in the local case over $\CC^4$. This is called the Nekrasov genus and we obtain an explicit expression for it in terms of the Mac--Mahon function. It is important to remind the reader here that these results are relying on Claim \ref{conjecture quot WC}, which will be proved in \cite{bojko3}. 

We focus only on the case with a single K-theory class $\alpha\in K^0(X)$ of rank $a$. Set the notation:
\begin{align*}
N(\alpha,y;q) &= \sum_{n\geq 0}q^n\hat{\chi}^{\vir}\Big(\Lambda^{\bullet}_{y^{-1}}(\alpha^{[n]})\otimes \textnormal{det}^{-\frac{1}{2}}\big(\alpha^{[n]}y^{-1}\big) \Big)\,,\\
S_{E,\alpha}(q) &= \sum_{n\geq 0}q^n\int_{[Q_X]^{\vir}}s\big(\alpha^{[n]}\big)\,.
\end{align*}
then we recall first the following relation, which is a special case of the arguments used in the proof of \cite[Prop. 5.5]{bojko2}:
\begin{equation}¨
\label{eqCoL}
\textnormal{lim}_{y\to 1^+}\big(1-y^{-1}\big)^{(e-a)n}[q^n]\Big\{N\big(\alpha,y;q\big)\Big\}
= (-1)^{ne}[q^n]\big\{S(\alpha;q)\big\}\,.
\end{equation}
Following  \cite{Nekrasov1, Nekrasov2} and Cao--Kool--Monavari \cite{CKM} we called this the \textit{cohomological limit}. 
Finally, note that we obtain the following expression generalizing that of  \cite[Thm. 5.15]{bojko2}.
\begin{theorem}
	\label{thmNek}
	If Claim \ref{conjecture quot WC} holds, then for all $\alpha$ with  $a:=\textnormal{rk}(\alpha)=2b+1$ and point-canonical orientations, we have
	\begin{align*}
	\label{eqNek1}
	N(\alpha,y;q)&=\prod_{i=1}^eU_e\bigg[\frac{(y-u_i)^2}{(y-1)^2u_i}\bigg]^{\frac{1}{2}\big(c_1(\alpha)\cdot c_3(X)+\mu_X(E_X)\big)}\,,\\
	\textnormal{\quad where}&\quad
	q=\frac{\big(u^{\frac{1}{2}}_i - u^{-\frac{1}{2}}_i\big)^e}{\big(y^{\frac{1}{2}}u^{-\frac{1}{2}}_i-y^{-\frac{1}{2}}u^{\frac{1}{2}}_i\big)^a}\,.
	\numberthis
	\end{align*}
	When   $\textnormal{rk}(\alpha)=e$, then
	$$
	N(\alpha,y; q) = \big(M\big(y^{\frac{e}{2}}q\big)M\big(y^{-\frac{e}{2}}q\big)\big)^{\frac{1}{2}\big(c_1(\alpha)\cdot c_3(X)+\mu_X(E_X)\big)}\,.
	$$
	where $M(q) = \prod_{n>0}(1-q^n)^{-n}$. 
\end{theorem}
\begin{proof}
	Recall from \cite[§5.2]{bojko2} that for this choice of genera, we obtain
	\begin{align*}
	f(x) &= \Big(y^{\frac{1}{2}}e^{-\frac{x}{2}}-y^{-\frac{1}{2}}e^{\frac{x}{2}}\Big) \,,\\
	g(x)g(-x)&=\frac{x}{e^{\frac{x}{2}}-e^{-\frac{x}{2}}}\,.
	\end{align*}
	Using \eqref{eqfaCY}, we compute that
	$$
	N(\alpha,y;q) = U_e\bigg[\prod_{i=1}^e\bigg(\frac{y^{\frac{1}{2}}e^{-\frac{H_i(q)}{2}}-y^{-\frac{1}{2}}e^{\frac{H_i(q)}{2}}}{y^{\frac{1}{2}}-y^{-\frac{1}{2}}}\bigg)^{c_1(\alpha)\cdot c_3(X)+\mu_X(E_X)}\bigg]\,,
	$$
	where 
	$$
	H_i^e(q) = q\Big(y^{\frac{1}{2}}e^{-\frac{H_i(q)}{2}}-y^{-\frac{1}{2}}e^{\frac{H_i(q)}{2}}\Big)^a\cdot \bigg(\frac{H_i(q)}{e^{\frac{H_i(q)}{2}}-e^{-\frac{H_i(q)}{2}}}\bigg)^e\,.
	$$
	Using $u_i(q) = e^{H_i(q)}$ and 
	$$
	\frac{y^{\frac{1}{2}}u_i^{-\frac{1}{2}}-y^{-\frac{1}{2}}u_i^{\frac{1}{2}}}{y^{\frac{1}{2}}-y^{-\frac{1}{2}}} = \sqrt{\frac{(1-u_iy^{-1})^2}{(1-y^{-1})^2u_i}}
	$$
	we conclude \eqref{eqNek1}. 
	
	Setting $a=e$ and using the notation $e^{\frac{2\pi i k}{e}} = \omega^k_e$, one can compute that
	$$
	u_k = \frac{1+y^{\frac{1}{2}}q^{\frac{1}{e}}\omega^k_e}{1+q^{\frac{1}{e}}y^{-\frac{1}{2}}\omega^k_e}\,.
	$$
	Plugging this into the formula leads to 
\begin{align*}
	N(\alpha,y;q) &= U_e\Big(\prod_{i=1}^e\big(1-y^{\frac{1}{2}}q^{\frac{1}{e}}\omega^i_e\big)\big(1-y^{-\frac{1}{2}}q^{\frac{1}{e}}\omega^i_e\big)\Big)^{-\frac{1}{2}\big(c_1(\alpha)\cdot c_3+\mu_X(E_X)\big)}\\
	&=\prod_{n>0}\prod_{k=1}^n\Big((1-y^{\frac{e}{2}}q)(1-y^{-\frac{e}{2}}q)\Big)^{-\frac{n}{2}\big(c_1(\alpha)\cdot c_3+\mu_X(E_X)\big)}\\
	&=M\big(y^{\frac{e}{2}}q\big)M(y^{-\frac{e}{2}}q)^{\frac{1}{2}\big(c_1(\alpha)\cdot c_3+\mu_X(E_X)\big)}\,.
	\end{align*}
\end{proof}
Using \eqref{eqCoL}, we derive the following corollary generalizing our previous result \cite[Thm. 3.10]{bojko} related to the virtual pull-back of Park \cite[Cor. 0.3]{Huy}.
\begin{corollary}
	Let $\alpha\in K^0(X)$ be of rank $e$. The generating series 
	$$
	I_{E_X,\alpha}(q):=\sum_{n\geq 0}q^n\int_{[Q_X]^{\textnormal{vir}}}c_n\big(\alpha^{[n]}\big)\,.
	$$
	is given by 
	$$
	I_{E_X,\alpha}(q) = M((-1)^eq)^{c_1(\alpha)\cdot c_3(X) + \mu_X(E_X)}\,.
	$$
\end{corollary}
Using that
$$
I_{E_X,\alpha}(q)=S_{E_X,-\alpha}(q)
$$
this recovers  Example \ref{Nekex}. 

We now study the analogue of the Segre-Verlinde duality observed for surfaces. Let us \textit{untwist} the virtual Euler characteristic the same way we did in \cite{bojko2}. This time, we define
$$
\chi^{\vir}(A) = \hat{\chi}^{\vir}(A\otimes \mathsf{E}^{\frac{1}{2}})\,,
$$
where $\mathsf{E}=\textnormal{det}\Big(\big(E^*_X\big)^{[n]}\Big)$\,. Note that this definition reduces to the one in \cite[Def. 5.10]{bojko2} when $E_X=\mathcal{O}_X$.

We obtain the following generalization of \cite[Thm. 5.15]{bojko2}.
\begin{corollary}
	\label{Cor:comp}
	Let $(S,E_S,\alpha_S)$ and $(X,E_X,\alpha_X)$ be such that
	\begin{enumerate}[(i)]
		\item $E_S$ is a torsion-free sheaf and $E_X$ is locally free and RS both of rank $e$,
		\item we have the equalities 
		$$
		\mu_X(E_X) = \mu_S(E_S)\,,\qquad c_3(X)\cdot c_1(\alpha_X) = c_1(S)\cdot c_1(\alpha_S)\,,
		$$
		\item $S$ satisfies $c_1(S)^2 =0$,
	\end{enumerate}
	then the generating series of invariants for $Y=S,X$ are related by the universal transformation
	$$
	Z_{E_X}\big(f,g;\alpha_X,q\big)=U_e\big(Z_{E_S}(f,h;\alpha_S,q)\big)\,,
	$$
	where $h(x) = g(x)g(-x)$\,.
	Moreover, if
	$$
	K_{E_Y}(f,\alpha;q) =\sum_{n\geq 0}q^n\chi^{\vir}\big(f(\alpha^{[n]})\big)\,,
	$$
	for $Y=X,S$, then 
	$$
	K_{E_X}(f,\alpha_X;q) = U_e\Big(K_{E_S}(f,\alpha_S;q)\Big)\,.
	$$
\end{corollary}
\begin{proof}
	The proof follows immediately from comparing the two generating series using \eqref{eqfagT2} and \eqref{eqfaCY}.
\end{proof}
Combining Corollary \ref{Cor:comp} and Theorem \ref{thm4fo}, we obtain the following result:
\begin{theorem}
	Let $E,F\to X$ be locally free sheaves and $\alpha\in K^0(X)$, then the following dualities hold whenever the invariants are defined:
	\begin{align*}
	S_{E,\alpha}\big((-1)^eq\big)&=V_{E,\alpha}(q)\,,\\
	S_{E,F}\big((-1)^eq\big) &= S_{F,E}\big((-1)^fq\big)\,,\\
	V_{E,F}(q) &= V_{F,E}(q)\,.
	\end{align*}
\end{theorem}
\subsection{Twelve-fold correspondence}
\label{secTwF}
We study here the invariants for curves $C$ using the previous results for surfaces in §\ref{sec: surfaces}. The main goal is to obtain the lower floor of Figure \ref{fig1}. Therefore for any curve $C$ define on $\QuotC$ the following K-theory class:
$$
\mathcal{D}_n =\textnormal{det}^{\frac{1}{2}}\big(\underline{\textnormal{Hom}}_{Q_C}(\mF,\mF\otimes \Theta)\big)\,.
$$
Note that by arguments along the lines of \cite[proof of Thm. 15]{AJLOP} for a smooth canonical curve $C\stackrel{i}{\hookrightarrow} S, E_S$ on $S$ and $E_C = i^*(E_S)$, we can show
$$
\underline{\textnormal{Hom}}_{Q_C}(\mathcal{F},\mathcal{F}\otimes \Theta) = N^{\vir}\,,
$$
where $N^{\vir}$ is the virtual normal bundle of $\QuotC \hookrightarrow \Quot$. Recall that when it comes to curves, we work with the twisted Verlinde series
$$
S_{E_C,\alpha}\big((-1)^eq\big)=\sum_{n\geq 0}q^n\chi\big(\det(\alpha^{[n]})\cdot \mD_n\big)\,.
$$
\begin{theorem}
	Let $E,F\to Y$ be torsion-free sheaves and $\alpha\in K^0(Y)$ for $Y=C,S,X$, then the following dualities hold whenever the invariants are defined:
	\begin{align*}
	S_{E,\alpha}\big((-1)^eq\big)&=V_{E,\alpha}(q)\,,\\
	S_{E,F}\big((-1)^eq\big) &= S_{F,E}\big((-1)^fq\big)\,,\\
	V_{E,F}(q) &= V_{F,E}(q)\,.
	\end{align*}
\end{theorem}
\begin{proof}
	Using the arguments of Ellingsgrud--Göttsche--Lehn \cite{EGL}, it is sufficient to compute the Euler characteristics of the form
	\begin{equation}
	\label{Eq: twistedcharacteristicC}
	\chi\big(\QuotC, f(\alpha^{[n]})\otimes \mD_n\big)
	\end{equation}
	for the case
	\begin{equation}
	\label{eqEC}
	C=\mathbb{P}^1\,,\qquad E_C = \mO_{\mathbb{P}^1}(d_1)\oplus\cdots \oplus \mO_{\mathbb{P}^1}(d_e)\,.
	\end{equation}
	In particular, we take $i: C\hookrightarrow S$ to be the inclusion of the exceptional rational curve in the blow-up of a K3 surface at a point and denote the corresponding exceptional divisor by $D$. We also take
	\begin{equation}
	\label{eqES}
	E_S = \mO_{S}(-d_1D)\oplus \ldots\oplus \mO_{S}(-d_e D)\,,\quad \textnormal{such that}\quad E_C = i^*E_S\,. 
	\end{equation}
	Using this splitting, we have have a $\big(\CC^*\big)^e$ action on $\textnormal{Quot}_C(E_C,n)$, $\textnormal{Quot}_S(E_S,n)$, $\textnormal{Quot}_C(\mO^{e}_C,n)$, $\textnormal{Quot}_S(\mO^{e}_S,n)$ inducing fixed point loci
	\begin{align*}¨
	\label{eqFpL}
	G_{E_S} &= \bigsqcup\limits_{n_1+\cdots +n_k = n}\prod_{i=1}^e\textnormal{Quot}_{S}\big(\mO_{S}(-d_iD),n_i\big)\xhookrightarrow{j_S}\Quot\,,\\
	G_{E_C} &= \bigsqcup\limits_{n_1+\cdots +n_k = n}\prod_{i=1}^e\textnormal{Quot}_{C}\big(\mO_{\PP^1}(d_i),n_i\big)\xhookrightarrow{j_C}\QuotC\,, \\
	G_{S} &=\bigsqcup\limits_{n_1+\cdots +n_k = n}\prod_{i=1}^eS^{[n_i]}\xhookrightarrow{j_{S,\textnormal{tw}}}\QuotCC\,,\\
	G_{C} &=\bigsqcup\limits_{n_1+\cdots +n_k = n}\prod_{i=1}^e\mathbb{P}^{n_i}\xhookrightarrow{j_{C,\textnormal{tw}}}\QuotCCC\,.
	\numberthis
	\end{align*}
	We then have the commutative diagram
	\begin{equation*}
	\begin{tikzcd}
	\QuotC\arrow[r,"i"]&\Quot\\
	\arrow[u,"j_C"]G_{E_C}\arrow[r,"i_f"]&\arrow[u,"j_S"]G_{E_S}\\
	\arrow[u,"\textnormal{tw}_C"]G_C\arrow[r,"i_{f,\textnormal{tw}}"]\arrow[d,"j_{\textnormal{tw}}"]&\arrow[u,"\textnormal{tw}_S"]G_S\arrow[d,"j_{S,\textnormal{tw}}"]\\
	\QuotCCC\arrow[r,"i_{\textnormal{tw}}"]&\QuotCC
	\end{tikzcd} \,.
	\end{equation*}
	where $\textnormal{tw}_C$ resp. $\textnormal{tw}_S$ are induced by tensoring the universal pairs by a line bundle, i.e. inducing the map on $\CC$-points:
	\begin{align*}
	\big[\mO_C\to \mO_Z\big]&\mapsto \big[\mO_C(d_i)\to \mO_Z\otimes \mO_{C}(d_i)\big]\,,\\
	[\mO_S\to \mO_Z]&\mapsto [\mO_S(-d_iD)\to \mO_Z\otimes \mO_{C}(-d_iD)]\,.
	\end{align*}
	and we used $\big(\mathbb{P}^1\big)^{[n]} = \mathbb{P}^n$. The horizontal maps $i,i_f,i_{f,\textnormal{tw}}$ are the maps induced by the inclusion $i: C\hookrightarrow S$ and $j_{Y}, j_{Y,\text{tw}}$ are the inclusions of fixed point loci for $Y=C,S$. We will now generalize the proof of \cite[Thm. 15]{AJLOP} to include the non-trivial vector bundle.
	
	Denoting by $N_{(-)}$ the normal bundles of any of the fixed point loci in \eqref{eqFpL}, recall that we have the following localization formulae 
	\begin{align*}
	\mO^{\vir}_{\Quot}=(j_S)_*\frac{\mO^{\vir}_{G_{E_S}}}{\Lambda_{-1}N_{G_{E_S}}^\vee}\,,\qquad \mO_{\QuotC} = (j_C)_*\frac{\mO_{G_{E_C}}}{\Lambda_{-1}N^\vee_{G_{E_C}}}\,,\\
	\mO^{\vir}_{\QuotSC}=(j_{S,\tw})_*\frac{\mO^{\vir}_{G_{S}}}{\Lambda_{-1}N_{G_{S}}^\vee}\,,\qquad \mO_{\QuotCCC} = (j_{C,\tw})_*\frac{\mO_{G_{C}}}{\Lambda_{-1}N^\vee_{G_{C}}}\,,
	\end{align*}
	We want to show that
	$$
	\frac{\mO^{\vir}_{G_{E_S}}}{\Lambda_{-1}N_{G_{E_S}}^{\vee}}=(-1)^ni_{f\,*}\frac{j^*_C\,\mathcal{U}^{\frac{1}{2}}_n}{\Lambda_{-1}N^\vee_{G_{E_C}}}\,.
	$$
	to compare the expression \eqref{Eq: twistedcharacteristicC} to the invariants of the form 
	$$
	\chi\big(\QuotS, \mO^\vir\otimes f(\beta^{[n]})\big)
	$$
 for a K-theory class $\beta$ on $S$.
	Let $\mathcal{F}_i$ denote the quotients on $C\times \textnormal{Quot}_C\big(\mO(d_i),n_i\big)$ or $C\times \mathbb{P}^{n_i}$ of the fixed point locus and $\Theta\cong \mO_{\mathbb{P}_1}(-1)$, then by the same computation as in Oprea--Pandharipande \cite[Lem. 34]{OP1} have the formula
	$$
	i^*N_{G_{E_S}} - N_{G_{E_C}} = \bigoplus_{i<j}\Big(U_{i,j} + U_{i,j}^\vee[-1]\Big)\,,  \quad \textnormal{where}\quad U_{i,j} = \underline{\textnormal{Hom}}_{G_{E_C}}\big(\mathcal{F}_i,\mathcal{F}_j\otimes \Theta\big)
	$$
	implying that 
	$$
	\frac{1}{i^*\Lambda_{-1}N_{G_{E_S}}^\vee} =\frac{1}{\Lambda_{-1}N^\vee_{G_{E_C}}} \prod_{i<j}\textnormal{det}\big(U_{i,j}\big)\,.
	$$
	Taking the short exact sequence 
	$$
	0\longrightarrow I_i\longrightarrow \mO(-d_iD)\longrightarrow Q_i\longrightarrow 0
	$$
	for a point in $\text{Quot}_{S}(\mO(-d_iD),n_i)$, we note that $\mO^{\vir}_{G_{E_S}} = \Lambda_{-1}\textnormal{Obs}_i^\vee $ and $\textnormal{Obs}_i \cong \textnormal{Ext}^1(I_i,Q_i)\cong \textnormal{Ext}^2(Q_i,Q_i)$ which implies in particular that $\textnormal{tw}^*_S\big(\textnormal{Obs}^{\vee}_i\big) = \textnormal{Obs}^{\vee}_i$. This leads to
	\begin{align*}
	\textnormal{tw}^*_C\bigg(i_{f\,*}\bigg(\frac{\bigotimes_i\Lambda_{-1}\textnormal{Obs}_i^\vee}{\Lambda_{-1}N^\vee_{G_{E_S}}}\bigg)\bigg) &= i_{f\,*}\bigg(\frac{\bigotimes_{i}\Lambda_{-1}\textnormal{Obs}^\vee_i}{\textnormal{tw}^*_{S}\big(\Lambda_{-1}N^{\vee}_{G_{E_S}}\big)}\bigg)= \frac{i_{f\,*}\big(\bigotimes_i \Lambda_{-1}\textnormal{Obs}^\vee_i\big)}{\textnormal{tw}^*_C\big(i^*\Lambda_{-1}N^\vee_{G_{E_S}}\big)}\\
	&=\frac{i_{f\,*}\Big(\bigotimes_i\Lambda_{-1}\textnormal{Obs}_i^\vee\Big)}{\textnormal{tw}^*_C\Big(\Lambda_{-1}N^\vee_{G_{E_C}}\Big)}\cdot \textnormal{tw}^*_C\bigg(\prod_{i<j}\textnormal{det}\big(U_{i,j}\big)\bigg)\\
	&=\frac{i_{f\,*}\Big(\bigotimes_i\Lambda_{-1}\textnormal{Obs}_i^\vee\Big)}{\Lambda_{-1}N^\vee_{G_{C}}}\cdot \bigg(\prod_{i<j}\textnormal{det}\big(U_{i,j}\big)\bigg)\\
	&=\frac{j_C^*\Big(\mathcal{U}^{\frac{1}{2}}_n\Big)}{\Lambda_{-1}N^\vee_{G_C}} = \textnormal{tw}^*_C\bigg(\frac{j^*_C\,\mathcal{U}^{\frac{1}{2}}_n}{\Lambda_{-1}N^\vee_{G_{E_C}}}\bigg)\,,
	\end{align*}
	where the second to last step of the computation uses the result of \cite[(35)]{AJLOP} stating that 
	$$
	j^*_C(\mU_n) = \det \uHom_{G_C}\Big(\sum_{i=1}^e\mF_i, \sum_{i=1}^e\mF_j\otimes \Theta\Big)\,.
	$$
	
	We, therefore, see that
	$$
	\chi\big(Q_S, \mO^{\vir}\otimes \textnormal{det}\big(\alpha_S^{[n]}\big)\big) = \chi\big(Q_C, \textnormal{det}\big(\alpha_C^{[n]}\big)\otimes \mD_n\big) \,.
	$$
	Following similar arguments and \cite[Lemma 3.4]{OP1}, we may also conclude 
	$$
	V(\alpha_C;q) =V(\alpha_S;q)
	$$
	for $\alpha_C = \alpha_S|_C$. The statement of the theorem then follows from Theorems \ref{thm4fo}.
\end{proof}
\appendix
\section{Stability conditions and assumptions} 
\label{App: checking of assumptions}
We describe here a slightly more general setup than we need to cover the wall-crossing of Theorem \ref{thmCS} which includes all sheaves on curves and surfaces. For a K-theory class $\alpha$ we denote by 
\begin{itemize}
    \item $\dim(\alpha)$ the dimension of support of its Chern character,
    \item $\rnk(\alpha) = [q^{\dim(\alpha)}]\big\{P_{\alpha}\big\}$ for the Hilbert polynomial $P_{\alpha}$ of $\alpha$ and $p_{\alpha} =\frac{P_\alpha}{\rnk(\alpha)}$ 
\end{itemize}
 When $\llbracket F\rrbracket=\alpha$, we also write $p_F = p_{\alpha}$.

\begin{definition}
\label{Def:stability}
    Fix $Y$ of dimension $d\leq 2$ with a torsion-free sheaf $E\to Y$ and a monic  polynomial $p$. Define an exact category $\acute{\mA}_{E,p}$ to consist of the pairs
    $$
    V\otimes E\longrightarrow F\,,\qquad \textrm{s.t.} \quad p_F = p,\quad  \dim(F) =s
    $$
    for $F$ being $\mu$-semistable or $F=0$. We work with the K-group 
    $$
    K(\acute{\mA}_{E,p} )  = \ZZ\times \{\alpha\in K^0(Y) : p_{\alpha} = p\}
    $$
    admitting a map from the Grothendieck group of $\acute{\mA}_{E,p}$: $K^0(\acute{\mA}_{E,p})\to  K(\acute{\mA}_{E,p} )$
    and a continuous family of stability conditions $ \tau^t:  K(\acute{\mA}_{E,p} ) \to \HH$\footnote{We use here the standard notation for the upper halfplane without the negative real line.} for $t\in [0,1)$:
    \begin{equation}
    \label{Eq: stability}
    \tau^t(r,\alpha) = \begin{cases}
   r e^{i\pi t}
    &\textrm{if}\quad r\neq 0\\
    \rnk(\alpha) e^{i\frac{\pi}{2}}&\textrm{if}\quad r= 0
    \end{cases}\,,
    \end{equation}
    where $r$ is the dimension of $V$.
    \begin{figure}[h]
        \centering
    \includegraphics[scale=0.15]{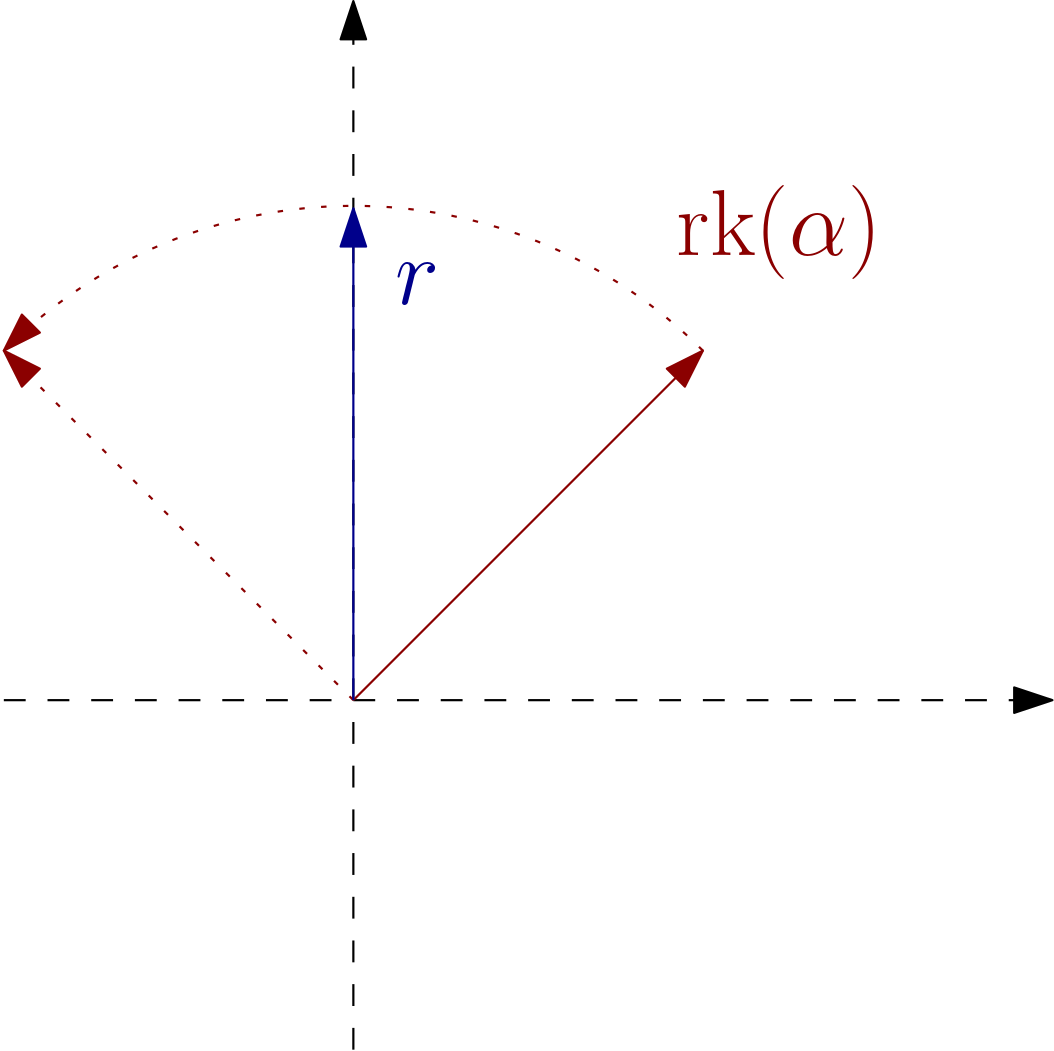}
    \end{figure}
    \end{definition}
    
We described essentially only three stability conditions. Denote by $(d,s) =\big(\dim(Y),\dim(\alpha)\big)$, then the semistable pairs are described as follows:
\begin{enumerate}
    \item when $t\in [0, 1/2)$, then for $(d,s) = (1,1), (2,1)$, the stability condition $\tau^t$ is equivalent to $\bar{\tau}^{0}_1$-stability of Joyce \cite[(5.21), Example 5.6]{JoyceWC}  which we will call \textit{Bradlow stability}. Thus if $r = 1$, then the pair $E\xrightarrow{m} F$ satisfies that 
    \begin{itemize}
        \item F is Gieseker-semistable
        \item if $F'\subset F$ is a strict subsheaf containing the image of $m$, then $p_{F'}< p_{F}$.
    \end{itemize}
   For $(d,s) = (1,0), (2,0)$, the stable pairs are given by surjective maps $E\to F$ to a zero-dimensional sheaf $F$. In all of these cases, we denote $P^{\JS}_{\alpha}$ the resulting moduli space of stable pairs. 
    \item When $t=1/2$, $\tau^0$ is trivial. 
    \item When $t\in (1/2,1)$ only $V\otimes E\to 0$ and $0\to F$ are semistable. 
\end{enumerate}

The construction of stacks in \cite[Def. 8.4]{JoyceWC} can be adapted to our case (also extending our Definition \ref{defAMTH})  and all the corresponding data needed for the construction of vertex algebras with it. They are described in  \cite[Def. 8.2.1]{JoyceWC} for a line bundle $L$, but this works for any $E$. We have already seen how to define the data in Definition \ref{defAMTH} when $F$ is zero-dimensional. 

We only recall that the stack is constructed by starting from 
$$\mV\to  [*/\text{GL}(d,\CC)]\,,\qquad \mF \quad \text{on}\quad  \mM^{\text{ss}}_{\alpha}\times Y$$
the universal vector bundle and the universal sheaf, respectively. Our stack $\acute{\mM}_{(d,\alpha)}$ is defined to parametrize the morphisms $ \pi^*_{1}(\mV)\boxtimes E\to \pi^*_{2,3}(\mF)$ on $[*/\text{GL}(d,\CC)] \times \mM^{\text{ss}}_{\alpha}\times Y$ and exists by Grothendieck--Dieudonn{\'e} \cite[§III.7.7.8-§III.7.7.9]{EGA}. We then take a union over all $(d,\alpha)\in K(\acute{\mA}_{E,p})$ additionally satisfying $\dim(\alpha)=s$ to get $\acute{\mM}_{E,p}$. We denote by 
\begin{equation}
\label{Eq:pi1pi2}
\pi_{1}: \acute{\mM}_{(d,\alpha)} \to [*/\text{GL}(d,\CC)] \,,\qquad \pi_{2}: \acute{\mM}_{(d,\alpha)} \to \mM^{\sms}_{\alpha}
\end{equation}

\begin{definition}
    We say that $E$ is \textit{sufficiently negative} with respect to a fixed $p$ if 
    $$
    \Ext^i\big(E,F\big) = 0,\qquad \textrm{whenever}\, i>1
    $$
    for any semistable $F$ with $p_F = p$. In this case, we denote the moduli space of $\tau^t$-stable pair of class $(1,\alpha)$ by $P^{\JS}_{\alpha}$ when $t<\frac{1}{2}$ and by $[P^{\JS}_{\alpha}]^\vir$ the corresponding virtual fundamental class.
\end{definition}
\begin{remark}
\begin{enumerate}
\item   Notice that for zero-dimensional sheaves $F$ on a surface $X=S$ this condition is satisfied by $\Ext^2(E,F) = \Ext^0(F,E)\otimes \omega_S)^*=0$. 
\item
;   ;       In our set-up, we are forced to work with Gieseker stability conditions instead of $\mu$-stability. This is a consequence of needing some control over lower-dimensional quotients whenever $G\hookrightarrow F$ is an injective map of sheaves of non-zero rank with $\mu(G) = \mu(F)$. Without the additional control over the lower degree coefficients of the Gieseker polynomial, we would not be able to use the stability condition in \eqref{Eq: stability} that sees only $\rnk(\alpha)$ without violating some of the assumptions in the proof of Theorem \ref{Thm:PJS}.
    \end{enumerate}
\end{remark}

The consequence of varying the stability condition from $t= 0$ to $t=1-\epsilon$ is the following wall-crossing formula.
\begin{theorem}
\label{Thm:PJS}
Let $E$ be sufficiently large for a fixed $p_{\alpha}$, then the invariants $[P^{\JS}_{\alpha}]^\vir$ are determined by 
    $$
    [P^{\JS}_{\alpha}]^\vir = \sum_{\begin{subarray}a \underline{\alpha} \vdash \alpha\\ p_{\alpha_i}=p_{\alpha}\end{subarray}}\frac{1}{l!}\, \big[[M_{\alpha_1}]^\inva, \big[\ldots, \big[[M_{\alpha_l}]^\inva, e^{(1,0)}\big]\ldots\big]\big]\,, 
    $$
    where $[-,-]$ is the  Lie bracket on $\check{H}_*(\acute{\mM}_{E,p})$ and $e^{(1,0)}$ is the class of a point $\{(E, 0)\}$ in the component $\acute{\mM}_{(1,0)}$. 
\end{theorem}
\begin{proof}
  To prove this, we need to check Assumptions \cite[Ass. 5.1 - 5.3]{JoyceWC} for the classes in $$\mP(\acute{\mA}_{E,p}) = \{(e,\alpha)\in K(\acute{\mA}_{E,p}): e=0,1 \}\,.$$ Most of the assumptions are easy enough to check, so we only give a brief summary of them to pin down the ones that need some additional care.
  \begin{enumerate}
      \item Assumptions 5.1 (a)-(c) are meant to lift the data necessary for constructing vertex algebras to derived stacks. They are addressed in \cite[§8.2.2]{JoyceWC} for $E=L$ but again extended to our setting.
      \item Assumptions 5.1 (d) - (f) are related to the existence of perfect obstruction theories on the stack $\acute{\mM}_{E,p}$ and their compatibilities with the vertex algebra.  In our case, it is given by 
      $$
      \Ob\to \LL_{\acute{\mM}_{(d,\alpha)}}\,,
      $$
    where 
    \begin{equation}
    \label{Eq:Ob}
    \Ob = \pi_{2,3\,*}\Big(\pi_1^*(\mV)\otimes E\to ( \pi_2\times \id_Y)^*\mF, (\pi_2\times \id_Y)^*\mF \Big)^\vee\,.
    \end{equation}
      In this case, the obstruction theories are allowed to have degree $-1$ contributions as explained in Poma's \cite[Def. 3.1]{Po14} and \cite[Def. 2.11, 3.6]{JoyceWC}. The next lemma proves assumption  \cite[Ass 5.1 (f) ii. ]{JoyceWC} \footnote{Note that the quasi-smoothness assumption Joyce uses,  requires that the derived enrichment of $\mathcal{P}^t_{\alpha}$ has cotangent complex of tor-amplitude $[-1,1]$. The perfect obstruction theory is the restriction of this complex to the classical truncation, so it is equivalent.}:
\begin{lemma}
\label{Lem:obs}
The  pair obstruction theory given by \eqref{Eq:Ob} on $\acute{\mM}_{(1,\alpha)}$ is perfect of tor-amplitude $[-1,1]$.
\end{lemma}
\begin{proof}
Setting the notation $I = [E\to F]$, the cohomologies of the obstruction theory at this point become 
$$
\Ext^i(I,F)\,.
$$
Using the long exact sequence
$$\Ext^{i}(F, F)\to \Ext^i(E, F)\to\Ext^{i}(I, F)\to \Ext^{i+1}(F, F)\to \Ext^{i+1}(E,F)\,.$$
 the vanishing of $\Ext^i(I,F)$ for $i\geq 2$ follows from $\Ext^2(E,F) = 0$. 
\end{proof}
The rest of the said assumptions focus on the isomorphism
$$
\Delta^*\Theta^{E_Y,\pa}\cong \Ob
$$
that holds naturally in our setting, together with the compatibility diagram  with respect to taking direct sums that is already proved in \cite[Def. 8.15]{JoyceWC}.
      \item Assumptions 5.1 (g) describe the necessary data for the construction of the quiver-pairs leading to (enhanced) master spaces. In the case of sheaves the data $(\mB_k, Q_k, \lambda_k)$ of this assumptions is given for every $k\geq 0$ by 
$$
\mB_k  = \big\{F \in \Coh(X) : H^2\big(F(k)\big) = 0\big\}\,,\quad Q_k(E) = H^0\big(E(k)\big) \,, 
$$
and $\lambda_k(E) =  \dim\big(Q_k(E)\big)$. In our case, \cite[8.1.1]{JoyceWC} uses a fully faithful embedding 
$$
G:\acute{\mA}_{E,p}\longrightarrow \Coh^{\text{SU}(2)}(X\times \PP^1)
$$
of García--Prada \cite[Prop. 3.9]{GP94} where the $\text{SU}(2)$ in superscript denotes the $\text{SU}(2)$ equivariant sheaves. Using the corresponding data on sheaves described above, we get by composition $\acute{Q}_k =Q_k\circ G$ and $\acute{\lambda}_k=\lambda_k\circ G$ the necessary data
$$
\{(\acute{B}_k, \acute{Q}_k,\acute{\lambda}_k)\}_{k\in \ZZ}
$$
for $\acute{\mA}_{E,p}$. Here $\acute{\mB}_k$ is preimage of $B_k\subset \Coh^{\text{SU}(2)}(X\times \PP^1)$ under $G$.
\item We postpone Assumption 5.1 (h) dealing with quiver pairs to the end of this section.
\item Assumptions 5.2 (a) requires the existence of Harder--Narasimhan filtrations for any object $P\in \acute{\mA}_{E,p}$ of class in $\mP(\acute{\mA}_{E,p})$. This is true by the description of $\tau^t$ for each $t$ we gave above.
  \item In (b), the condition of being $\tau^t$-(semi)stable is required to be an open condition on the moduli stack $\acute{\mM}_{E,p}$ and $\acute{\mM}^{\rig}_{E,p}$ leading to open, finite type substacks $\mM^{\text{st}}_{(d,\alpha)}, \mM^{\text{ss}}_{(d,\alpha)}\subset \acute{\mM}^{\rig}_{E,p}$ whenever $(d,\alpha)\in \mP(\acute{\mA}_{E,p})$. This is the consequence of this result for $\mu$-semistable sheaves \cite[§7.2.3]{JoyceWC} and Bradlow pairs \cite[Def. 5.5, §7.4]{JoyceWC} using our description of $\tau^t$ for $t\neq \frac{\pi}{2}$. In the case $t=\frac{\pi}{2}$ this follows because the projection $\pi_1\times \pi_2$ from \eqref{Eq:pi1pi2} are finite type with $[*/\GL(d,\CC)]$, $\mM^{\sms}_\alpha$ being also finite type.
\item In $(c)$, semistable summands $P_1,P_2$ with $\tau^t(P_1) =\tau^t(P_2)$ of an object $P$ with class in $\mP(\acute{\mA}_{E,p})$ need to be again in $\mP(\acute{\mA}_{E,p})$ which is trivially satisfied.
\item For condition (d), we define for any $(d,\alpha),(e,\beta)\in C(\acute{\mA}_{E,p})_{\pe}$ the function  $$\lambda^t_{(d,\alpha)}\big(e,\beta\big) = \big(t-\frac{1}{2}\big)(e-d)$$
which serves the purpose that if $\tau^{\frac{1}{2}}(d,\alpha) =\tau^{\frac{1}{2}}(e,\beta)$, then $$\tau^t(d,\alpha)\quad \Box\quad \tau^t(e,\beta)\iff \lambda^t_{(d,\alpha)}(e,\beta)\quad\Box\quad 0\,,$$
where $\Box$ stands for $<$ or $>$.
\item Assumption (g) requires that when there are no strictly $\mu^t$-semistables of class $(d,\alpha)\in \mP(\acute{\mA})$, then the moduli scheme of $\tau^t$-semistable sheaves is proper. This only needs to be checked for $t\neq 1/2$ in our case and was done in \cite[§7.2.3 and § 8.4]{JoyceWC} for sheaves and Bradlow pairs respectively. 
\item Finally, in our case the assumption 5.3 (a) the family of Definition \ref{Def:stability} is clearly continuous in the sense of \cite[Def. 3.15]{JoyceWC}.
\item The Assumption 5.3 (b)  reduces to requiring that the set of non-zero coefficients $U((\underline{d},\underline{\alpha}),\tau^{t_-}, \tau^{t_+})$ for the wall $t=1/2$ such that $\underline{d}\vdash 1, \underline{\alpha}\vdash \alpha $, $(d_i,\alpha_i), (d,\alpha)\in \mP(\acute{\mA}_{E,p}) $ is finite, but there is nothing to check here.
  \end{enumerate}
  The coefficients $1/l!$ were obtained by applying the arguments in Joyce--Song \cite[ Propositions 13.8 and Lemma 13.9]{JoyceSong} to their definition in \cite[Def. 3.10]{JoyceWC}.

 \paragraph{Assumption 5.1 (h)} 
Joyce constructs a category of quiver-pairs with new stability conditions. These are described as follows:  Let $Q =(V,E)$ be a quiver with distinguished sub-graphs
$
Q_0 =(V_0,E_0)\,,  Q_1=(V_1,\emptyset)\,,
$
such that if we have an edge $e\in E\backslash E_0$, then its origin $o(e)\in V_0$ and target $t(e)\in V_1$.
\begin{figure}[h]
\centering
\includegraphics[scale=0.40]{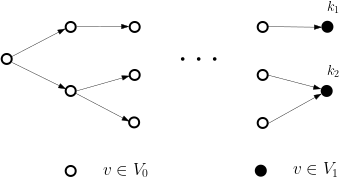}
\caption{Quiver-pairs}
\label{FigGenQu}
\end{figure}
Fixing a map $\underline{k}:V_1\to \ZZ$, the category $\acute{\mathcal{B}}_{\,Q,\underline{k}}$ is defined to have the objects
$
(V\otimes E\to F,V_Q,\phi_Q),
$
where 
$$P = V\otimes E\to F\in\bigcap_{k_{v_1}, v_1\in V_1} \acute{\mB}_{k_{v_1}}\,.$$
The map $V_{Q}$ assigns to every $v\in Q_0$ a vector space $V_v$ and $\phi_Q$ to each $e\in E_0$ a morphism of these vector spaces and to each $e\in E\backslash E_0$ the map
$$
\phi_{e}: V_{o(e)}\to \acute{F}_{k_{t(e)}}(P)\,.
$$

We then require that all the assumptions that we discussed thus far for $\acute{\mA}_{E,p}$ hold for the choice of stability condition constructed in \cite[(5.21)]{JoyceWC} out of our Definition \ref{Def:stability} and the set of permissible classes
$$
\mP(\acute{\mB}_{\underline{k}}) = \prod_{v_1\in V_1}\ZZ_{\geq 0}\times C(\acute{\mA}_{E,p})\,.
$$
Going through the proofs of \cite[§7.4]{JoyceWC}, one sees that the only non-trivial check is the one of properness of moduli spaces. We separate it into three cases:
  \begin{enumerate}
  \item For $t\in [0,1/2)$ properness follows from the comparison to Bradlow-stability $\bar{\tau}^0_1$  and the assumptions having been checked for it in \cite[§8.1.4,§8.2]{JoyceWC}.
  \item For $t=1/2$ this is just the result that the stack of semistables in class $(r,\alpha)$ is the full $\acute{\mM}_{(r,\alpha)}$. The projection $\pi_1\times \pi_2$  to $[*/\GL(d,\CC)]\times\mM^{\text{ss}}_{\alpha}$ satisfy the valuative criterion of properness (VCP). The arguments for the moduli stack $\mM^{\text{ss}}_{\alpha}$ in \cite[§7.4]{JoyceWC} show that VCP holds for it so we conclude it for $\acute{\mM}^{\sms}_{(d,\alpha)}$, $d=0,1$. Applying the procedure in \cite[§7.4]{JoyceWC}, this leads to the proof of properness also for quiver-pairs in Figure \ref{FigGenQu}
\end{enumerate}
\end{proof}
\printbibliography

@article{BJ,
author = "Borisov, D. and Joyce, D.",
doi = "10.2140/gt.2017.21.3231",
fjournal = "Geometry and Topology",
journal = "Geom. Topol.",
number = "6",
pages = "3231--3311",
publisher = "MSP",
title = "Virtual fundamental classes for moduli spaces of sheaves on Calabi–Yau four-folds",
volume = "21",
year = "2017",
eprint={1504.00690v2},
archiveprefix={arXiv}
}

@article{CGJ,
title = "Orientability of moduli spaces of Spin(7)-instantons and coherent sheaves on Calabi–Yau 4-folds",
journal = "Advances in Mathematics",
volume = "368",
year = "2020",
author = "Yalong Cao and Jacob Gross and Dominic Joyce",
eprint={arXiv:1811.09658},
archiveprefix={arXiv}
}

@book{DK,
  title={The Geometry of Four-manifolds},
  author={Donaldson, S.K. and Kronheimer, P.B.},
  isbn={9780198502692},
  lccn={89077530},
  series={Oxford mathematical monographs},
  year={1990},
  publisher={Clarendon Press}
}

@misc{GJT,
AUTHOR = {Gross, Jacob and Joyce, Dominic and Tanaka, Yuuji},
     TITLE = {Universal structures in {$\Bbb C$}-linear enumerative
              invariant theories},
   JOURNAL = {SIGMA Symmetry Integrability Geom. Methods Appl.},
  FJOURNAL = {SIGMA. Symmetry, Integrability and Geometry. Methods and
              Applications},
    VOLUME = {18},
      YEAR = {2022},
     PAGES = {Paper No. 068, 61},
    eprint={2005.05637},
    archivePrefix={arXiv},
}

@book {Hatcher,
    AUTHOR = {Hatcher, Allen},
     TITLE = {Algebraic topology},
 PUBLISHER = {Cambridge University Press, Cambridge},
      YEAR = {2002},
      url="http://pi.math.cornell.edu/~hatcher/AT/AT.pdf"
}

@article{JoyceSong,
    AUTHOR = {Joyce, Dominic and Song, Yinan},
     TITLE = {A theory of generalized {D}onaldson-{T}homas invariants},
  JOURNAL = {Memoirs of the American Mathematical Society},
    VOLUME = {217},
      YEAR = {2012},
    NUMBER = {1020},
      archivePrefix="arxiv",
  eprint="0810.5645v6",
}

@article{JTU,
title = "On orientations for gauge-theoretic moduli spaces",
journal = "Advances in Mathematics",
volume = "362",
year = "2020",
author = "Dominic Joyce and Yuuji Tanaka and Markus Upmeier",
      eprint={1811.01096},
      archivePrefix={arXiv},
      pages ={106957}
}

@book{MayPonto,
  title={More concise algebraic topology: localization, completion, and model categories},
  author={May, J Peter and Ponto, Kate},
  year={2011},
  publisher={University of Chicago Press}
}

@article {Nekrasov1,
    AUTHOR = {Nekrasov, Nikita},
     TITLE = {Magnificent four},
   JOURNAL = {Ann. Inst. Henri Poincar\'{e} D},
  FJOURNAL = {Annales de l'Institut Henri Poincar\'{e} D. Combinatorics, Physics
              and their Interactions},
    VOLUME = {7},
      YEAR = {2020},
    NUMBER = {4},
     PAGES = {505--534},
           eprint={1712.08128},
    archivePrefix={arXiv},
}

@article{Nekrasov2,
  title={Magnificent four with colors},
  author={Nekrasov, Nikita and Piazzalunga, Nicol{\`o}},
  journal={Communications in Mathematical Physics},
  volume={372},
  number={2},
  pages={573--597},
  year={2019},
  publisher={Springer}
}

@misc{OT,
    title={Counting sheaves on Calabi--Yau 4-folds, I},
    author={Jeongseok Oh and Richard P. Thomas},
    year={2020},
    eprint={2009.05542},
    archivePrefix={arXiv},
}

@inproceedings{TVaq,
  title={Moduli of objects in dg-categories},
  author={To{\"e}n, Bertrand and Vaqui{\'e}, Michel},
  booktitle={Annales scientifiques de l'Ecole normale sup{\'e}rieure},
  volume={40},
  number={3},
  pages={387--444},
  year={2007},
  eprint={math/0503269v5},
  archivePrefix={arXiv}
  
}

@misc{JoyceWC,
      title={Enumerative invariants and wall-crossing formulae in abelian categories}, 
      author={Dominic Joyce},
      year={2021},
      eprint={2111.04694},
      archivePrefix={arXiv}
}

@online{Joycehall,
title="Ringel–Hall style Lie algebra structures on the homology of moduli spaces",
author="Joyce, D.",
url = "https://people.maths.ox.ac.uk/joyce/hall.pdf",
year = "2019",
note ="preprint"
}

@misc{gross,
    title={The homology of moduli stacks of complexes},
    author={Jacob Gross},
    year={2019},
    eprint={1907.03269},
    archivePrefix={arXiv},
}

@book{KacVA,
  title={Vertex Algebras for Beginners},
  author={Kac, V.G.},
  isbn={9780821882696},
  lccn={98041276},
  series={University lecture series},
  publisher={American Mathematical Society}
}

@book{LLVA,
  title={Introduction to Vertex Operator Algebras and Their Representations},
  author={Lepowsky, J. and Li, H.},
  isbn={9780817634087},
  lccn={2003063839},
  series={Introduction to Vertex Operator Algebras and Their Representations},
  year={2004},
  publisher={Birkh{\"a}user}
}

@inbook{dold, place={Cambridge}, series={London Mathematical Society Lecture Note Series}, title={Relations between ordinary and extraordinary homology}, DOI={10.1017/CBO9780511662584.015}, booktitle={Algebraic Topology: A Student's Guide}, publisher={Cambridge University Press}, author={Dold, Albrecht and Adams, J. F. and Shepherd, G. C.}, year={1972}, pages={166–177}, collection={London Mathematical Society Lecture Note Series}}

@misc{bojko,
      title={Orientations for DT invariants on quasi-projective Calabi-Yau 4-folds}, 
      author={Arkadij Bojko},
      year={2020},
      eprint={2008.08441},
      archivePrefix={arXiv},
            journal = {Advances in Mathematics},
volume = {388},
pages = {107859}
}

@book {BZVA,
    AUTHOR = {Frenkel, Edward and Ben-Zvi, David},
     TITLE = {Vertex algebras and algebraic curves},
    SERIES = {Mathematical Surveys and Monographs},
    VOLUME = {88},
   EDITION = {Second},
 PUBLISHER = {American Mathematical Society, Providence, RI},
      YEAR = {2004},
}

@misc{CKM,
      title={$K$-theoretic DT/PT correspondence for toric Calabi-Yau 4-folds}, 
      author={Yalong Cao and Martijn Kool and Sergej Monavari},
      year={2020},
      eprint={1906.07856},
      archivePrefix={arXiv},
}

@article{Blanc,
  title={Topological K-theory of complex noncommutative spaces},
  author={Anthony Blanc},
  journal={Compositio Mathematica},
  year={2016},
  volume={152},
  pages={489-555},
     eprint = "1211.7360",
    archivePrefix = "arXiv",
}

@article{LiTian,
  title={Virtual moduli cycles and Gromov-Witten invariants of algebraic varieties},
  author={Li, Jun and Tian, Gang},
  journal={Journal of the American Mathematical Society},
  volume={11},
  number={1},
  pages={119--174},
  year={1998},
     eprint={alg-geom/9602007},
    archivePrefix={arXiv}
}

@article{BF,
   title={The intrinsic normal cone},
   volume={128},
   journal={Inventiones Mathematicae},
   publisher={Springer Science and Business Media LLC},
   author={Behrend, K. and Fantechi, B.},
   year={1997},
   pages={45–88},
       eprint={alg-geom/9601010},
    archivePrefix={arXiv},
}

@article{Lehn,
  title={Chern classes of tautological sheaves on Hilbert schemes of points on surfaces},
  author={Lehn, Manfred},
  journal={Inventiones mathematicae},
  volume={136},
  number={1},
  pages={157--207},
  year={1999},
  publisher={Springer},
      eprint={math/9803091},
    archivePrefix={arXiv},
}

@misc{EGL,
   AUTHOR = {Ellingsrud, Geir and G\"{o}ttsche, Lothar and Lehn, Manfred},
     TITLE = {On the cobordism class of the {H}ilbert scheme of a surface},
  JOURNAL = {Journal of Algebraic Geometry},
    VOLUME = {10},
      YEAR = {2001},
    NUMBER = {1},
     PAGES = {81--100},
    year={1999},
    eprint={math/9904095},
    archivePrefix={arXiv},
}

@article {Borcherds,
	author = {Borcherds, Richard E.},
	title = {Vertex algebras, Kac-Moody algebras, and the Monster},
	volume = {83},
	number = {10},
	pages = {3068--3071},
	year = {1986},
	doi = {10.1073/pnas.83.10.3068},
	publisher = {National Academy of Sciences},
	journal = {Proceedings of the National Academy of Sciences}
}

@book {FLM,
    AUTHOR = {Frenkel, Igor and Lepowsky, James and Meurman, Arne},
     TITLE = {Vertex operator algebras and the {M}onster},
    SERIES = {Pure and Applied Mathematics},
    VOLUME = {134},
 PUBLISHER = {Academic Press, Inc., Boston, MA},
      YEAR = {1988},
}

@misc{KRdraft,
author={Kool, Martijn and Rennemo, Jorgen},
    title={In preparation},
}

@article{NP,
  title={Magnificent four with colors},
  author={Nekrasov, Nikita and Piazzalunga, Nicol{\`o}},
  journal={Communications in Mathematical Physics},
  volume={372},
  number={2},
  pages={573--597},
  year={2019},
  publisher={Springer},
     eprint={1808.05206},
    archivePrefix={arXiv},
}

@article{MOP1,
 AUTHOR = {Marian, Alina and Oprea, Dragos and Pandharipande, Rahul},
     TITLE = {Segre classes and {H}ilbert schemes of points},
   JOURNAL = {Ann. Sci. \'{E}c. Norm. Sup\'{e}r. (4)},
  FJOURNAL = {Annales Scientifiques de l'\'{E}cole Normale Sup\'{e}rieure. Quatri\`eme
              S\'{e}rie},
    VOLUME = {50},
      YEAR = {2017},
    NUMBER = {1},
     PAGES = {239--267},
  eprint={arXiv:1507.00688},
  archiveprefix={arXiv},
}

@misc{AJLOP,
    title={The virtual K-theory of Quot schemes of surfaces},
    author={Noah Arbesfeld and Drew Johnson and Woonam Lim and Dragos Oprea and Rahul Pandharipande},
    year={2020},
    eprint={2008.10661},
    archivePrefix={arXiv},
     JOURNAL = {Journal of Geometry and Physics},
    VOLUME = {164},
      YEAR = {2021},
}

@misc{OP1,
    title={Quot schemes of curves and surfaces: virtual classes, integrals, Euler characteristics},
    author={Dragos Oprea and Rahul Pandharipande},
    year={2019},
    eprint={1903.08787},
    archivePrefix={arXiv},
}

@article{ON,
  title={Membranes and sheaves},
  author={Nekrasov, Nikita and Okounkov, Andrei},
  journal={arXiv preprint arXiv:1404.2323},
  year={2014}
}

@misc{DYS,
      AUTHOR = {Diaconescu, Duiliu-Emanuel and Sheshmani, Artan and Yau,
              Shing-Tung},
     TITLE = {Atiyah class and sheaf counting on local {C}alabi {Y}au
              fourfolds},
  FJOURNAL = {Advances in Mathematics},
    VOLUME = {368},
      YEAR = {2020},
     PAGES = {107132, 54},
      eprint={1810.09382},
      archivePrefix={arXiv},
}

@book {mochizuki,
    AUTHOR = {Mochizuki, Takuro},
     TITLE = {Donaldson type invariants for algebraic surfaces},
    SERIES = {Lecture Notes in Mathematics},
    VOLUME = {1972},
      NOTE = {Transition of moduli stacks},
 PUBLISHER = {Springer-Verlag, Berlin},
      YEAR = {2009},
}

@incollection {DT,
    AUTHOR = {Donaldson, S. K. and Thomas, R. P.},
     TITLE = {Gauge theory in higher dimensions},
 BOOKTITLE = {The geometric universe ({O}xford, 1996)},
     PAGES = {31--47},
 PUBLISHER = {Oxford Univ. Press, Oxford},
      YEAR = {1998},
}

@article{T,
  title={A holomorphic Casson invariant for Calabi-Yau 3-folds, and bundles on $ K3 $ fibrations},
  author={Thomas, Richard P},
  journal={Journal of Differential Geometry},
  volume={54},
  pages={367--438},
  year={2000},
  publisher={Lehigh University}
}

@article{PT,
   title={Curve counting via stable pairs in the derived category},
   journal={Inventiones mathematicae},
   publisher={Springer Science and Business Media LLC},
   author={Pandharipande, R. and Thomas, R. P.},
   year={2009},
   pages={407–447},
   eprint={arXiv:0707.2348},
   archiveprefix={arXiv},
}

@misc{RK,
author={Kool, Martijn and Rennemo, Jorgen},
    title={In preparation},
}

@article{BBJ,
  title={A Darboux theorem for derived schemes with shifted symplectic structure},
  author={Brav, Christopher and Bussi, Vittoria and Joyce, Dominic},
  journal={Journal of the American Mathematical Society},
  volume={32},
  number={2},
  pages={399--443},
  year={2019},
        eprint={1305.6302},
      archivePrefix={arXiv},
}

@article{SS,
  title={Equivalences of monoidal model categories},
  author={Schwede, Stefan and Shipley, Brooke},
  journal={Algebraic \& Geometric Topology},
  volume={3},
  number={1},
  pages={287--334},
  year={2003},
  publisher={Mathematical Sciences Publishers},
  preprint={arXiv:math/0209342},
  archiveprefix={arXiv},
}

@misc{Lim,
AUTHOR = {Lim, Woonam},
     TITLE = {Virtual {$\chi_{-y}$}-genera of quot schemes on surfaces},
   JOURNAL = {J. Lond. Math. Soc. (2)},
  FJOURNAL = {Journal of the London Mathematical Society. Second Series},
    VOLUME = {104},
      YEAR = {2021},
    NUMBER = {3},
     PAGES = {1300--1341},
      eprint={2003.04429},
      archivePrefix={arXiv},
}

@article{gebert,
   title={Introduction to vertex algebras, Borcherds algebras, and the monster Lie algebra},
   volume={08},
   journal={International Journal of Modern Physics A},
   publisher={World Scientific Pub Co Pte Lt},
   author={Gebert, Reinhold W.},
   year={1993},
   pages={5441–5503}
}

@misc{bojko2,
      title={Wall-crossing for zero-dimensional sheaves and Hilbert schemes of points on Calabi--Yau 4-folds}, 
      author={Arkadij Bojko},
      year={2021},
      eprint={2102.01056},
      archivePrefix={arXiv},
}

@misc{bojko3,
author="Bojko, A.",
title={Wall-crossing for Calabi--Yau fourfolds},
year={In preparation.},
}

@misc{stark,
      title={Cosection localisation and the Quot scheme $\mathrm{Quot}^{l}(\mathscr{E})$}, 
      author={Samuel Stark},
      year={2021},
      eprint={2107.08025},
      archivePrefix={arXiv},
}

@article {FanGo,
    AUTHOR = {Fantechi, Barbara and G\"{o}ttsche, Lothar},
     TITLE = {Riemann-{R}och theorems and elliptic genus for virtually
              smooth schemes},
   JOURNAL = {Geom. Topol.},
  FJOURNAL = {Geometry \& Topology},
    VOLUME = {14},
      YEAR = {2010},
    NUMBER = {1},
     PAGES = {83--115},
}

@article{JOP,
  title={Rationality of descendent series for Hilbert and Quot schemes of surfaces},
  author={Johnson, Drew and Oprea, Dragos and Pandharipande, Rahul},
  journal={Selecta Mathematica},
  volume={27},
  number={2},
  pages={1--52},
  year={2021},
  publisher={Springer}
}

@phdthesis{PhD,
author={Arkadij Bojko},
title={Wall-crossing and orientations for invariants counting  coherent sheaves on CY fourfolds [Oxford PhD thesis]},
year={2021},
}

@article{knuthconpol,
  title={Convolution polynomials},
  author={Knuth, Donald E},
  journal={arXiv preprint math/9207221},
  year={1992}
}

@article {HuyTho,
    AUTHOR = {Huybrechts, Daniel and Thomas, Richard P.},
     TITLE = {Erratum to: {D}eformation-obstruction theory for complexes via
              {A}tiyah and {K}odaira-{S}pencer classes },
   JOURNAL = {Math. Ann.},
  FJOURNAL = {Mathematische Annalen},
    VOLUME = {358},
      YEAR = {2014},
    NUMBER = {1-2},
     PAGES = {561--563},
}

@book{Ser,
    AUTHOR = {Sernesi, Edoardo},
     TITLE = {Deformations of algebraic schemes},
    SERIES = {Grundlehren der Mathematischen Wissenschaften},
    VOLUME = {334},
 PUBLISHER = {Springer-Verlag, Berlin},
      YEAR = {2006},
}

@article {STV,
    AUTHOR = {Sch\"{u}rg, Timo and To\"{e}n, Bertrand and Vezzosi, Gabriele},
     TITLE = {Derived algebraic geometry, determinants of perfect complexes,
              and applications to obstruction theories for maps and
              complexes},
   JOURNAL = {J. Reine Angew. Math.},
  FJOURNAL = {Journal f\"{u}r die Reine und Angewandte Mathematik. [Crelle's
              Journal]},
    VOLUME = {702},
      YEAR = {2015},
     PAGES = {1--40}
}

@article{ricolfi1,
  title={Virtual classes and virtual motives of Quot schemes on threefolds},
  author={Ricolfi, Andrea T},
  journal={Advances in Mathematics},
  volume={369},
  pages={107182},
  year={2020},
  publisher={Elsevier}
}

@misc{Woonamthesis,
author={Woonam Lim},
title={Virtual invariants of Quot schemes of surfaces [University of California San Diego PhD thesis]},
year={2014},
url="https://escholarship.org/uc/item/4hs6q40t"
}

@article{Shen,
  title={Cobordism invariants of the moduli space of stable pairs},
  author={Shen, Junliang},
  journal={Journal of the London Mathematical Society},
  volume={94},
  number={2},
  pages={427--446},
  year={2016},
  publisher={Oxford University Press}
}

@misc{Huy,
      title={Virtual pullbacks in Donaldson-Thomas theory of Calabi-Yau 4-folds}, 
      author={Hyeonjun Park},
      year={2021},
      eprint={2110.03631},
      archivePrefix={arXiv},
      primaryClass={math.AG}
}

@misc{stark2,
      title={On the Quot scheme $\mathrm{Quot}^{l}(\mathscr{E})$}, 
      author={Samuel Stark},
      year={2021},
      eprint={2107.03991},
      archivePrefix={arXiv},
}

@misc{bojko4,
    title={A short proof of a new combinatorial identity},
    author={Arkadij Bojko},
    year={2021},
    eprint={2111.09868},
    archivePrefix={arXiv},
}

@article {GK1,
    AUTHOR = {G\"{o}ttsche, Lothar and Kool, Martijn},
     TITLE = {Virtual refinements of the {V}afa-{W}itten formula},
   JOURNAL = {Comm. Math. Phys.},
  FJOURNAL = {Communications in Mathematical Physics},
    VOLUME = {376},
      YEAR = {2020},
    NUMBER = {1},
     PAGES = {1--49},
}

@article {GK2,
    AUTHOR = {G\"{o}ttsche, L. and Kool, M. and Williams, R. A.},
     TITLE = {Verlinde formulae on complex surfaces: {$K$}-theoretic
              invariants},
   JOURNAL = {Forum Math. Sigma},
  FJOURNAL = {Forum of Mathematics. Sigma},
    VOLUME = {9},
      YEAR = {2021},
     PAGES = {Paper No. e5, 31},
}

@article{BBBJ,
  
	year = 2015,
	month = {may},
  
	publisher = {Mathematical Sciences Publishers},
  
	volume = {19},
  
	number = {3},
  
	pages = {1287--1359},
  
	author = {Oren Ben-Bassat and Christopher Brav and Vittoria Bussi and Dominic Joyce},
  
	title = {A `Darboux theorem' for shifted symplectic structures on derived Artin stacks, with applications},
  
	journal = {Geometry and Topology}
}

@misc{Arb22,
       eprint={2201.07392},
      archivePrefix={arXiv},

  
  author = {Arbesfeld, Noah},
  
  title = {K-theoretic descendent series for Hilbert schemes of points on surfaces},
  
  publisher = {arXiv},
  
  year = {2022},
  
  copyright = {Creative Commons Attribution 4.0 International}
}

@misc{BML22,
       eprint={2210.05266},
      archivePrefix={arXiv},
  author = {Bojko, Arkadij and Lim, Woonam and Moreira, Miguel},
  title = {Virasoro constraints on moduli of sheaves and vertex algebras},
  publisher = {arXiv},
  year = {2022},
}

@inproceedings{EGA,
  title={Elements de geometrie algebrique III: Etude cohomologique des faisceaux coherents},
  author={Alexander Grothendieck and Jean Alexandre Dieudonn{\'e}},
  year={1961}
}

@article{Po14,
  
	year = 2014,
	month = {aug},
  
	publisher = {Springer Science and Business Media {LLC}
},
  
	volume = {146},
  
	number = {1-2},
  
	pages = {107--123},
  
	author = {Flavia Poma},
  
	title = {Virtual classes of Artin stacks},
  
	journal = {Manuscripta Mathematica}
}

@article{GP94,
author = {Garc\'{i}a–Prada, Oscar},
title = { Dimensional Reduction of Stable Bundles, Vorcises and Stable Pairs},
journal = {International Journal of Mathematics},
volume = {05},
number = {01},
pages = {1-52},
year = {1994}}

@article{Mo93,
title = {Comparison of the Donaldson polynomial invariants with their algebro-geometric analogues},
journal = {Topology},
volume = {32},
number = {3},
pages = {449-488},
year = {1993},
author = {John W. Morgan}
}

@article {Do83,
    AUTHOR = {Donaldson, S. K.},
     TITLE = {An application of gauge theory to four-dimensional topology},
   JOURNAL = {J. Differential Geom.},
  FJOURNAL = {Journal of Differential Geometry},
    VOLUME = {18},
      YEAR = {1983},
    NUMBER = {2},
     PAGES = {279--315},
}

@incollection {DT96,
    AUTHOR = {Donaldson, S. K. and Thomas, R. P.},
     TITLE = {Gauge theory in higher dimensions},
 BOOKTITLE = {The geometric universe ({O}xford, 1996)},
     PAGES = {31--47},
 PUBLISHER = {Oxford Univ. Press, Oxford},
      YEAR = {1998},
}

@article{Inaba,
author = {Michi-aki Inaba},
title = {{Toward a definition of moduli of complexes of coherent sheaves on a projective scheme}},
volume = {42},
journal = {Journal of Mathematics of Kyoto University},
number = {2},
publisher = {Duke University Press},
pages = {317 -- 329},
year = {2002},
}

@article {KCF,
    AUTHOR = {Ciocan-Fontanine, Ionu\c{t} and Kapranov, Mikhail},
     TITLE = {Virtual fundamental classes via dg-manifolds},
   JOURNAL = {Geom. Topol.},
  FJOURNAL = {Geometry \& Topology},
    VOLUME = {13},
      YEAR = {2009},
      NUMBER = {3},
     PAGES = {1779--1804},
}
\end{document}